\documentclass{amsart}
\usepackage{amssymb}
\usepackage{graphicx}
\usepackage{amsmath}
\usepackage{hyperref}
\newcommand{\reqnomode}{\tagsleft@true}
\numberwithin{equation}{section}

\newtheorem{theorem}{Theorem}[section]

\newtheorem{definition}{Definition}
\newtheorem{proposition}{Proposition}
\begin{document}
\title [On  Non-Newtonian  fluids] {On  Non-Newtonian  fluids  and  phase field  approach: \\
 Existence  and  Regularity}
\date{}            


\author{Rodolfo Salvi}
\address {\hfill\break
Rodolfo Salvi\hfill\break
Depertment of Mathematics
 Politecnico  di  Milano \hfill\break
 20133 Milan -  ITALY }
\email {rodolfo.salvi@polimi.it}

\begin{abstract}
The object  of  this  paper  is  twofold.  Firstly, we study   a  class of
generalized Newtonian fluid related to " power law ". For  the  corresponding
non-Newtonian Navier-Stokes problems,
the  existence of a  weak and periodic solutions is  proved  in the  large
 for  a bounded domain  in  $\mathbb{R}^3$. Further, variational inequalities and  local-in-time well-posedness of the initial-boundary value problem are investigated .
 Secondly, we  deduce   a  generalization
 of  the Graffi-Kazhikhov-Smagulov model  based on  an advective-diffusion process  in  the  context  of
 multiphase  theory. Local in time well-posedness of the initial-boundary value problem is investigated.

\end{abstract}

\maketitle

\section{Introduction}

In an Eulerian description of the flow field of  a  fluid the balance  of a transferable property $\mathcal{P}$,
defined per  unit of mass in a unit control volume, reads
\begin{align}
 \partial_t(\rho \mathcal{P}) +  \nabla\cdot (\rho\mathbf{v}\mathcal{P})=- \nabla \cdot \mathbf{I}_{\mathcal{P}}+\mathbf{F}_{\mathcal{P}}.
 \label{eq:1}\\ \nonumber
 \end{align}
Here,  and  in the following, Cartesian   tensor  notation will  be  used,  and  the  summation convention with  respect to  repeated indices.

\noindent
$\rho$  stands  for  the  density  of  the fluid  and $\mathbf{v}$  stands  for  the  velocity  vector  field.

\noindent
 The left-hand side shows the change with time of the property (in a unit volume),
 and the change  due to the divergence in a convective transport  by the flow  through the
 boundaries  of the control volume. The first  term  on the right-hand side describes
 the divergence of the transport $\mathcal{I}_{\mathcal{P}}$ through these boundaries by molecular  effects. The  last term stands   for  any
 internal or external  process, or source,  that contributes  to the change  of  $\mathcal{P}$  in  the  control  volume.
The expression of  the left-hand side is, in  general, independent of   $\mathcal{P}$ and  the  process concerned.

\noindent
 The first  term on the right-hand side depends  on  the nature of  the  property, while the last term depends  on   both  the nature of the
 property and the process considered.

\noindent
 If we consider the mass as the transferable  property, thereby  making  no distinction between possible components,
  then $\mathcal{P}$   becomes equal  to unit,   being    the mass per unit of mass, and $\mathbf{I}_{\mathcal{P}}=0$.
   If,  further,  assume no  sources  or sinks of mass present in the flow,  also $\mathbf{F}_{\mathcal{P}}  =0$.  We then obtain  the equation for the  conservation  of   mass

$$\partial_t\rho +  \nabla\cdot \rho\mathbf{v}=0.$$

Concerning the  equation  of  the  conservation of  the  momentum, we  may  start from the  above  general  balance  equation
 (\ref{eq:1}), where in  this  case    we  take  for $\mathcal{P}$   any   of the   components $v_i$ of the momentum   per unit  of  mass.
 The transfer of momentum     by   molecular   effect
results  in  a stress, so  that  for  the  balance  equation in  $x_i$-direction we  have
   $(\mathcal{I}_{v_i})_j=-\sigma_{ji}$.  Here  $\mathbf{\sigma}:=(\sigma_{ji})$ with   $1\leq (j,i)\leq  n$.

\noindent
 The  stress  is  defined positive  if it is   directed
in  the  positive  direction  of  $x_i$.

Let $F_i$  an   external  force  working on a  unit  volume  of  the  fluid  in  $x_i$-direction.  The  balance  equation for  the moment $v_i$  then reads
\begin{align}
 \partial_t(\rho v_i) +  \partial_{x_j} (\rho v_jv_i)=\partial_{x_j} \sigma_{ji}+F_i.
 \label{eq:2}\\ \nonumber
 \end{align}

 The stress tensor $\mathrm{\sigma}$ can be divided  into a part  that corresponds to the average value of the normal
 stresses  for all directions, that is, a spherically-symmetric part, which is invariant under rotation of  the
 coordinate system, and  an anti-symmetric part.

The spherically-symmetric part is equal to  $1/3\sigma_{ii}$. It contains the thermodynamic pressure  $ \pi$
and an additional term proportional to $\nabla \cdot \mathbf{v}$. So:
  $$\frac{1}{3}\sigma_{ii}=-\pi +k\nabla \cdot \mathbf{ v},$$
$k$ is  referred as  volume  viscosity. Consequently,  the expression of the stress tensor can be written

$$ \sigma_{ji}=  (-\pi + k \nabla \cdot  \mathbf{v})\delta_{ji} + T_{ji},$$
where  $T_{ji} $ is the anti-symmetric part, so  that $T_{ii}=0$  and is called the deviator of stress tensor.
System (\ref{eq:2})   contains the deviator $\mathbf{T}$
of the stress tensor which is not expressed explicitly via the unknowns of the system.

In conclusion,
from the balance equations of mass, momentum  the motion
of  fluid is described by the  system of equations  in the Cauchy form
\begin{align}
 &\rho\partial_t\mathbf{v} +    \rho \mathbf{v}\cdot \nabla\mathbf{v}
 - \nabla \cdot(\mathbf{T} -  (\pi -k\nabla\cdot \mathbf{ v} )\mathbf{I} )=\rho\mathbf{f},\label{eq:2a}\\ \nonumber
 &\partial_t\rho +  \nabla\cdot \rho\mathbf{ v}=0, \\ \nonumber
\end{align}
where $\mathbf{v}$ is the velocity vector of a particle in a point $ x$ at time $t$  and
$v_1, v_2,\ldots,v_n$ are the components of $\mathbf{ v}$, $\pi $ is the fluid pressure ,  $\mathbf{ f}$ is the density of external force and  $\mathbf{I}$  is  unit  matrix and  $\rho$ is the  mass  density.  $\nabla \cdot \mathbf {T}$ stands
for the vector
 $$
 (\sum_{j=1}^n \partial_{x_j}T_{j1}, \sum_{j=1}^n\partial_{x_j}T_{j2},\ldots,\sum_{j=1}^n\partial_{x_j}T_{jn}),
 $$
 whose coordinates are the divergence of rows of the matrix $\mathbf{ T}=( T_{ji})$.

 For $n=3$, the above four scalar equations contain twelve  unknown quantities. To achieve a unique solution
 up to a constant in the pressure, further relations between the velocity field $\mathbf{v}$ and the stress
 tensor $\mathbf{T}$ are necessary.

System  (\ref{eq:2a}) describes flows of  all kinds of  fluids. But  it contains  the  deviator
  $\mathbf{T}$ of the stress tensor which is not
 expressed explicitly  via the unknowns of the system.
Since the  stresses   and  deformations in  a  fluid
 (apart  external  forces) depend on  each  other
it is reasonable related the stress tensor to the spatial deformations, i.e. the symmetric part $\mathbf{D}:=(d_{ij})$ of the spatial variation    and is called deformation tensor. So,  as a rule, to express  the deviator of the stress tensor via the
 unknowns of the system, one uses relations between the deviator of stress tensor, the rate  of deformation tensor or  deformation  rate
  $$\mathbf{D}=(d_{ji})_{j,i=1}^n, \quad d_{ji}= \frac{1}{2}[\frac{\partial v_j}{\partial {x_i}} + \frac{\partial v_i}{\partial {x_j}}]=
 \frac{1}{2}(\nabla \mathbf{u} + (\nabla \mathbf{ u})^T)_{j,i}.
 $$

($(\cdot,\cdot)^T$  is  the  transpose matrix).

By establishing  the  connection between the deviator  and  the    deformation rate,  we  determine  the  type of  fluid.

  These relations depending on the considered substance, are called
{ \it Constitutive or Rheological equations}. These relations are hypothesis to be checked out for  concrete
fluids by experimental data.

A  relation,  largely  used  in the  last  150  years, is a linear  relation

$$\mathbf{T}=2\mu \mathbf{D} -\frac{2}{3}\mu \nabla\cdot \mathbf{v}\mathbf{I}, $$

\noindent
and  the  constant $\mu$  is referred as  the  dynamic viscosity.

  Such  fluids  are   classically  known as  the  Newtonian  fluids.
Newtonian fluid has been the main object of mathematical research in hydrodynamics. It has the determining relation
$$
 \mathbf{\sigma} = 2\mu\mathbf{D}+ ((k-\frac{2}{3}\mu) \nabla\cdot \mathbf{v}+\pi)\mathbf{I}.
 $$
With this expression  of $\mathbf{T}$ the system (\ref{eq:2a}) assume the following form
\begin{align}
 &\partial_t \rho\mathbf{v} + \nabla\cdot (\rho \mathbf{v}\otimes  \mathbf{v})
 - \mu\Delta \mathbf{v} + \nabla \pi - (k+\frac{1}{3}\mu)\nabla\nabla \cdot\mathbf{v}  =
  \rho  \mathbf{f},\label{eq:3a}\\ \nonumber
  &\partial_t \rho +\nabla\cdot \rho \mathbf{v}=0.\\ \nonumber
\end{align}
For  Newtonian fluids $k\simeq 0$.

This system of equations (Navier-Stokes system)  describes flows, under moderate velocities,
of most viscous  fluids.

\noindent
The Navier-Stokes equations are generally accepted as a right governing equations for compressible
or incompressible motion of viscous fluids.  The classical Navier-Stokes equations are essentially the simplest
equations describing the motion of a fluid.

Physically, it is assumed that
the constituent particles of the fluid are too small for their
dynamics can interact substantially with macroscopic motion or it  makes  no difference what
the  fluids  consisted of. But, the Newtonian model
is inadequate for fluids having a complex chemical structure.  Many substances  of  industrial significance,
especially  of  multi-phase  nature (forms, emulsions,  dispersions etc) do  not  conform to  the  Newtonian  postulate of
the linear  relationship  between  $\mathbf{T}$  and  $\mathbf{D}$. Accordingly, these fluids
are  variously  known  as non-Newtonian, nonlinear, complex or  rheologically  complex  fluids.

In general,  there are essentially  two  possibilities  to get constitutive equations; either
based on microstructure  of the substances  or based on the phenomenological  rheology.
The description of mechanical properties  of different materials in different deforming regimes is
the subject of rheology. The motion of a body as a whole is not considered, (it is the subject of theoretical
mechanics), rather it is the relative motion of particles of a body that is under consideration.
Rheology determines the dependence between forces acting on a material body and its deformations.
 So far, one of the main problem of rheology is the determination of links between stress, deformation, velocity of
 deformation, and their derivatives
with respect to time.
The  simplest possible  deviation from  the Newtonian  fluid  behavior  occurs  when  the  apparent  viscosity  is  not  constant.
Such  fluids  are  called non-Newtonian.

System for  which  the  value  of $\mathbf{T}$ at  a point within the fluid is determined only by the current value of $ \mathbf{D}$ or
$\dot {\gamma}$ (shear stress) at  that  point. These  substance  are  variously known as  purely  viscous, inelastic, time-independent or
generalized Newtonian  fluids.  Their  shear behavior can  be  described  by  the  relation
$$  \mathbf{T}\sim f(\dot \gamma).$$
Depending  upon  the  form of the  above function  $f$ three  possibilities  exist.

\subsection{Shear-Thinning  Fluids}

 This  is perhaps  the most widely encountered type   of time-independent  non-

 \noindent
 Newtonian  fluid  behavior  in  engineering  practice.
 It  is characterized  by  a  viscosity   which  gradually decreases  with increasing shear rate. This behavior
 is  due  to the  progressive orientations  of the molecules in the motion direction.

Shear  thinning  behavior  fluids  can  be  characterized  by  viscosity  of  the  form
\begin{align}
 &\mu= m(\dot \gamma)^{p-1}\label{eq:p} \\ \nonumber
\end{align}
 called " Power  law " or " Ostwald-de Waele  Equation ".

 Clearly, $0<p<1 $ will  yield  $d\mu/d\dot \gamma <0$, i.e. shear-thinning  behavior  fluids are  characterized by  a  value
  of  $p$  (power-law index)  smaller  than  unity, for  commonly used  $\mu$.

  \subsection{Shear-Thickening or  Dilatant   Behavior}

  For  some  fluids,  the  opposite  behavior  take  place;  the  viscosity  increases   with  the  increasing share  rate, and  hence  the  name  " Shear-Thickening ".

  Of  the  time-independent  fluids,  this  sub-class has  generate  very little  interest.  The  currently available limited  information suggests
that  it  is  possible  to  approximate the  behavior   for  these systems also by   the power  law model  with the  power-law  index  $p$
   taking  on  values  greater  than  unity.

  \subsection{Visco-plastic  fluid  behavior}

    This  type  of  non-Newtonian fluid  behavior  is  characterized by  the  existence  of  a  threshold stress  which  must  be  exceeded   for  the  fluid to  deform or   flow. Essentially  such  substance   will behave   like  an  elastic solid   when  the externally  applied  stress  is less  than
 the  yield  stress.  Example  is  the  Binghan  fluid.

 \subsection{Time-Dependent  Behavior}
In case of incompressible pure viscous
fluids,  for  constitutive equations on the
foundation of differential models, the stress can  be expressed as power series involving
increasing powers of the deformation tensor. Continuing, this conception leads to the
usage of convected derivatives based on the application of the invariance of material
properties with respect to the frame reference.
  Many  substances, notably  in  food,  pharmaceutical product  manufacturing  sectors  display flow  characteristics
   which  cannot  be  described by  a  simple mathematical expression as $f$. This is  so  because    their  viscosities   are
   not  only  functions  of  the  applied  shear  stress  but  also of  the  duration  for  which  the  fluid   has  been  subjected
     to  shearing   as   well  as   his  previous kinematic  history.  These  fluids present  two typical   features: stress  relaxation  and  creep.  The first   phenomenon is  the  progressive rather  than  instantaneous  stress  decay  when  the  fluid  deformation  suddenly  vanishes. The  second  effect, dual  to the  former,  consists  in  a  non linear  increasing  deformation,  though  the  fluid undergoes   a  constant  stress.
      One  of  the  simplest  ways  of  modifying  the  constitutive    equation of  a Newtonian fluid in  order  to  account  for the  " memory  properties "  of  a  given  fluid,  is  to  add a  term  containing  a  time  derivative or  some  sort of time  derivative of the  stress.
      The  addition of  the term containing  the  time  derivative of  the  stress   makes  the  equations capable of  representing  the  phenomenon
      of  stress   relaxation.  This  type of  model  is  called  differential  model.    Rate type  of constitutive  equation  containing  time  derivative  of  first  order  only, or  of the  general  form

$$\frac{d {\mathbf{\tau}}}{dt}= f(\mathbf{\tau}, kinematic\;  tensors).$$

    In  general, $\tau$  is  a  non-deviatoric  extra  stress.
 $\mathbf{\tau}$  and  the kinematic  tensors  are  calculated  in  $t$. It  follows   that    the stress   at  any time  $t$ can, in  principle, calculate  from  a
 knowledge  of  the  stress  at  previous time   $\bar t$.
For  this  fluids  the  stress is  expressed   in integral form.
$$\mathbf{\sigma} =\int_0^{\infty} f(s)\mathbf{G}(t-s)ds.$$
The  function $\mathbf{G}$  is called  the  stress   relaxation modulus.
The  physical    assumption   underlying  the  form  is  clear:  it  is  assumed that  all  the  deformations
   which  occurred  in the past  as  measured by the Cauchy tensor contribute linearly to  the present  value  of  the  stress.

   The  weighing function  $f(s)$ is a  material  function  which  completely determines  a  particular  material obeying such  a  linearity rule.

In  this  paper  we  consider  two  classes   of  non  Newtonian fluids. The  first a  generalized Newtonian  fluid.
The second is  a model deduced in  the  context of multi-phase flow.

The plain  for  the paper is as   follows.  Section 2  contains some preliminary results  and  notations.  Section 3
 is  devoted  to  the  formulation of  the  problem  and  contains   the proof of the  long  life existence of  a  so-called  weak  solution. In section 4 we  study  the  existence of  periodic solution. Section 5 treats the  existence problem for variational inequalities. Section 6  a  well-posedness  problem  is  discussed. Finally, in  section 7 we  deduce a   Graffi-Kazhikov-Smagulov  model  and discuss  the  existence problem.
\section{Preliminaries}
 In the sequel we will assume that
 $ \Omega $ denotes an open set in $ \mathbb{R}^n $ which is generally assumed to be bounded hence ${\bar \Omega} $
 is compact. $ \Gamma $ denotes the boundary of $ \Omega $.  Moreover, it is assumed that $ \Omega $ is a smooth domain
 of class $C^k$ with $k $ a positive integer. Furthermore, we assume that the unit normal
vector field $\mathbf{ n}(x) $ with $x\in \Gamma $ is inward pointing on $\Gamma$. If it is necessary we consider also an
extension of $ \mathbf{n} $ in a neighborhood of $ {\bar \Omega }$.
 With symbols $c$, $c_0$, $c_1$, etc., we denote generic
positive constants. When we wish to emphasize the dependence of $c $ on some
parameter $r$, we shall write $c(r)$.

 We do not distinguish in our notations
whether the functions are  $ \mathbb{R}^m$-valued   (vector) or $\mathbb{R}^m\times\mathbb{R}^m$-valued  (tensor  or
matrix). So  $\mathbf{v}=(v_i)$ denotes a vector with components $v_i$;   $\mathbf{T}=  (T_{ji})$   denotes  a  tensor  with
components ${T_{ji}}$.  Furthermore, $|\mathbf{u}|^2=\sum_{i=1}^n|u_i|^2$ and   for a  matrix or  $2$-tensor
$\mathbf{A}=(a_{ji}),  \;1\leq j,i\leq n$  the  Frobenius norm  is
$$\|A\|^2=a:a =a_{ji}a_{ji}=\sum_{j,i=1}^n|a_{j,i}|^2.$$
The  form  of $\|A\|^p$  is  obvious.

We define $ C^{\infty}_0(\Omega)$ to be the linear space of infinitely many times differentiable functions (vectors, tensors) with
compact supports in $ \Omega $. Now let  $ (C^{\infty}_0(\Omega))'$
denote the dual space of  $ C^{\infty}_0(\Omega)$, the space of
distributions on $\Omega$.
In  our notations $C^{\infty}(\Omega;R^n) $  coincides with $C^{\infty}(\Omega)\times C^{\infty}(\Omega)\ldots \times C^{\infty} (\Omega)$.
In  what follows  we do  not  specify  this  extension.
We denote, in general, by $<\cdot,\cdot> $ the duality
pairing between  $ (C^{\infty}_0(\Omega))'$ and  $C^{\infty}_0(\Omega)$.

Let $\mathbf{\alpha} = (\alpha_1,...,\alpha_n )\in \mathbb{N}^n$ and set  $|\alpha | =
\sum_{i=1}^n \alpha_i $. We set $$\frac {\partial}{\partial x_{i}} =
\partial_{x_i}=\partial_i,\quad  D_x^{\alpha} = D^{\alpha}=
\partial ^{|\alpha|}_{x^{\alpha_1}_1,...,x_n^{\alpha_n}},
$$
 $ \nabla = ( \partial_{{x_1}},...,\partial_{ {x_n}}) $ the
gradient operator and $\nabla \cdot $ the divergence operator.

\noindent
 We denote
${\mathcal C}^{\infty}_0 $  the linear subspace of divergence free
functions of $C^{\infty}_0 $.

\noindent
For  any $s, q $, $ s \geq 0 ,q \geq 1$ , $ W^{s}_{q}(\Omega)$   denotes
the usual Slobodeckii-Sobolev space of order $s $ on $ L_q(\Omega)$.
Further, the norm (defined intrinsically involving first order
differences of the highest-order derivatives)
 on $ W^{s}_{q}(\Omega)$ is denoted by
$\|\phi\|_q^s $.  If   $s=1$   we  write $\|\phi\|_q$ and $|\phi|_q$  denotes  the norm in  $L_q(\Omega)$-spaces.
When $ q = 2 $ $W^s_2(\Omega)$ is
usually denoted by $ H^s(\Omega) $ and we drop the subscript $q=2$
when referring to its norm. $ H^s (\Omega) $ ($s\in N$) is a Hilbert
space for the scalar product
$$ ((u,v))_s =  \sum_{|\alpha|\leq s} \int_{\Omega}
D^{\alpha}u D^{\alpha}vdx .$$
 In particular, in $L_q(\Omega)$, we
write the $L_q$-duality pairing $ (u,v)_q = \int_{\Omega}uvdx$ with
$u\in L_q$, and $ v\in L_{q'}$ with  $q' = q/(q-1)$ and the norm $
|v|_q $.

\noindent
Further, we define $ W^{s}_{q,0}(\Omega) $ the closure of $
C^{\infty}_0(\Omega) $ for the norm $\|\cdot\|^s_q $.

\noindent
We denote $ W^{-s}_{q'}(\Omega) $ the dual space of $
W^{s}_{q,0}(\Omega) $ and $ \|\cdot\|^{-s}_{q'}$ denotes its norm
where $q'$ satisfies $1/q + 1/{q'} = 1 $.

\noindent
Moreover  $W^{2-2/q}_q(\Omega)\hookrightarrow BUC^1(\Omega)$  with  $q>n+2$
($BUC$=bounded  uniformly continuous).

The  embedding $$ W^1_p(\Omega)  \subset  L_q(\Omega)$$
 is   compact if $p<n$  and  $ q<\frac{np}{n-p}$ or  if $p\geq n$ and $\infty >q > 1$.

In   particular
 the  embedding
 $$ L_{q^{'}}(\Omega)\subset   W^{-1}_{p^{'}}(\Omega)$$
 is  compact.

Let us introduce the following spaces of divergence-free functions.
 We denote by
 $$ V^s_q = \{\mathbf{v}|\mathbf{v}\in W^s_{q,0}(\Omega), \nabla \cdot \mathbf{ v} =0 \}. $$
$ V_q^s $ is the closure of ${\mathcal C}^{\infty}_0(\Omega)$ (divergence  free) for the norm
$\|\cdot\|^s_q $, and it is a closed subspace of $ W^s_q(\Omega) $.

Moreover in  the  sequel
  we  use  the  triplet

$$V_q^s\subset  H\subset (V_q^s)^{'}.$$
We set $V^s_2:=V^s$, $ V^1 := V $ and $ V^0 := H $. Moreover, we introduce the projection operator $P_q
 \;(P_2\equiv P): L_q \rightarrow V^0_q$. It is well known that the operator $P_q$ is
 continuous on $L_q$ and the subspace $V^0_q$ is complemented. Thus, the  following
decomposition of $L_q$
$$L_q = Range {P_q} \oplus Ker {P_q}  $$
holds true.

It is interesting to observe  that $ KerP_q =\{\mathbf{\phi}\in L_q|\mathbf{\phi}
=\nabla p_1 +\nabla p_2\}$ where $p_1,p_2$ are generalized solutions
of the problems
$$ \Delta p_1 = 0, \; \partial_\mathbf{n} p_1 = \mathbf{f}\cdot\mathbf{ n} \; on \;\Gamma,
$$
and
$$
\Delta p_2 = \nabla \cdot\mathbf{ g}, \; p_2= 0 \; on \; \Gamma,
$$
respectively.

\noindent
Here $\mathbf{g},\mathbf{f} \in L_q$ and $\mathbf{f}\cdot \mathbf{ n}\in H^{-1/2}(\Gamma)$ satisfying $<f\cdot n,1>_{\Gamma}=0$.

We recall that analogous decomposition of $L_q$ holds working with
the subspace
 $\bar V= \{\mathbf{\phi}| \mathbf{\phi}\in L_q, \nabla\cdot \mathbf{\phi} =0\}$.

We further define the Stokes operator on $ L_q$
$$A_q=-P_q\Delta,
$$
with domain $D(A_q)=W^2_q(\Omega)\cap V^1_{q}$. If  $q=2$  we  set  $A_2:=A.$

For any Banach space $X$ and for any $T>0 $ we denote by
$L_r(0,T;X)$ the set of $X$-valued functions defined a.e. in $[0,T]$
and $L_r$ summable in sense of Bochner. Frequently we consider $ X =
W^s_q(\Omega)$.  In such cases, for any $\phi\in
L_r(0,T;W^s_q(\Omega))$,  $\phi$ stands for the function $\phi(t)$
or $\phi(\cdot,t)$.

Throughout the paper we denote $Q_t = (0,t)\times \Omega $ and the
parabolic

\noindent
Slobodeckii-Sobolev space
$W_q^{s,r}(Q_T)$ of order $ s$
in space variable and of order $r$ in time variable on $L_q$. We
will denote  $ \|\cdot\|^{s,r}_q$ the norm  in this space. In the
following we make use of  the inequality, for $q >3$ ,
$$sup_{(x,t)\in Q_T}|v| \leq \|v\|_{W^{2,1}_q(Q_T)}.$$
Moreover,  we  set
$$Z(T):= W^1_q(0,T; L_q(\Omega))\cap L_q(0,T;W^2_q(\Omega))\hookrightarrow C((0,T; W^{2-2/q}_q(\Omega)).$$
In  this  case  the  embedding constant  can blow up  as  $t\rightarrow  0^+$ if  the  functions are     different     from   zero  in  $t=0$ .

\noindent
$W^{2-2/q}_q(\Omega))$  can  be  considered  as  time-trace   of  $Z(T)$.

\noindent
In addition, let us consider the affine space
$$
\tilde H^k(\Omega) = \{\phi\in H^k(\Omega), \partial_n \phi=0
\;{\textrm on}\; \Gamma, \int_{\Omega}\phi dx =c\}.
$$
In this manner, the functions in $\tilde H^k$ are uniquely fixed and
 we can  not distinguish the norms $\|\phi\|_{H^2} $ and $|\Delta \phi|_2$  , $\|\phi\|_{H^3}$ and
 $|\nabla \Delta\phi|_2$  in $\tilde H^3$.
Throughout the paper we shall use  the following propositions.

\begin{proposition}[Gagliardo-Nirenberg inequality] Let
$\Omega \subset R^n$ bounded and sufficiently regular. The
multiplicative inequality,
\begin{align}
\sum_{|\alpha| = r}|D_x^{\alpha}\phi|_{q} \leq
c|\phi|_{q_1}^{1-\theta}(\sum_{|\alpha| =l}|D_x^{\alpha}\phi|_{q})^{\theta} ,
\label{eq:i}\\ \nonumber
\end{align}
for $ 1\leq q_1, q_2 \leq \infty, \quad 0\leq r \leq l,$
$$
\frac{n}{q} - r = (1-\theta) \frac{n}{q_1} + \theta( \frac{n}{q_2}-
l),
 \quad \frac{r}{l} \leq \theta \leq 1,
$$
\noindent holds with the following exceptions:

a) if $ r=0,\; l < \frac{n}{q_2}$, and $ q_1 = \infty$ and $\Omega$
unbounded, we assume in addition that or $ \phi \rightarrow 0$ as $
x \rightarrow \infty$ or $ \phi \in L_p $ for some $p > 0$;

b) if $1 < q_1 <\infty $  and $l-r -\frac{n}{q_2} $ is a
non-negative integer, then
 does not hold for $ \theta
=1$.
\end{proposition}

 Let $\Omega\subset\mathbb{R}^n$  be  a   open bounded   set and  $\mathrm{X}$ is  a  topological metric  space with  metric $d_X$. We consider  the  Fr\'{e}chet
 space $\mathrm{C}(\bar{\Omega}; \mathrm{X})$  of continuous functions  $f: \bar{\Omega} \rightarrow \mathrm{X} $  equipped with  the  metric
 $$d(f,g)=  \sup_{y\in \Omega} d_X(f(y),g(y)).$$ Further  we  denote $\mathrm{X^*_w}$  the  dual  of  $\mathrm{X}$
  equipped  with  the  weak topology.

\begin{theorem} [Abstract  Ascoli-Arzel$\grave{a}$  Theorem]  Let  $\Omega$ be  compact  and $\mathrm{X}$  is a   topological  metric  space equipped  with  a  metric  $d_X$.
Let  $\mathcal{F}=\{f_n\}$  be  a sequence  of  functions  in $(\bar{\Omega}; \mathrm{X})$.  We  assume  that
\begin{enumerate}
 \item $\mathcal{F}(y)=\{f_n(y)|f\in  \mathcal{F}\} $ has compact closure in  $\mathrm{X}$;
\item $\mathcal{F}$is  equi-continuous  under  $d_X$, i.e.
$$\sup_yd_{X}(f_n(y+h)-f(y)) \rightarrow 0,\; for \;   h\rightarrow 0,$$
uniformly  in $n$.
\end{enumerate}
 Then $\mathcal{F}$ is  precompact in
$C(\bar{\Omega}                                                                                                         ; \mathrm{X})$.

Moreover,  let   the  set  of  functions  $<f_n(y),\phi>$  ( $y\in \Omega$)  be  equi-continuous   for  any  fixed  $\phi$  belonging  to   a dense  subset  of  $\mathrm{X}$.

Then  $f_n\in C(\bar{\Omega}; \mathrm{X^*_w})$   and  there exists   an  $f \in
 C(\bar{\Omega}; \mathrm{X^*_w})$  such  that

 $$f_n\rightarrow  f,$$
in $C(\bar{\Omega}; \mathrm{X^*_w})$ as  $n\rightarrow  \infty$.
\end{theorem}

  The previous  theorem  has $L_p(\bar {\Omega})$-version  (see  \cite {rif35}, \cite {rif38}, for  example).

\begin{theorem} [Frech\'et-Kolmogorov Theorem] -  Let  $X \subset E \subset Y$
 be  Banach  spaces, the embeddings $X\subset  E \subset Y$    be  continuous.
Let  $\mathcal{F}=\{f_n\}$  be  a sequence  of  functions  in $L_p(0,T; X)$ with $1\leq p\leq \infty$.  We  assume  that
\begin{enumerate}
 \item $\mathcal{F}(t)=\{f_n(t)|f\in  {F}\} $ has compact closure in  ${E}$;
\item $\mathcal{F}$is  equi-continuous  in  $L_p(0,T; Y)$ i.e.
$$\|(f_n(t+h)-f(t)||_{L_p(0,T-h;Y)} \rightarrow 0,\; for \;   h\rightarrow 0,$$
uniformly  in $n$.
\end{enumerate}
 Then $\mathcal{F}$ is  precompact in
$L_p(0,T;Y)$.  If  $p=\infty$  then $\mathcal{F}$  belongs to  $C(0,T;E)$ and  is  relatively  compact  in  this  space.                                                                                                        \end{theorem}
Making  use  of  the  derivative in time the  theorem  assumes the  following  form.
\begin{theorem}    Let  $X \subset E \subset Y$
 be  Banach  spaces, the embeddings  $X\subset  E$, and   $E \subset Y$   be  continuous.
Let  $\mathcal{F}=\{f_n\}$  be  a sequence  of  functions  in $L_p(0,T; X)$ with $1\leq p\leq \infty$  and  $\partial_t f_n\in L_r(0,T;Y)$  with  $1\leq r\leq \infty$. We  assume  that
\begin{enumerate}
 \item $\mathcal{F}(t)=\{f_n(t)|f\in  {F}\} $ is  compact  in  $E$ ;
\item the  set  $\{\partial_tf_n(t)\}$ is  bounded in $L_r(0,T;Y)$
uniformly  in $n$.
\end{enumerate}
 Then $\mathcal{F}$ is  precompact in
$L_p(0,T;E)$ if  $p< \infty$.    If  $p=\infty$ the  property continues  to  hold  for $r>1$.
\end{theorem}

\subsection{About  Cauchy and Schwarz's Inequalities}

Let $\mathbf{ a}=(a_1,a_2,...,a_n)$,
$\mathbf{ b}=(b_1,b_2,...,b_n)$,
be two vectors with $n$ components.

Then
$$ det
\left (
\begin{array}{ccc}
\sum_{i=1}^n a_i^2 & \sum_{i=1}^n a_ib_i\\
\sum_{i=1}^na_ib_i& \sum_{i=1}^nb_i^2\\
\end{array}
\right)>0  \Leftrightarrow
$$
$$
\sum_{i=1}^na_i^2\sum_{i=1}^nb_i^2- (\sum_{i=1}^na_ib_i)^2=\frac{1}{2}\sum_{i,j=1}^n(a_ib_j-a_jb_i)^2>0.$$

We  have  written  the  classical Cauchy   inequality  in  matrix  form.
An alternative proof is as follows (see \cite{rif14}).   The quadratic form
$$ \sum_{i=1}^n(xa_i +yb_i)^2=x^2\sum_{i=1}^na_i^2 + 2xy\sum_{i=1}^na_ib_i + y^2\sum_{i=1}b_i^2$$
is positive for all $x,y$ and therefore has a negative  discriminant, unless $xa_i+yb_i=0$ for some $x$, $y$ not both  zero, and for all  $i$.
A generalization of the above inequality is obtained by considering the leading principal minors property of
 a  matrix $\mathbf{A}$ obtained  with $n$ vectors $\mathbf{a}_i$ (linearly  independent),  $i=1,2,\ldots n$,  with $n$ real components in  the  following  manner:
 $$ \sum_{i=1}^n(x_1a_1^i +x_2a_2^i+\cdot\cdot\cdot+ x_na_n^i)^2>0\Leftrightarrow$$
 $$
 det
\left (
\begin{array}{cccc}
 \mathbf{a}_1\cdot \mathbf{a}_1 &  \mathbf{a}_1\cdot \mathbf{a}_2 &\ldots\ldots &\mathbf{a}_1\cdot \mathbf{a}_n\\
\mathbf{a_2}\cdot \mathbf{a_1}& \mathbf{a}_2\cdot\mathbf{a}_2 & \ldots\ldots & \mathbf{a}_2\cdot \mathbf{a}_n\\
\ldots\ldots &\ldots\ldots&\ldots\ldots  &\ldots\ldots \\
\ldots\ldots &\ldots\ldots&\ldots\ldots  &\ldots\ldots \\
\ldots\ldots &\ldots\ldots&\ldots\ldots  &\ldots\ldots \\
\mathbf{a}_n\cdot \mathbf{a}_1& \mathbf{a}_n\cdot \mathbf{a}_2  &\ldots\ldots &\mathbf{a}_n\cdot \mathbf{a}_n
\end{array}
\right)
>0.
$$
It  is also possible to express the  determinant as the  sum of  squares of determinants.
A simple proof  to obtain  $ det\mathbf{A}\neq 0$
(we  need  this  in  what   follows)  is a  reductio ad  absurdum argument (see \cite{rif2}).
Assume  $ det\mathbf{A}=0$,  then  there  exists  $n$  numbers $\lambda_i$ , $i=1, 2,...,n$  (  not all equal  to   zero) such  that
$$ \sum_{i=1}^n\lambda_i\mathbf{a}_i\cdot \mathbf{a}_j =0.$$

Multiplying  the  above  equation   by  $\lambda_j$ and summing up  with  respect to $j$ we  obtain
$$ \sum_{i=1}^n\lambda_i\mathbf{a}_i \cdot  \sum_{i=1}^n\lambda_i\mathbf{a}_i=0.$$

This contradicts the  linear  independence of  $\{\mathbf{a}_i\}$.

Bearing  in mind the Leibnitz's "Calculus  Summatorium ",  we can  translate the  Cauchy inequalities  in integral  inequalities: Schwarz's inequality  and  Gram-Schwarz's inequality:

1) Schwarz's inequality: let $f, g$ be  two  functions defined and  continuos in a domain  $\Omega$.
Then  $$(\int_{\Omega}f(x)g(x)dx)^2< \int_{\Omega}f^2(x)dx\int_{\Omega}g^2(x)dx.$$

unless $ af(x)=bg(x)$  with  $a, b$  constants no  both zero.

2)  Gram-Schwarz's inequality:  let $f_i(x)$ be $n$  continuous  function in $\Omega$  then

$$
 det
\left (
\begin{array}{cccc}
 \int f_1(x)f_1(x)dx &  \int f_1(x)f_2(x)dx &\ldots\ldots &\int f_1(x)f_n(x)dx\\
\int f_2(x)f_1(x)dx&\int f_2(x)f_2(x)dx& \ldots\ldots & \int f_2(x)f_n(x)dx\\
\ldots\ldots &\ldots\ldots&\ldots\ldots  &\ldots\ldots \\
\ldots\ldots &\ldots\ldots&\ldots\ldots  &\ldots\ldots \\
\ldots\ldots &\ldots\ldots&\ldots\ldots  &\ldots\ldots \\
\int f_n(x)f_1(x)dx& \int f_n(x)f_2(x)dx  &\ldots\ldots &\int f_n(x)f_n(x)dx
\end{array}
\right)
>0
$$

\noindent
  unless  the  functions $f_i(x)$ i=1,2,...,n,  are linearly  dependent.

\section{Solvability of non-Newtonian Flow Problems}

In this section we concentrate on incompressible generalized Newtonian
fluids  related   to power law. The  section  is  devoted  to the  solvability  of  system (\ref{eq:3a})  in  bounded  domain  $\Omega$.
Most  of  the  results have their origin  in  the  earlier results  of  the author concerning  the  theory  of non-homogeneous viscous  fluids.

\subsection{Additional Notation}

 Set
 $$\mathbf{D}:=\mathbf{D}(\mathbf{u})=(d_{j,i})=\frac{1}{2} (\nabla \mathbf{u} + (\nabla \mathbf{u})^T),$$

 \noindent
 the symmetric part of the tensor $\nabla \mathbf{u}$ with
$d_{ji}=\frac{1}{2} (\frac{\partial u_j}{\partial x_i} +\frac{\partial u_i}{\partial x_j})$.

Let us remark, recalling the Frobenius norm, that

$$\|\mathbf{D}\|^2=\mathbf{D}:\mathbf{D}=|d_{j,i}|^2:= \sum_{j,i=1}^{n}d_{j,i}d_{j,i}.$$

\noindent
We  consider  as viscous
  stress tensor a modified   power-law, commonly    used in   practice  and    in mathematical  literature, i.e.
$$\mathbf{T}:= \mathbf{T}(\mathbf{u}) =2\mu(\|\mathbf{D}\|^2)\mathbf{D}(\mathbf{u}),$$ with $\mu$   a suitable  function. Then, the $i^{th}$
entry  of  $\nabla \cdot \mathbf{ T}$ can be written as follows
\begin{align}
&(\nabla\cdot (2\mu(\|\mathbf{D}\|^2) \mathbf{D}_i=\sum_{j=1}^n \partial_j(2\mu(\|\mathbf{ D}\|^2) d_{ji}=2\mu'(\|\mathbf {D}\|^2)\sum_{j=1}^nd_{ji}\partial_j\|\mathbf { D}\|^2+\label{eq:t} \\ \nonumber
&\mu(\|\mathbf{ D}\|^2)\sum_{j=1}^n \partial_j(\partial_ju_i+\partial _iu_j) =
 \mu(\|\mathbf{ D}\|^2)\Delta u_i +\\ \nonumber
&4\mu'(\|\mathbf { D}\|^2)\sum_{j,k,l=1}^nd_{ik}d_{jl}\partial_k\partial_lu_j=
\sum_{j,k,l=1}^nt_{ij}^{kl}(\mathbf{u})\partial_k\partial_lu_j.\\  \nonumber
\end{align}
Here $$t_{ij}^{k,l}=\mu(\|\mathbf{ D}\|^2)\delta_{kl}\delta_{ij} +4\mu'(\|\mathbf {D}\|^2)d_{ik}(u)d_{jl}(u),$$

\noindent
with
$$\partial_j\|\mathbf {D}\|^2= 2\sum_{k,l=1}^nd_{kl}\partial_jd_{kl}, \; \nabla\cdot  u  =0.$$

We define the quasi-linear differential operator
$$\mathbf{T}_p:= \mathbf{T}_p(\mathbf{u})=-\nabla\cdot
\mathbf{T}(\mathbf{u})=
\sum_{k,l=1}^nT^{k,l}(\mathbf {u})\partial_k\partial_l,$$

\noindent
with the matrix-valued  coefficients
$$ T^{k,l}(\mathbf{u})=(t_{ij}^{k,l}).$$
For the viscosity function $\mu(\cdot )$ several models may be
given. For  convenience, we tacitly   have  in mind the model
$$ 2\mu(\|\mathbf{D}\|^2)\mathbf{ D}:= \mu_0(1 +   \|\mathbf{D}\|^2)^{\frac{p-2}{2}}\mathbf{D},$$

\noindent
 ($\mu_0$  is  a  positive  constant) which is often quoted as standard model in the mathematical literature.

\noindent
If $p > 1$ the operator $\mathbf{T}_p$ is strongly elliptic ( see \cite{rif37}).

 Thus we consider the system
\begin{align}
&\rho\partial_t\mathbf{ u} +  \rho\mathbf{ u}\cdot \nabla  \mathbf{u}
 - \nabla \cdot(\mathbf{ T}(\mathbf{u})\mathbf{u} - \pi I)  =\rho \mathbf{ f},\label{eq:4a}\\ \notag
&\partial \rho +\nabla \cdot(\rho \mathbf{ u})=0,\\ \nonumber
& \nabla \cdot \mathbf{ u} =0,\;
 \mathbf{ u}(0)=\mathbf{ u}_0,\; \mathbf{u}_{\Gamma}=0, \;\rho(0)=\rho_0\geq 0,\;
 1<p. \\ \nonumber
\end{align}
The question is for which exponent, respectively for which $p$,
does a solution of the system      (\ref{eq:4a}) exist.

\subsection{ Structure conditions}

The  approach  to  find a  solution of  system  (\ref{eq:4a}) is via
 monotone operator theory   and compactness  arguments.  So we   set  classical  structure  condition
 assumptions.
We denote  $  \mathbf{M}_{sim}$ the set of symmetric matrices   of
order  $n$. In general, we assume that for a $p> 1$ and $q \in
[p-1,p)$ there exist $\alpha, \; \beta > 0$ such that for all
$\mathbf{\eta} \in \mathbf{M}_{sim}$ the  function $\mathbf{ A}   $  satisfies
\begin{enumerate}
\item (continuity)  $\mathbf{A}: Q_T\times\mathbb{R}^n\times \mathbf{M}_{sim}\rightarrow\mathbf{M}_{sim}$
is  a Charath\'eodory function,

\noindent
i.e. $(x,t)\mapsto
\mathbf{A}(x,t,\mathbf{u},\mathbf{D})$ is   measurable for every
$(\mathbf{u},\mathbf{D})$ and

\noindent $(\mathbf{u},\mathbf{D})\mapsto
\mathbf{A}(x,t,\mathbf{u},\mathbf{D})$ is  continuous for almost
$(x,t)\in Q_T$;

\item (Growth and  coercivity) There  exist $c_1>0,\; c_2 >0,\; g_1 \in L_{p'}(Q_T)$, \;

\noindent $g_2\in L_1(Q_T)$, $g_3\in
L_{(p/r)'}(Q_T)$, such  that
$$|\mathbf{ A}(x,t, \mathbf{u},\mathbf{D})|\leq g_1(x,t) +c_1(|\mathbf{u}|^{p-1}+ |\mathbf{D}|^{p-1}),$$
$$\mathbf{ A}(x,t, \mathbf{u},\mathbf{D}):\mathbf{D} \geq -g_2(x,t)-g_3(x,t)|\mathbf{u}|^r+c_2|\mathbf{D}|^p;$$
\item (Strict monotonicity)   For  $(x,t)\in Q_T, $   the  map $(\mathbf{u},\mathbf{D})\mapsto
\mathbf{ A}(x,t, \mathbf{u},\mathbf{D})$  is  $C^1$-function and
$$(\mathbf{A}(x,t, \mathbf{u},\mathbf{D}_1)- \mathbf{A}(x,t, \mathbf{u},\mathbf{D}_2):(\mathbf{D}_1-\mathbf{D}_2))
>0, \;\forall \mathbf{D}_1\neq\mathbf{D}_2.$$
\end{enumerate}
Now we state  the definition of solution which we are seeking.

First, $\mathbf{T}$  satisfies the   structure conditions. The  form
$$\mathbf{T}_p := -\nabla \cdot (2\mu(\|\mathbf{D}\|^2)\mathbf{D}(\mathbf{u}), $$

defines a continuous   operator $\mathbf{T}_p$
 acting from $ W_{p,0}^{1}(\Omega))$ into $( W_{p,0}^{1}(\Omega))'$ and
$$ < \mathbf{T}_p\mathbf{u},\mathbf{v}> =\int_{\Omega}\mu(\|\mathbf{D}\|^{p})\mathbf{D}(\mathbf{u})\mathbf{D}(\mathbf{v})dx,$$

\noindent
for $\mathbf{u}, \mathbf{v}\in W_{p,0}^{1}(\Omega)$.

 \noindent
Remark: The operator $\mathbf{T}_p$ defines a one-to-one correspondence  between $W_{p,0}^1(\Omega))$

\noindent
and $ (W_{p,0}^1(\Omega))'$ with inverse $(\mathbf{T}_p)^{-1}$, monotone, bounded and continuous.

By the Korn-inequality,
$$|||\mathbf{u}|||^p=\int_{\Omega} |\mathbf{D}(\mathbf{u})|^{p}dx,$$
is a new equivalent norm on $W^{1}_{p,0}(\Omega)$
and  the operator $\mathbf{T}_p$ is well defined in this space and conserves the continuity from $ W_{p,0}^{1}(\Omega)$
into $ (W_{p,0}^1(\Omega))'$.

\noindent
Further, an  accurate   choice of  basis is   relevant.

$\mathrm{V}^s$ is  the  adherence of $\mathcal{V}$ in $\mathrm{H}^s(\Omega).$

We  choose $s$  such  that if   $\mathbf{v}\in \mathrm{H}^s(\Omega)$ then $\partial_i \mathbf{v}\in  L^{\infty}(\Omega)$,
so  holds
$$ \mathrm{V}^s\subset \mathrm{V}_p\subset \mathrm{H} \subset (\mathrm{V}_p)'\subset \mathrm{V}^s)'.$$
In  the  case that $p=2$   we  omit the  subscript.

In  general,  we  denote $\{\mathbf{w}_m\}$  a  total  sequence, that is,  a  sequence  such  that  span $\{\mathbf{w}_m:m\in \mathbb{N}\}$
is  dense  in $\mathrm{V}_p$. This   is  guaranteed by  the   assumption that  $\mathrm{V}_p$ is  separable.
According to the  problem treated,  we, tacitly, choose or  an  arbitrary total  sequence $\{\mathbf{w}_m\}$ or
 the  set of eigenvectors   to the  problem
$$( \mathbf{w}_i,\mathbf{\phi})_{V^s}= \lambda_i(\mathbf{w}_i,\mathbf{\phi}),$$
where $(\cdot,\cdot)$  is  the  scalar product  in  $\mathrm{H}$.
\begin{definition} - Let $p>1$, $\mathbf{ u}_0\in H$, $\mathbf{ f}\in
L_{p}(0,T;(W^1_{p,0}(\Omega))')$ with $p'=\frac{p}{p-1}$ , $q$ a Sobolev conjugate and  $M$  is  a positive  constant.
Then a couple  of functions $(\mathbf{ u},\rho)$ is called a weak
solution to the problem (\ref{eq:4a})
 if:
\begin{align*}
& i)\; \sqrt{\rho} \mathbf{ u}\in L_{\infty}(0,T;H)\cap L_p(0,T;W^1_{p.0}(\Omega)),\partial_t(\rho\mathbf{u})\in (L_{p}(0,T;V_{p}(\Omega)))',\\
 &\rho\in
L_{\infty}(Q_T),\;\partial_t \rho\in (L_{p'}(0,T;W^{1}_{q'}(\Omega)))' \cap L_{5p/3}(0,T;W^{-1}_{5p/3}(\Omega)),
0  \leq \rho\leq M, \\
&ii)\;  \text{the  following
integral identities hold for  all smooth }\;
\mathbf{\psi} \\
& \text {and   divergence free} \;\mathbf{\phi}, \text{ with }\;\mathrm{\psi}(T)=\mathbf{\phi}(T)=0,\\
& \int_0^T(\rho, \partial_t \psi
+ \mathbf{ u}\cdot \nabla \psi)dt =-
(\rho_0 \mathbf{u}_0, \psi(0)),\\
 &\int_0^T((\rho \mathbf{ u},\partial_t \mathbf{\phi}) + (\rho\mathbf{u},\mathbf{u}\cdot \nabla \mathbf{\phi}) -
(\mathbf{T(\mathbf{u})u},\mathbf{D(\mathbf{\phi})}) +(\rho\mathbf{ f},\mathbf{\mathbf{\phi}}))dt = \\
&-(\rho_0\mathbf{ u}_0,\mathbf{\phi}(0)).
\end{align*}
\end{definition}

The problem (\ref{eq:4a})  with  $\rho $  constant is studied, relatively. The existence of weak solutions for $ p \geq \frac{3n+2}{n+2}$
 had first appeared in \cite{rif17}, \cite{rif18}. Later, Ne$\check{c}$as and his  collaborators investigated the  existence  and
 regularity of  the  problem (\ref{eq:4a}).  The  best  results  are  obtained  in the case  of  periodic spatial functions.

 We do not want to comment on all of the literature here, but concerning  the  existence,
the  regularity  problem and the  attempts to  find  optimal exponent $p$,
we refer to the survey article  \cite{rif23}, and to the paper \cite{rif3}.

The case  $\rho \geq   c  >0$ is   considered  in \cite{rif12} .

  We prove next the existence of weak solution of  the  problem (\ref{eq:4a}).

\begin{theorem} - Let $\mathbf{ u}_0 \in H$, $\rho_0\in L_{\infty}(\Omega)$, $0\leq \rho_0 \leq M$
and

\noindent
$\mathbf{ f}\in L_p(0,T;(W^1_{p,0}(\Omega))')$, $p\geq 1+\frac {2n}{n+1}$.
Then there  exists  a weak solution to the
 problem (\ref{eq:4a}).
 \end{theorem}

 \begin{proof}
 The plan of the proof  is as follows:
 \begin{enumerate}
 \item construction of a semi- Faedo - Galerkin approximation;
 \item  a priori estimates;
 \item compactness results;
 \item passing to the limit in the semi- Faedo- Galerkin approximation.
\end{enumerate}

 \subsection{Semi- Faedo- Galerkin approximation}

\medskip
Throughout the paper $P$ denotes the orthogonal projection from $L_2(\Omega)$ onto $\mathrm{H}$. We shall denote by   $\{\mathbf{w}_k(x)\}$ $k\in \mathbb{ N}$ a total  sequence in  $V_p$, sufficiently  regular.

  Let  $W_{(m)}:=span\{\mathbf{w}_k(x):  1\leq  k\leq m\}$.  We  denote $P_m$  the  projection  from $H$  onto  $W_{(m)}$.

\medskip
\noindent
The basic idea of the existence proof is to approximate  $\mathbf{u}(x,t)$ by  functions  $\mathbf{u^m}(x,t)$ in
 finite-dimensional subspace $W_{(m)}$ of $\mathrm{H}$ of dimension $m$ and approximate $\rho(x,t)$ by functions $\rho^m(x,t)$
  in an infinite-dimensional  space by solving the transport equation replacing  $\mathbf{u}$ for $\mathbf{u}^m$. This approximating procedure
  is called semi-Faedo-Galerkin  approximation. This gives a system of ODE's for $\mathbf{u}^m$.

  For our analysis, we take
  $\mathrm{W}_{(m)} \subset \mathrm{H}$ the linear space  spanned by the vectors $(\mathbf{ w}_1, \mathbf{ w}_2,...,\mathbf{ w}_m)$
 where $\mathbf {w}_1,  \mathbf{ w}_2,...,\mathbf{w}_m $ are the the first $m$-solutions of the spectral problem

$$(\mathbf {w}_j,\mathbf {v})_{\mathrm{V^s}}=\lambda_i(\mathbf {w}_j,\mathbf {v}),$$

\noindent
$\forall \mathbf {v}\in \mathrm{V}^s.$

We look for    ($\mathbf{ u}^m$, $\rho^m$ ) solution of  the  following  system
\begin{align}
&\mathbf{ u}^m = \sum_{i=1}^mc^m_i(t)\mathbf{ w}_i(x)\in C(0,T;\mathrm{W}_{(m)}) ,  \label{eq:5a}\\ \nonumber
&\rho^m \partial_t \mathbf{ u}^m + \rho^m \mathbf{ u}^m\cdot\nabla \mathbf{ u}^m -
 \nabla \cdot\mathbf{ T}\mathbf{u}^m + \nabla \pi^m =
\rho^m \mathbf{f},\\ \nonumber
&\partial_t \rho^m +\mathbf{ u^m}\cdot \nabla \rho^m=0, \\ \notag
&  \nabla\cdot \mathbf{u^m}=0, \;
\mathbf{ u}^m(0)= \mathbf{ u}_0^m, \; \mathbf{ u}^m_{\Gamma}=0,
\;\rho ^m(0)= \rho^m_0,\\ \nonumber
\end{align}

\noindent
where $\mathbf{u}^m(0)\rightarrow \mathbf{u}_0$ in $\mathrm{H}$ and $\rho^m _0 $
is  a  smooth function such that $\rho^m_0 \rightarrow \rho_0 $ in $L_q(\Omega)$ for some $q \geq 1$ and  $M \geq \rho_0^m \geq \frac{1}{m}  >0$ .

 \subsection{Continuity equation}

 \medskip
First,   we  treat the existence  of  the  continuity  equation $(\ref{eq:5a})_3$.

Assuming $\mathbf{ u}^m \in C^1(0,T; H^2(\Omega))$ is known, the solution of the continuity equation

\begin{align}
&\partial_t \rho^m +{\mathbf{ u}^m}\cdot \nabla \rho^m=0,\label{eq:6a}\\ \notag
&\rho^m(0)=\rho^m_0,\\ \nonumber
\end{align}

\noindent
 with $0<\frac{1}{m}\leq \rho_0^m \leq M$
is obtained by the method of the characteristics.

According, the characteristic $y^m$ to
(\ref{eq:6a}) are determined as a solution of the system
\begin{align}
&\frac{dy^m(\tau;t,x)}{d\tau}= \mathbf{u}^m(\tau;y^m),\label{eq:7a}\\ \notag
&y^m(0;t,x)=x. \\ \notag
\end{align}

The Cauchy-Lipschitz theory  provides the existence of a solution of (\ref{eq:7a}). Consequently,
the  explicit  form  for  $\rho^m$  is
$$\rho^m(x,t)= \rho^m_0(y^m(0,t,x)).$$

It  follows  the  maximum principle:
$$\frac{1}{m}\leq \rho^m(x,t)\leq M,  \;\forall (x,t)\in  Q_T.$$

 Integrating  over  $\Omega$  $(\ref{eq:6a})_1$,  we  get
 $$d_t\int_{\Omega}\rho^mdx =0,$$
 and  after integration   over   interval $(0,t)$,  the  conservation   of  mass results:

 $$\int_{\Omega}\rho^m(t)dx  =\int_{\Omega}\rho^m(0).$$

It  may  be noticed that, at  this  level, the estimate of  the density does  not  depend  on the  regularity  of
$\mathbf{ u}^m$.

\noindent
Furthermore,  if  $\mathbf{u}^m \in  \mathrm{W}_{(m)} $  then
$\partial_t\rho^m\in L_p(0,T;(\mathrm{ W}_q^{1}(\Omega))')\cap L_r(0,T;H^{-1}(\Omega))$  with $r\geq1$ and  $q= \frac{3p}{3-p}$ if $p<3$ or   $q>1$  if  $p\geq 3$,
 uniformly with respect to  $m$.

Now, applying  the gradient  operator  $\nabla$  to $(\ref{eq:6a})_1$,  we easily  get
\begin{align}
&\partial_t\nabla \rho^m +\nabla\mathbf{ u}^m\cdot \nabla \rho^m+\mathbf{ u}^m\cdot \nabla \nabla\rho^m=0.\label{eq:8a}\\  \notag
\end{align}
Multiplying (\ref{eq:8a})  by $|\nabla \rho^m|^{p-2}\nabla \rho^m$ ($p\geq 2$) and,  after  integration over $\Omega$, we get easily
$$\frac{d}{dt}|\nabla \rho^m|^p_p\leq  c|\nabla \rho^m|^p_p|\nabla \mathbf{u}^m|_{\infty},$$
\noindent
and  it  follows
$$|\nabla \rho^m(t)|_p \leq |\nabla \rho^m_0|_pexp\int_0^tc|\nabla \mathbf{u}^m|_{\infty}d\tau.$$
 It is  worth  to  note  that, at  first  level  of  regularity  of  the  density, we  need
$\nabla\mathbf{u}^m\in L_1(0,T;L_{\infty}(\Omega))$.

(In   the  case of compressible fluid  the  above regularity  for  the divergence  of the velocity is
 required  also  at  zero  level.
This  is the main  obstacle  to  prove the existence  of a  weak  solution  in that  context.
According,
to  obtain  some  results  on  the   existence  of  weak  solution for  compressible fluids
assumption  on  the  summability  of the  density does seem to be required.)

\noindent
We  continue  to prove  estimates   of  the solution  of  the  continuity  equation.
Multiplying (\ref{eq:8a})  by $(\rho^m)^{p-1}$ , ($p\geq 2$) and,  after  integration
over $\Omega$, we get

$$\frac{d}{dt} (\rho^m)^p+ \mathbf{u}^m\cdot \nabla (\rho^m)^p=0,$$
then
$$|\rho^m (t)|_p=|\rho^m_0|_p, $$
for   every $ t\in[0,T].$

This  simple  relation implies the  uniqueness of  the solution and
if $\rho_0^m$  converges  to $\rho_0$ in $ L_p(\Omega)$    then
$$|\rho^m (t)|_p\rightarrow|\rho_0|_p.$$
Later,  we  will  discuss  the consequences  of  the  above  estimate.

\noindent
Now, we pass to consider the existence problem for the system  (momentum equation)
\begin{align}
&\rho^m \partial_t \mathbf{u}^m + \rho^m \mathbf{u}^m\cdot\nabla \mathbf{u}^m -
\nabla\cdot\mathbf{T} \mathbf{u}^m +\nabla \pi^m=
\rho^m \mathbf{ f},\label{eq:9a}\\ \nonumber
&\mathbf{u}^m(0)= \mathbf{u}^m_0,\\ \nonumber
\end{align}
with  $\mathbf{u}^m=\sum_{i=1}^m c^m_i(t)\mathbf{w}^m_i$,  divergence  free.

It is  sufficient to prove  the  existence of  the coefficient $(c_i^m(t))_{i=1}^m$ such that (\ref{eq:9a}) holds.
For  this,
we project $(\ref{eq:9a})_1$  onto $\mathrm{W}_{(m)}$
and obtain
\begin{align}
&(\rho^m \partial_t \mathbf{ u}^m,\mathbf { w}_i) + (\rho^m \mathbf{ u}^m\cdot\nabla \mathbf{ u}^m,\mathbf{ w}_i) +
 (\mathbf{ T}(\mathbf{ D})\mathbf{u}^m,\mathbf{ D}(\mathbf{ w}_i)) = \label{eq:10a}\\ \nonumber
&(\rho^m \mathbf{ f}, {\mathbf w}_i), \quad i=1,...,m,\\ \nonumber
&\mathbf{ u}^m(0)= \mathbf{ u}_{0}^m.\\ \nonumber
\end{align}

\noindent
We introduce the matrix  $\mathbf{A}= (a_{i,j})_{i,j=1}^m$ with components
$$
a_{i,j}=\int_{\Omega}\rho^m\mathbf{ w}_i\mathbf{w}_jdx.
$$

The system (\ref{eq:10a}) can be written as
$$\mathbf{A}\frac{ d}{dt} \mathbf{ c}^m = \mathbf{H}(\mathbf{c}^m), \; \mathbf{ c}^m(0)= \mathrm{ c}^m_0, $$
\noindent
where $\mathbf{ c}^m= (c_i^m(t),c_2^m(t),...,c_m^m(t))$ and $\mathbf{H}$ is easily understood.
The existence of a solution of (\ref{eq:10a}) follows by standard ODE theory once we prove that $\mathbf { A}$ is invertible.
This fact is a consequence of the results  in   subsection 2.1   and of linear independence of the vectors
$\{\sqrt{\rho^m}\mathbf{w}_i\}$ ,  $i=1,2,\ldots, m$.

\noindent

(In \cite{rif2} the  invertibility of  $\mathbf{A}$ is   obtained  by "reductio ad  absurdum" ; in other  papers it is  assumed  true).

\noindent
Consequently, $\mathbf{A}$ is invertible and the system (\ref{eq:10a}) can be written
$$\frac {d}{dt} \mathbf{ c}^m = \mathbf{A}^{-1}\mathbf{H}(\mathbf{c}^m). $$
Standard ODE theory implies the  solvability of (\ref{eq:10a}) in   $(0,t_m)\subseteq  (0,T).$

 By the Carath\'eodory's extension theorem,  the global solvability of (\ref{eq:10a}) derives from the  following global a priori estimates for $\mathbf{u}^m$.
\subsection{A Priory Estimates}

Passing to consider the momentum equation, we
multiply $(\ref{eq:10a})_1$ by $c_i^m(t)$
 and summing the  result  over $i$ we get
\begin{align}
&\frac{1}{2}\int_{\Omega}(\rho^m\partial_t |\mathbf{u}^m|^2+ \rho_m\mathbf{u}^m \cdot \nabla |\mathbf{u}^m|^2 )dx +
(\mathbf{T}\mathbf{u}^m,\mathbf{D}(\mathbf{u}^m))= \label{eq:11a}\\ \nonumber
&\int_{\Omega}\rho^m\mathbf{ f}\mathbf{ u^m}dx.\\ \nonumber
\end{align}

Multiplying the continuity equation by $ |\mathbf{u}^m|^2/2$,
integrating over $\Omega$ and adding the result to (\ref{eq:11a}),  we get
$$
\frac{1}{2}\int_{\Omega}\partial_t (\rho^m | u^m|^2)dx + (\mathbf{ T}\mathbf{u}^m,\mathbf{D}(\mathbf{u}^m))= \int_{\Omega}\rho^m\mathbf{ f}\mathbf{ u}^mdx.
$$
According to the assumptions on $\mathbf {T}$ and the Korn inequality, after
integration over $(0,t)$ ($t<t_m$), we obtain

\begin{align}
&|\sqrt {\rho^m(t)} \mathbf{ u}^m(t)|_2^2 +\int_0^{t}\|\mathbf{ u^m}\|^p_pd\tau \leq
c\int_0^t\| \mathbf{f}\|^{p'}_{(\mathrm{V}_p)'}d\tau + |\sqrt{\rho^m_0}\mathbf{u}^m_0|_2^2.
\label{eq:12a}\\ \nonumber
\end{align}

Consequently,
$$\sqrt {\rho^m}\mathbf{ u}^m \in L_{\infty}(0,t_m;L_2(\Omega)),\;\mathbf{u}^m\in L_p(0,t_m;V_p),$$
uniformly with  respect to $m$ .

We  notice  that  $\mathbf{u}^m \in L_{\infty}(0,t_m;L _2(\Omega))$ holds   for  finite  $m$, only.

Thanks to the above uniform estimates,
$$|\mathbf{c}^m(t)|\leq C, $$
follows,
uniformly with respect to $m$,
jointly to  the  continuity  of  $\mathbf{c}^m(t)$ on  $[0,t_m)$.  By  Carath\'eodory's  theory  we can set $t_m=T$.

\noindent
We  estimate now the time-derivative of  the  unknowns.
 First, we recall that
$ 0<\frac{1}{m}\leq \rho^m\leq M$.
Moreover
$$\partial_t \rho^m \in L_q(0,T;H^{-1}(\Omega))\cap L_{p}(0,T;W^{-1}_r(\Omega),)$$
for arbitrary
$q\geq 1$ and $r$ is a  Sobolev conjugate. Since $\rho^m$ is bounded uniformly with respect to $m$, thanks to
 the  $L_q$-(Ascoli-Arzel$\grave{a}$)  theorem,
$\{\rho^m\}$ is a  compact set in $L_{\infty}(0,T; W^{-1}_r(\Omega))$ or in $L_{\infty}(0,T; H^{-1}(\Omega))$,
and, in addition,  is a compact set in
$C(0,T; (L_q(\Omega))_w)$ with  $q\geq r$  ,   for  example.

\noindent
Now, taking  into  account  the  continuity  equation, $(\ref{eq:9a})_1$ can  assume  the  form
\begin{align}
&\partial_t(\rho^m \mathbf{ u}^m) + \nabla\cdot(\rho^m \mathbf{ u}^m\otimes                                                                         \mathbf{ u}^m) -
\nabla\cdot\mathbf {T}\mathbf{u}^m +\nabla \pi^m=
\rho^m \mathbf{f}.\label{eq:13a}\\ \nonumber
\end{align}

Set
$$b(\rho^m\mathbf{u^m}, \mathbf{u^m},\mathbf{w}):=\int_{\Omega}\nabla\cdot(\rho^m \mathbf{ u}^m\otimes\mathbf{ u}^m )\mathbf{w}dx.$$

We  note that for   $\mathbf{w}\in  \mathrm{V^s}$
$$|b(\rho^m\mathbf{u^m}, \mathbf{u^m},\mathbf{w})|=|\int_{\Omega}\nabla\cdot(\rho^m \mathbf{ u}^m\otimes\mathbf{ u}^m )\mathbf{w}dx|=$$
$$|\int_{\Omega}(\rho^m \mathbf{ u}^m\cdot\nabla \mathbf{ w} )\mathbf{u}^mdx|\leq |\sqrt{\rho^m}\mathbf{ u}^m|_2^2\|\mathbf{ w}\|_{V^s},$$

\noindent
then
$$b(\rho^m\mathbf{u^m}, \mathbf{u^m},\mathbf{w}):=<\mathbf{g}^m(t),\mathbf{ w}>,                                               $$

\medskip
\noindent
where  $\mathbf{g}^m(t)$  belongs  to  a  bounded  set  in  $L_{\infty}(0,T;(\mathrm{V^s} )')$.

Considering the projection  operator $\mathrm{P}_m:\mathrm{H}\rightarrow \mathrm{W}_{(m)} $, since $\mathbf {T}_p\mathbf{u}^m,\;\mathbf{ g}^m  ,  \mathbf{f}$ belong to  a  bounded set  in $L_{p'}(0,T,(V^s)')$, uniformly with  respect  to  $m$, (\ref{eq:13a}) gives
$$\mathrm{P}_m\partial_t(\rho^m\mathbf{u}^m)\;
 \text{belongs  to  a  bounded  set  in}\; L_{p'}(0,T,(V^s)'), $$
 $  \text{ uniformly  with  respect  to}  \; m.$
\subsection{Time estimate - compactness results}
The a priori estimates obtained  in  the  last  subsection imply
$$\sqrt {\rho^m}\mathbf{ u}^m \in L_{\infty}(0,T;L_2(\Omega))\cap L_p(0,T;L_q(\Omega)),$$
$q$  is  a  Sobolev  conjugate.

We  deduce  that
$$ \rho^m\mathbf{ u}^m \rightharpoonup \mathbf{v} \;\text{ weakly}^*\;  \text{in}\; L_{\infty}(0,T;L_2(\Omega)).$$

 We  prove  next that
$$\mathbf{v}=\rho\mathbf{u}.$$
According  to  the  estimates  on the  density, it  holds, for example,
$$\|\partial_t\rho^m\|_{L_2(0,T; H^{-1}(\Omega))}\leq C.$$  Since $\{\rho_m(t)\}$  is  a  compact  set in  $H^{-1}(\Omega)$
for $ t  \in (0,T)$
hence
$$\rho^m \rightarrow \rho \;\text{strongly} \; \text{in}\; L_2(0,T; H^{-1}(\Omega)).$$

Now,
$\{\mathbf{u}^m\} $ is  a  bounded set in $L_p(0,T;W^1_{p,0}(\Omega))$ with  $p\geq 2$, then  follows

$$\mathbf{u}^m \rightharpoonup \mathbf {u} \; \text{weakly}\; \text{in}\; L_2(0,T; H_0^{1}(\Omega)).$$
Consequently,
$$\rho^mu_i^m \rightarrow \rho u_i,$$  in the  sense of distributions.

\noindent
So, we  have  proved  that
$\mathbf{v}=\rho\mathbf{u}$.

Finally,
$$\rho^m\mathbf{u}^m \rightharpoonup \rho \mathbf{u} \; \text  {weakly in }  \;L^2(Q_T). $$

Now, we  prove  that
$$\rho^mu_i^mu^m_j\rightarrow \rho u_iu_j.$$
The  estimate  of  the time  derivative  of  $\rho^m$ implies  the time continuity  of  $\rho^m(t)$   in $(L_2(\Omega))_w$.

At this  stage  we  consider the integral  identity that  is   satisfied by $\rho$  and  $\mathbf{u}$
with  smooth  $\psi$  with  $\psi(T)=0$, that  is

\begin{align}
&\int_0^T(\rho, \psi_t + \mathbf{u} \cdot \nabla \psi)dt = (\rho_0,\psi(0)).\label{eq:14a}\\ \nonumber
\end{align}

Kazhikhov- Smagulov  \cite{rif15} extend $\mathbf{u}$   and $\rho$ by  zeros   onto the  exterior  of  the  domain
 $Q_T$, so  (\ref{eq:14a})  holds  in   the region  $(-\infty,+\infty)\times \mathbb{R}^3$.  Using  test  function in  form of an  average (Steklov-functions) prove  that,
for  almost $t_1,t_2$  in  [0,T], is  valid
$$|\rho(t_1)|_q=|\rho(t_2)|_q,$$
for  $2\leq q<\infty$.

It  being  understood that  $(\rho(t_2)-\rho(t_1),\psi)\rightarrow 0$  when  $t_2\rightarrow t_1$,
  we obtain  that $\rho(t_2)\rightarrow \rho(t_1)$    strongly in  $L_2(\Omega)$  and,  in  virtue  of  the  boundedness  of  $\rho$,
the   convergence  is  valid   in   $L_q(\Omega)$  for  every finite $q\geq 1$.
It  is  immediate  that
$|\rho(t)|_q=|\rho_0|_q$.

\noindent
It  follows, in  particular,
$$\rho^m\rightarrow \rho  \;\text { strongly in } \;L_q(Q_T) .$$

\medskip
\noindent
Consequences  of the  previous estimates - summary.

\begin{enumerate}
\item $\rho^m
\rightharpoonup \rho$ weak* in $L_{\infty}(Q_T)$;
\item$\rho^m \rightarrow \rho$ strongly in $L_q(0,T;W^{-1}_r(\Omega))$ for arbitrary  finite $q >1$ and $r\in (2,p)$;
\item $\rho^m \rightarrow \rho$ strongly in $L_q(Q_T)$ for arbitrary finite  $q >1$ ;
\item $\rho^m \rightarrow \rho$  in $C([0,T];(L_{q}(\Omega))_w)$ for arbitrary finite $q >1$;
\item  $\rho^m \rightarrow \rho$ in $C([0,T];L_q(\Omega))$ for arbitrary finite  $q >1$;
\item $\mathbf{ u}^m   \rightharpoonup \mathbf{ u}$ weakly in $L_p(0,T;W^1_p(\Omega))$;
\item $ \rho^m u_i^mu_j^m \rightharpoonup \alpha_{i,j}$ weakly in
$L_{r}(Q_T )$ for some $r>1$ if $p> 6/5$.
\end{enumerate}
Now,  we prove that $\alpha_{ij}= \rho u_i u_j$ using  the method  introduced  in \cite{rif27}.

 To  begin  with, we  recall that
 $$\partial_t(P_m\rho^m \mathbf{ u}^m ) \;\text  {is uniformly bounded in} \;
 L_{p}(0,T;(W^s_{p.0}(\Omega))')  \;\text {with }\; s>1.$$

\noindent
In  addition, $\{P_m\rho^m\mathbf {u}^m(t)\} $ is a compact set
in $(W^1_{p,0}(\Omega))'$ for every  $t\in (0,T)$.

  It  follows
\begin{align*}
 &\int_0^T\int_{\Omega} \rho^{m}|\mathbf {u}^m|^2dxdt =
\int_0^T(\rho^{m}(t)\mathbf{u}^m(t),\mathbf{u}^m(t))dt =\\
&\int_0^T(P_m\rho^{m}(t)\mathbf{u}^{m}(t),\mathbf{u}^m(t))dt =\\
&\int_0^T<P_m\rho^m(t)\mathbf{u}^{m}(t),\mathbf{u}^m(t)>_{(W^1_{p ,0}(\Omega))'\times W^1_{p,0}(\Omega)}dt.
\end{align*}

Then
\begin{align*}
&\lim_{m\rightarrow \infty}\int_0^T\int_0^T<P_m\rho^{m}(t)\mathbf{u}^{m}(t),\mathbf{u}^m(t)>_{(W^1_{p,0}(\Omega))'\times W^1_{p,0}(\Omega)}dt=\\
&\int_0^T<\rho(t) \mathbf{u}(t),\mathbf{u}(t)>_{(W^1_{p,0})'\times W^1_{p,0}(\Omega)}dt =
\int_0^T(\rho(t)\mathbf{u}(t),\mathbf{u}(t))dt =\\
&\int_0^T\int_{\Omega}\rho|\mathbf{ u}|^2dxdt.
\end{align*}
Moreover, since
$$\rho^m \rightarrow\rho \; \text{strongly}\; in\; L^q(Q_T), \; \forall q >1,$$
we   have
$$\sqrt {\rho^m}\rightarrow \sqrt {\rho} \;\text{strongly}\; in\; L^q(Q_T),$$
consequently
$$\sqrt {\rho^m}\mathbf {u}^m\rightharpoonup{\sqrt \rho} \mathbf{ u}
\text{ weakly in}\; L^2(Q_T).$$

 By the  previous  results, we  can  conclude that
$$\sqrt {\rho^m}\mathbf {u}^m\rightarrow\sqrt {\rho} \mathbf{ u}
 \;\text{strongly in}\; L^2(Q_T).$$

 This implies that

$$\rho u_i^m u^m_j=\sqrt{\rho^m} u_i^m \sqrt{\rho^m}u^m_j\rightarrow \rho u_i u_j,$$

\noindent
in  the  sense  of  distributions and,  finally,

$$\alpha_{i,j}=\rho u_i u_j.$$

\subsection{Initial conditions}

The  time  evolution of  the integral  average

$$t\in(0,T)\mapsto \int_{\Omega}\rho(t,x)\psi(x)dx,$$

\noindent
is  governed  by  equation (\ref{eq:13a}). This  function,  considered  as  function   of  $t$,  is  absolutely  continuous,  in  other  words,  by  virtue  of  the estimates proved,

$$\rho\in C(0,T,L_{q,w}(\Omega)).$$

The  instantaneous value of  the  density  is  representable by  a  function $\rho(t)\in L_q(\Omega)$ .  So, the  initial  datum  is  assumed  in  the  following sense

$$\lim_{t\rightarrow 0^+}\int_{\Omega}\rho(t)\psi(x)ds:=\rho(0^+)=\int_{\Omega}\rho_0(x)\psi(x)dx.$$

\noindent
Analogously, the function
$$h(t)= \int_{\Omega}\rho(t)\mathbf{u}(t)\mathbf{\phi}(x)dx$$

\noindent
represents the  instantaneous value  of  the momentum and  the  initial  data  are  assumed in  the  following  sense:
$$\lim_{t\rightarrow 0^+}\int_{\Omega}\rho(t)\mathbf{u}(t)\mathbf{\phi}(x)dx:=\rho(0^+)\mathbf{u}(0^+)=\int_{\Omega}\rho_0(x)\mathbf{u}_0\mathbf{\phi}(x)dx.$$
\noindent
Since the  density can vanish, there  is  not information  concerning  the  time  weak continuity of  the  velocity .
\subsection{Monotony  and convergence of  stress  tensor}
The  previous estimates    imply that, for $m\rightarrow \infty$,
$$ P\mathbf{T}_p\mathbf{u^m}=P\nabla\cdot \mathbf{T}\mathbf{u^m} \rightharpoonup \mathbf{\chi}
\;\text {weakly in} \; (L_p(0,T;W^1_{p,0}(\Omega)))'.$$

\noindent
In  this  section  we  formulate  the  sufficient conditions which  permit to represent $\mathbf{\chi} $ as  $ P\mathbf{T}_p(\mathbf{u})$
($P$ is  a  "projection"   on  the  space of divergence free   function).
First, we  recall that
  for  a  fixed $\mathbf{w}_j$,  we have,  for  $m\rightarrow \infty$,
$$b(\rho^m\mathbf{u}^m,\mathbf{u}^m,\mathbf{w}_j)\rightarrow b(\rho\mathbf{u},\mathbf{u},\mathbf{w}_j),$$
in the sense  of  distributions  with  respect to  $t$  (  for  example  ).
In  fact,  for $\phi(t)\in C^{\infty}_0(0,T)$,
$$\int_0^Tb(\rho^m(t)\mathbf{u}^m(t),\mathbf{u}^m(t),\mathbf{w}_j)\phi(t)dt=$$
$$\int_0^T\phi(t)\int_{\Omega}\nabla\cdot(\rho^m(t) \mathbf{u}^m(t)\otimes\mathbf{u}^m(t))\mathbf{w}_jdxdt=$$
$$-\int_0^T\phi(t)\int_{\Omega}(\rho^m u_i^m(t)u_k^m(t))\partial_kw_{ij}dxdt
\rightarrow$$
$$\int_0^T\phi(t)\int_{\Omega}(\rho u_iu_k)\partial_kw_{ij}dxdt=\int_0^Tb(\rho \mathbf{u}(t),\mathbf{u}(t),\mathbf{w}_j(t))\phi(t)dt.$$

Consequently,  (\ref{eq:12a} )  implies
$$(\partial_t(\rho(t) \mathbf{u}(t)),\mathbf{w}_j)+ (\mathbf{\chi}(t),\mathbf{w}_j)+  b(\rho(t)\mathbf{u}(t),\mathbf{u}(t),\mathbf{w_j})=
(\rho(t)\mathbf{f}(t),\mathbf{w}_j).$$

It  follows
\begin{align}
&&(\partial_t(\rho(t) \mathbf{u}(t)),\mathbf{v})+ (\mathbf{\chi}(t),\mathbf{v})+  b(\rho(t)\mathbf{u}(t),\mathbf{u}(t),\mathbf{v})=
(\rho(t)\mathbf{f}(t),\mathbf{v}),\label{eq:15a}\\ \nonumber
\end{align}

\noindent
$ \forall \mathbf{v}\in  \mathrm{V^s} $
  and a.e.  in  $(0,T)$.

\noindent
 Since $W^1_p(\Omega)\hookrightarrow  L_q(\Omega)$  with
$\frac{1}{q}=\frac{1}{p}-\frac{1}{n}$,  ($p>1$) it  follows that the
form
$$b(\rho(t)\mathbf{u}(t),\mathbf{u}(t),\mathbf{v}), $$

\noindent
 is  continuous on $\mathrm{V}_p$
provided $\frac{2}{q}+\frac{1}{p}\leq 1$, i.e. if  $p\geq\frac{3n}{n+2}$.

\noindent
Thus, by  continuous extension, (\ref{eq:15a}) holds  $ \forall
\mathbf{v}\in \mathrm{V}_p$.

\noindent
Now,  we  analyze the  time-summability  of $b(\cdot,\cdot,\cdot)$.

\noindent
We consider $\frac{1}{p}-\frac{1}{n}>0$, otherwise there
is nothing  to  prove.

The  estimates  in subsection 3.5,  the  Sobolev theorem and
interpolation theory imply
$$\sqrt{\rho}\mathbf{u}\in  L_{\infty}(0,T);L_2(\Omega))\cap L_p(0,T;L_q(\Omega))\subset L_r(0,T;L_s(\Omega)), $$
where
$$\frac{1}{r}=\frac{1-\theta}{p} +\frac{\theta}{\infty}=\frac{1-\theta}{p},\; \frac{1}{s}=\frac{1-\theta}{q} +\frac{\theta}{2}.$$
Choosing $\theta$  such  that $\frac{1}{r}=\frac{1}{s}$, that  is
$\theta =\frac{2}{n+2}$, then
$$L_{\infty}(0,T);L_2(\Omega))\cap L_p(0,T;L_q(\Omega))\subset L_{s}(Q_T).$$
Finally, the result is  implied by  $\frac{2}{s}+\frac{1}{p}\leq 1$,
i.e. $p\geq \frac{11}{5}$.  In  conclusion,  for  such a  $p$  the
function
$$  t\rightarrow b(\rho(t)\mathbf{u}(t),\mathbf{u}(t),\mathbf{v}) =-\int_{\Omega}\sqrt{\rho(t)}\mathbf{u}(t)\cdot(\sqrt{\rho(t)}\mathbf{u}(t)\cdot\nabla \mathbf{v})dx$$
 is  in  $L_1(0,T)$ with $\frac{1}{s}=\frac{n}{(n+2)p}$.

Now,  integrating (\ref{eq:15a})  over  $(0,t)$,  we  get
\begin{align}
&(\rho(t)\mathbf{u}(t),\mathbf{v})-(\rho_0\mathbf{u}_0,\mathbf{v})= \label{eq:16a}\\ \nonumber
&-\int_0^t(b(\rho(\tau)\mathbf{u}(\tau),\mathbf{u}(\tau),\mathbf{v}) +  (\mathbf{\chi}(\tau),\mathbf{v})+ (\rho f(\tau), \mathbf{v}))d\tau. \\ \nonumber
\end{align}
Passing  to  the  limit  for  $t\rightarrow 0^+$ in (\ref{eq:16a}),
we  get
\begin{align}
&\lim_{t\rightarrow 0^+}(\rho(t)\mathbf{u}(t),\mathbf{v})-(\rho_0\mathbf{u}_0,\mathbf{v})=0.
\label{eq:17a} \\ \nonumber
\end{align}
Recalling
$$
\sqrt{\rho^m(t)}\mathbf{u}^m(t) \rightarrow \sqrt{\rho(t)}\mathbf{u}(t),$$
 strongly in $L_2(\Omega)) $  for   almost  all  $t\in [0,T]$ and letting $m \rightarrow \infty$
in (\ref{eq:12a}) we  get for  $\alpha >0$
 \begin{align}
 &(\rho(t)\mathbf{u}(t),\mathbf{u}(t)) \leq (\rho_0\mathbf{u}_0,\mathbf{u}_0) +ct^{\alpha}.
 \label{eq:18a}\\ \nonumber
 \end{align}

 \noindent
 Now,  since
$$|\sqrt{\rho(t)}(\mathbf{u}(t)-  \mathbf{u}_0)|_2^2=(\rho(t)\mathbf{u}(t),\mathbf{u}(t))-
 2(\rho(t)\mathbf{u}(t),\mathbf{u}_0)+(\rho(t)\mathbf{u}_0,\mathbf{u}_0),$$
  taking  into account  (\ref{eq:18a})  and   the  properties   of $\rho(t)$,  we  conclude
  $$\lim_{t\rightarrow 0^+}|\sqrt{\rho(t)}(\mathbf{u}(t)-  \mathbf{u}_0)|_2^2=0.$$
  Thus,  the  previous results  imply
  \begin{align}
  &\lim_{t\rightarrow 0^+}(\rho(t)\mathbf{u}(t),\mathbf{u}(t))=(\rho_0\mathbf{u}_0,\mathbf{u}_0).
  \label{eq:19a}\\ \nonumber
\end{align}

\begin{proposition}-  The  following  formula  holds
  for  almost  all  $t_0$,  $t_1$  $\in  [0,T]\;(t_0<t_1)$,

\begin{align}
&\int_{t_0}^{t_1}\partial_t(\rho\mathbf{u}(t),\mathbf{u}(t))dt=\frac{1}{2}|\sqrt{\rho(t_1)}\mathbf{u}(t_1)|_2^2-
\frac{1}{2}|\sqrt{\rho(t_0)}\mathbf{u}(t_0)|_2^2-\label{eq:20a}\\ \nonumber
&\int_{t_0}^{t_1}(\nabla\cdot (\rho\mathbf{u}(t)\otimes\mathbf{u}(t)),\mathbf{u}(t) )dt.
\\ \nonumber
\end{align}
\end{proposition}

\begin{proof}

Using  convolution product or  averaging, it   is  possible  prove
that (\ref{eq:20a}) holds  true.

\noindent
In  fact, let $\phi(t)$ be a regularizing
kernel, i.e.
$$\phi_{\epsilon}:=\frac{1}{\epsilon}\phi(\frac{\cdot}{\epsilon}):=
\frac{1}{\epsilon}\phi(\frac{\tau -t}{\epsilon}),$$
with a even  mollifier $\phi \in \mathcal{D}_+(\mathbb{R})$ and $\int_{\mathbb{R}_+}\phi(t)dt=1$.

Set $$\mathbf{u}_{\epsilon}(t)=\mathbf{u}\ast\phi_{\epsilon}(t)=\frac{1}{\epsilon}\int_{|\tau- t|<
\epsilon}\mathbf{u}(\tau)\phi_{\epsilon}(\frac{\tau - t}
{\epsilon})d\tau=\int_{|z|<1}\mathbf{u}(t+\epsilon z)\phi(z)dz.
$$
Notice
$$|\partial_t \mathbf{u}_{\epsilon}|_p\leq \frac{c}{\epsilon}|\mathbf{u}|_p.$$

We  analyze
\begin{align}
&\theta_{\epsilon}(t_0,t_1):= \int^{t_1}_{t_0}(\partial_t(\rho(t) \mathbf{u}(t)),\mathbf{u}\ast\phi_{\epsilon}\ast\phi_{\epsilon}(t))dt,
\label{eq:21a}\\ \nonumber
\end{align}
with $t_0> \epsilon$  and $t_1< T-\epsilon$.

By  integration  by  parts  and by the  properties  of  the  convolution, we  have
 \begin{align}
&\theta_{\epsilon}(t_0,t_1)=(\rho \mathbf{u}(t_1),\mathbf{u}\ast\phi_{\epsilon}\ast\phi_{\epsilon}(t_1))-
(\rho \mathbf{u}(t_0),\mathbf{u}\ast\phi_{\epsilon}\ast\phi_{\epsilon}(t_0))-
 \label{eq:22a}\\  \nonumber
&\int_{t_0}^{t_1}( (\rho \mathbf{u})\ast\phi_{\epsilon}(t),\partial_t\mathbf{u}\ast\phi_{\epsilon}(t))dt.\\ \nonumber
\end{align}
We  consider the  last term of  the  right hand  side of
(\ref{eq:22a}).
\begin{align}
&\int_{t_0}^{t_1}((\rho \mathbf{u})\ast\phi_{\epsilon}(t),\partial_t\mathbf{u}\ast\phi_{\epsilon}(t))dt=
\int_{t_0}^{t_1}(\rho(t) (\mathbf{u}\ast\phi_{\epsilon}(t)),\partial_t\mathbf{u}\ast\phi_{\epsilon}(t))dt+\label{eq:23a}\\ \nonumber
&\int_{t_0}^{t_1}(\frac{1}{\epsilon}((\rho \mathbf{u})\ast\phi_{\epsilon}(t)-\rho(t)(\mathbf{u}\ast\phi_{\epsilon}(t)),
\epsilon\partial_t\mathbf{u}\ast\phi_{\epsilon}(t))dt =\\ \nonumber
&\frac{1}{2}((\rho (t_1),|\mathbf{u}\ast\phi_{\epsilon}(t_1)|^2)-
(\rho (t_0),|\mathbf{u}\ast\phi_{\epsilon}(t_0)|^2)) +\\ \nonumber
&\frac{1}{2}\int_{t_0}^{t_1}(\nabla\cdot(\rho(t)\mathbf{u}(t)),|\mathbf{u}\ast\phi_{\epsilon}(t)|^2 )dt+ \\ \nonumber
&\int_{t_0}^{t_1}(\frac{1}{\epsilon}((\rho \mathbf{u})\ast\phi_{\epsilon}(t)-\rho(t)(\mathbf{u}\ast\phi_{\epsilon}(t)),
\epsilon\partial_t\mathbf{u}\ast\phi_{\epsilon}(t))dt.\\ \nonumber
\end{align}

Now, setting  $\tau-t =\epsilon z $  and thank  to  the  continuity  equation,  we obtain  as   $\epsilon  \rightarrow  0$
\begin{align*}
&\int_{t_0}^{t_1}\frac{1}{\epsilon}(((\rho\mathbf{u})\ast\phi_{\epsilon}(t)-\rho(t)(\mathbf{u}\ast\phi_{\epsilon}(t)),
\epsilon\partial_t\mathbf{u}\ast\phi_{\epsilon}(t))dt=-\\ \nonumber
&\int_{t_0}^{t_1}(\int_{|z|<1}\frac{1}{\epsilon}(\int_{t}^{t+\epsilon z}\nabla\cdot(\rho\mathbf{ u}(s))ds ) \mathbf{u}(t+\epsilon z)\phi(z)dz,\epsilon\partial_t \int_{|z|<1}\mathbf{u}(t+\epsilon z)\phi( z)dz)dt\\ \nonumber
& \\
&\rightarrow 0 .\\
\end{align*}
Indeed, making  use  of  the strong  convergence  of  translations  in  $L_p$ and  the  strong
 convergence  of  the Steklov function, we have  for  $\epsilon \rightarrow 0$

\begin{align*}
&\mathbf\int_{|z|<1}\frac{1}{z\epsilon}(\int_{t}^{t+\epsilon z}\rho(s)\mathbf{u}(s)ds)\cdot
 \nabla \mathbf{u}(t+\epsilon z)z\phi(z)dz\rightarrow\\
 &\rho(t)\mathbf{u}(t)\cdot\nabla \mathbf{u}(t)\int_{|z|<1}z\phi(z)dz=0.\\
\end{align*}
Similarly,
 $$\epsilon\partial_t \int_{|z|<1}\mathbf{u}(t+\epsilon z)\phi( z)dz\rightarrow \mathbf{u}(t)\int_{|z|<1}
 \partial_z\phi( z)dz=0,$$
a.e.  in  $(0,T)$.

In  conclusion,  as   $\epsilon \rightarrow 0$, we   get
\begin{align*}
&\theta_{\epsilon}(t_0,t_1)\rightarrow \int^{t_1}_{t_0}(\partial_t(\rho(t) \mathbf{u}(t)),\mathbf{u}(t))dt=\\
&\frac{1}{2}|\sqrt{\rho(t_1)}\mathbf{u}(t_1)|_2^2-
\frac{1}{2}|\sqrt{\rho(t_0)}\mathbf{u}(t_0)|_2^2-
\int_{t_0}^{t_1}(\nabla\cdot (\rho\mathbf{u}(t)\otimes\mathbf{u}(t)),\mathbf{u}(t))dt.
\end{align*}
Thus  (\ref{eq:20a})  is  proved.
\end{proof}
Now,  using (\ref{eq:20a}) and  setting  $\mathbf{u} =\mathbf{v}$
in the  relation
\begin{align}
&(\partial_t(\rho(t) \mathbf{u}(t),\mathbf{v}(t))+ (\mathbf{\chi}(t),\mathbf{v}(t))+ (\nabla\cdot(\rho(t)\mathbf{u}(t)\otimes\mathbf{u}(t)),\mathbf{v}(t))=\label{eq:24a}\\ \nonumber &(\rho(t)\mathbf{f}(t),\mathbf{v}(t)),
\\ \nonumber
\end{align}

\noindent
we  get

\begin{align}
&\frac{1}{2}((\rho (t_1),|\mathbf{u}(t_1)|^2)-(\rho_0,|\mathbf{u}_0|^2))+
\int_{0}^{t_1}(\mathbf{\chi}(t),\mathbf{u}(t)) dt= \label{eq:25a}\\ \nonumber
&\int_{0}^{t_1}(\rho(t)\mathbf{f}(t),\mathbf{u} (t ))dt.\\ \nonumber
\end{align}

\subsection{Conclusion}

Now we prove  that  $$\mathbf{ \chi}=P\mathbf{T}_p\mathbf{u}.$$

  We use  the  relation
 \begin{align}
  &\frac{1}{2}|\sqrt{\rho}\mathbf{u}(t)|_2^2 + \int_0^t(\mathbf{\chi}(\tau),\mathbf{u}(\tau))d\tau\geq
 \int_0^t(\rho \mathbf{f}(\tau),\mathbf{u}(\tau))d\tau +\frac{1}{2}|\sqrt{\rho_0}\mathbf{u}_0|_2^2.
 \label{eq:26a}\\ \nonumber
 \end{align}

 We  introduce,  for  $\mathbf{\phi}\in L_p(0,T;V_p)$,

$$X^{s}_m=\int_{0}^s(\mathbf{T}_p\mathbf{u^m}(t)-\mathbf{T}_p\mathbf{\phi}(t),\mathbf{u^m}(t)-\mathbf{\phi}(t))dt+
\frac{1}{2}|\sqrt{\rho}\mathbf{u}(s)|^2_2$$
a.e. in $(0,T)$.

 Thanks to  the  monotony of  the  operator  $\mathbf{T}_p$,  we  deduce  that
\begin{align}
 &\liminf_{m  \rightarrow \infty}X^s_m\geq \frac{1}{2}|\sqrt{\rho}\mathbf{u}(s)|^2_2.
\label{eq:27a}\\ \nonumber
\end{align}
(\ref{eq:11a}) implies
\begin{align}
&X^{s}_m=\int_0^s(\rho(t)\mathbf{f}(t),\mathbf{u}^m(t))dt
+\frac{1}{2}|\sqrt{\rho^m}\mathbf{u}^m(0)|^2_2-\label{eq:28a}\\ \nonumber
&\int_0^s(\mathbf{T}_p\mathbf{u}^m(t),\mathbf{\phi}(t))dt-
\int_0^s(\mathbf{T}_p\mathbf{\phi}(t),\mathbf{u}^m(t)-\mathbf{\phi}(t))dt\rightarrow {X}^s,
\\ \nonumber
\end{align}

\noindent
with
$${X}^s= \int_{0}^s(\rho\mathbf{f}(t),\mathbf{u}(t))dt +\frac{1}{2}|\sqrt{\rho_0}\mathbf{u}_0|^2_2-\int_0^s(\mathbf{\chi} (t),
\mathbf{\phi}(t))dt-
\int_{0}^s(\mathbf{T}_p\mathbf{\phi} (t),\mathbf{u}(t)-\mathbf{\phi}(t))dt   ,$$ and   by  (\ref{eq:27a}) we
get
\begin{align*}
&\int_0^s(\rho\mathbf{f}(t),\mathbf{u}(t))dt +\frac{1}{2}|\sqrt{\rho_0}\mathbf{u}_0|^2_2-\int_0^s(\mathbf{\chi}(t),\mathbf{\phi}(t))dt-\\
&\int_0^s(\mathbf{T}_p\mathbf{\phi}(t),\mathbf{u}(t)-\mathbf{\phi}(t))dt\geq
\frac{1}{2}|\sqrt{\rho}\mathbf{u}(s)|^2_2,\\
\end{align*}

\noindent and, finally, (\ref{eq:26a}) implies

$$\int_{0}^{s}(\mathbf{\chi}(t)-\mathbf{T}_p\mathbf{\phi}(t),\mathbf{u}(t)-\mathbf{\phi}(t))dt\geq 0, $$
 for almost  $s$.

  Since $\mathbf{T}_p $ is a     monotone and  hemicontinuous  operator, we have

$$\mathbf{\chi}=P\mathbf{T}_p\mathbf{u}.$$

The couple ($\rho,\mathbf{u}$) is  a  weak  solution of   system ($8$).
Theorem 2  is  proved.

\end{proof}

\section{Periodic problem}

This section is devoted to the existence of periodic solution of
problem  (\ref{eq:4a}), i.e.  we  look  for  a  solution  of
\begin{align}
 &\rho(\partial_t\mathbf{ u} +  \mathbf{ u}\cdot \nabla  \mathbf{u})
 - \nabla \cdot(\mathbf{ T}\mathbf{ u} - \pi I)  =\rho \mathbf{ f},\label{eq:p0}\\ \nonumber
 &\partial_t \rho +\nabla \cdot(\rho \mathbf{ u})=0, \\ \nonumber
 &\nabla \cdot \mathbf{ u} =0,  \;
 \mathbf{ u}(0)=\mathbf{ u}(T),\; \mathbf{ u}_{\Gamma}=0, \; \rho(0)=\rho(T).\\ \nonumber
 \end{align}

\begin{theorem}
  Assume the hypotheses  of  Theorem 3.1. Further
  $\mathbf{f}(x,t) $ is time periodic  with  period  $T$ in  $L_2(Q_T)$,  for  convenience,  $0< \alpha\leq \rho_0\leq \beta$ and $p\geq \frac{11}{5}$.                                              Then  there  exists  a  weak  solution  of  system (\ref{eq:p0}).
\end{theorem}
 \begin{proof}

To  prove the  existence of  a  weak  solution of (\ref{eq:p0}) we proceed as in Theorem 3.1.
 We follow the next scheme: first,  we assign the
velocity and consider a parabolic approximation of the continuity equation (a diffusion equation) proving the existence
of  a periodic solution and associated a priori estimates. Next, we pass to consider the existence of a
linearized momentum equation and then we conclude the scheme
passing  to  the  limit.

Since  we  have  treated, in  subsection 3.3, with the  semi-Faedo-Galerkin approximation, we  start with
\begin{align}
&\partial_t \rho^m(t) + \mathbf{u}^{m-1}(t) \cdot \nabla \rho^m(t)- \frac{1}{m} \Delta \rho^m(t) +
 \rho^m(t) = \rho^{m-1}(t), \label{eq:p1}\\ \nonumber
& \partial_{\mathbf{n}} \rho(x,t) = 0  \; on \; \Gamma, \;\rho^m(0) = \rho^m_0.
\\ \nonumber
\end{align}
$\mathbf{u}^{m}$  is  defined  in  subsection   3.3.
We assume $0<  \alpha\leq \rho^m_0 \leq \beta$  ($\alpha, \beta$  real positive numbers)  and $\mathbf{u}^{m-1}$,  $\rho^{m-1}$ are assigned  periodic  functions and  $\alpha \leq \rho^{m-1} \leq  \beta$.

\noindent
The existence of a
solution of problem (\ref{eq:p1}) is well known in literature.  We
need some estimates of the solution $\rho^m$.

\noindent

First, we prove that $\alpha\leq \rho^m \leq \beta$ ( maximum  principle).

In fact, multiplying
(\ref{eq:p1}) by $(\rho^m-\alpha)^- := min(0, \rho^m-\alpha)$, after integration by
parts, we have
\begin{align*}
&d_t|(\rho^m(t) -\alpha)^-|_2^2 + \frac{2}{m}| \nabla(\rho^m(t) -\alpha)^-|_2^2 +  2|(\rho^m(t) -\alpha)^-|_2^2 = \\
&2((\rho^{m-1}(t) -\alpha),(\rho^m(t)-\alpha)^-)\leq 0.\\
\end{align*}
Consequently, $\rho^m(t) \geq \alpha$  for  all   $t$.

  Analogously, multiplying (\ref{eq:p1}) by $(\rho^m-\beta)^+ := sup(0,
\rho^m-\beta)$ and,  after integration by parts, we get
\begin{align*}
&d_t|(\rho^m(t) -\beta)^+|_2^2 + \frac{2}{m}|\nabla (\rho^m(t) -\beta)^+|_2^2 + 2|(\rho^m(t)-\beta)^+|_2^2 = \\
&2((\rho^{m-1}(t)- \beta),(\rho^m(t)-\beta)^+)\leq 0,
\end{align*}
thus $\rho^m\leq \beta$ for all  $t$.

\noindent
 Now, following the procedure of subsection 3.3, we  get $H^2$-regularity;   multiplying   (\ref{eq:p1})  by
 $-\frac{1}{m}\Delta \rho^m$,  then  integrating   over  $Q_t$ and
recalling  $|\nabla \rho^m|^2_4\leq  c|\rho^m|_{\infty}|\Delta \rho^m|_2 $,

\noindent
we obtain
\begin{align*}
&\frac{1}{m}|\nabla \rho^m(t)|_2^2 + \frac{1}{m ^2}\int_0^t|\Delta \rho^m(\tau)|_2^2d\tau
+\frac{1}{m}\int_0^t|\nabla\rho^m(\tau)|_2^2d\tau\leq \\
&\frac{2}{m}|\nabla
\rho^m_0|_2^2 + c\int_0^t(\|u^{m-1}(\tau)\|^2 +|\rho^{m-1}(\tau)|^2_2)d\tau.\notag\\
\end{align*}

Similarly, we  can obtain     $H^3$- estimate of $\rho^m$ but  we  omit  details.
  By the  previous  estimates,  $\partial_t\rho^m \in  L_2(Q_T)$ .
The existence of a solution of (\ref{eq:p1}) permits to define a
map $\mathrm{S}: L_2(\Omega)\rightarrow L_2(\Omega)$,
\begin{align}
\mathrm{S}\rho^m(0)=\rho^m(T). \label{eq:p2}\\\nonumber
\end{align}
$\mathrm{S}\: \text {is a continuous map in}\; L_2(\Omega).$

In fact, let $\rho_1^m, \rho_2^m$ be
solutions of (\ref{eq:p1})
corresponding to initial conditions

\noindent
$\rho^m_1(0),\rho^m_2(0)$,
 respectively ( with  bounds $\alpha,\beta$ ).

 \noindent
  By (\ref{eq:p1}) we get

$$|\mathrm{S}\rho^m_1(0)-\mathrm{S}\rho^m_2(0)|_2^2 = |\rho^m_1(T)-\rho^m_2(T)|_2^2\leq|\rho^m_1(0)-\rho^m_2(0)|^2_2e^{-c_m T},$$
consequently,

$$\mathrm{ S} \;\text  {is a }\;L_2-\text {contraction map.}$$
Then,  the  fixed  point  of $\mathrm{S}$  yields   periodic  solution
of (\ref{eq:p1}).

\noindent
Briefly, we  show  that  the  fixed  point  is  in  $H^1(\Omega)$.

Indeed,  $\rho^m\in  L_2(0,T;H^1(\Omega))$ then  in $t=\epsilon >0$,
$\rho^m(\epsilon,x)\in H^1(\Omega)$.  Considering  the equation (\ref{eq:p1}) in the  interval $[\epsilon,T]$, we  can show that
  $\rho^m\in  C(\epsilon,T;H^1(\Omega))$ and    follows  $\rho^m(T)\in H^1(\Omega)$,  consequently  $\rho^m(0)\in  H^1(\Omega)$.

Now, we pass to consider the existence of a periodic solution of the momentum equation.

We assume, appealing  to  section 3,  the  existence  of
a solution of  the following system
\begin{align}
&\rho^m(\partial_t\mathbf{u}^m +  \mathbf{u}^{m-1}\cdot \nabla \mathbf{u}^m -
\mathbf{f}) +\frac{1}{2}(\frac{1}{m}\Delta \rho^m -\rho^m+\rho^{m-1})\mathbf{u}^m - \label{eq:p4}\\ \nonumber
&\nabla\cdot\mathbf{T} \mathbf{u}^m
+ \nabla \pi^m = 0,\\ \nonumber
&\nabla\cdot \mathbf{u}^m=0, \;\mathbf{u}^m(0)=\mathbf{u}^m_0. \\ \nonumber
\end{align}
Now, multiplying $(\ref{eq:p4})_1$ by
$\mathbf{u}^m$, integrating  by parts, applying the

\noindent
Poincar\'e's inequality  and  taking  into  account  the  diffusion  equation, we
have
\begin{align*}
&|\sqrt{\rho^m(t)}\mathbf{u}^m(t)|^2_2 \leq e^{-ct}(|\sqrt{\rho^m(0)}\mathbf{u}^m(0)|^2_2
+\int_0^te^{cs}|\rho \mathbf{f}(s)|^2_2ds).\\ \nonumber
\end{align*}

 We consider now the map
 $$ \mathrm{Z}:\sqrt{\rho^m(0)} \mathbf{u}^m(0) \rightarrow \sqrt{\rho^m(T)}\mathbf{u}^m(T).$$

We  prove  that $\mathrm{Z}$  is  a  contraction on  $L_2(\Omega)$.

First,
let $B(R)$ be a ball with radius $R$.

If $R\geq
(1-e^{-cT})^{-1}\int_0^T|\rho f|^2_2dt$ we can  prove  easily

       $$\mathrm{Z}B(R) \subseteq B(R).$$
Let $\mathbf{u}^m_1,\mathbf{u}^m_2$ be solutions of problem (\ref{eq:p4}) with initial conditions $\mathbf{u}^m_1(0),\mathbf{u}^m_2(0)$, respectively.
Thus,
$\mathbf{U}^m =\mathbf{u}^m_1-\mathbf{u}^m_2$ satisfies
\begin{align}
  &\rho^m(\partial_t\mathbf{U}^m +  \mathbf{u}^{m-1}\cdot \nabla \mathbf{U}^m ) -
  (\nabla\cdot(\mathbf{T}(\mathbf{u}_2^m)\mathbf{u}_2^m)+
  \nabla\cdot(\mathbf{T}(\mathbf{u}_1^m)\mathbf{u}_1^m)-\label{eq:p5}\\ \nonumber
&\frac{1}{2}\rho^m \mathbf{U}^m +
\frac{1}{2}\rho^{m-1}\mathbf {U}^m+\frac{1}{2m}\Delta \rho^m \mathbf{ U}^m +
\nabla \bar \pi^m = 0. \\ \nonumber
 \end{align}
Multiplying (\ref{eq:p5}) by $\mathbf{U}^m $,  integrating  the  result  over  $\Omega$,
 taking into  account  the  diffusion equation (\ref{eq:p1}) and the  structure  conditions,  the Gronwall's lemma implies
$$
|\sqrt {\rho^m (T)} \mathbf{U}^m(T)|_2^2\leq e^{-ct}|\sqrt {\rho^m(0)} \mathbf{U}^m(0)|^2_2.
$$
Thus the  map  $\mathrm{Z}$ is  a $L_2$- contraction and then
 the  existence of  a  periodic solution  of (\ref{eq:p4}) follows  as the   fixed  point  of  $\mathrm{Z}$.
 From now on  the proof of the existence of periodic weak solution proceeds  as in  Theorem 3.1.
\end{proof}

In the case of unbounded domains it is not possible to simply extend
the methods used for the bounded domains, since these involve, in
general, tools such as the Poicar\'e's inequality, compact
embedding, etc., that no longer holds for unbounded domains, in
general. Consequently, it is necessary to resort different
arguments. In \cite{rif31} the author solved the open problem of the
existence of weak and strong periodic solutions for the
Navier-Stokes equations in exterior domains using an alternative method. For the time being, this is  the only result for three-dimensional case
in exterior  domains.

\section{Variational  inequality}
 The differential inequalities appear in the presence of additional constraints
imposed on the unknowns of the problem in order to describe particular physical situations.
These type of problems  have been studied for instance in \cite{rif4}, \cite{rif7}, \cite{rif10},
\cite{rif16}, \cite{rif20}, \cite{rif21}, \cite{rif22}, \cite{rif25}, \cite{rif26}. In  \cite{rif20} the existence of  a global
weak solution of  the Navier-Stokes  equations  in a  convex set
is  obtained  under some additional  conditions.  In  \cite{rif7}  the  existence problem was  proved
 for  two  dimensional case   using  a  method of non linear  semigroup.
In \cite{rif26}   the  existence  of a  global  weak solution  was solved, in  three  dimension and in
the  general case,    for density-dependent Navier- Stokes equations, applying  the compactness method.
For the time being, this is  the only result for three-dimensional case.

The lines of the proof of  the existence for  variational inequality are  traced in subsection 3.3;
we  use  semi-Faedo-Galerkin  approximation jointly with a penalization operator.

\subsection{Formulation  of  the problem}
The  functional spaces  are those of  the section 3.  In  addition, for  convenience,  let  $\mathrm{K}\subset  \mathrm{H}$ be  an  arbitrary closed  convex set (independent  of  time) with  $0\in \mathrm{K}$.
We  assume  the  initial  data $\mathbf{u}_0\in  \mathrm{K}$, $ \rho_0\in L_{\infty}(\Omega)$ and  $0\leq \rho_0\leq  M $, the  external  force $\mathbf{f}\in L_2(Q_T)$ ( for  convenience).

We   consider  the following system for  $\mathbf{u},\mathbf{v}\in \mathrm{K}$,
\begin{align}
 &(\rho(t)\partial_t\mathbf{u}(t),\mathbf{v}(t)-\mathbf{u}(t)) + (\rho(t) \mathbf{u}(t)\cdot \nabla \mathbf{u}(t),   \mathbf{v}(t)-\mathbf{u}(t))
 - \label{eq:v1}\\  \nonumber
 &(\nabla \cdot\mathbf{T}\mathbf {u}(t),\mathbf{v}(t)-\mathbf{u}(t))  \geq
 (\rho(t) \mathbf{f}(t),\mathbf{v}(t)-\mathbf{u}(t)), \\ \nonumber
 &\partial_t \rho +\nabla \cdot(\rho \mathbf{u})=0, \\ \nonumber
 &\nabla \cdot \mathbf{ u} =0, \mathbf{ u}(0)=\mathbf{ u}_0,\; \mathbf{ u}_{\Gamma}=0, \;\rho(0)=\rho_0 .\\ \nonumber
\end{align}

Then, we  give  the  definition  of weak solution of  problem (\ref{eq:v1}). In  what follows  $\Omega$  is  a  bounded domain  in $\mathbb{R}^3$.

\begin{definition}  $(\mathbf{u},\rho)$  is  a  weak  solution  of (\ref{eq:v1})
if  hold
\begin{align}
 &\int_0^T((\rho\partial_t\mathbf{v },\mathbf{v}-\mathbf{u}) + (\rho \mathbf{ u}\cdot \nabla
  \mathbf{u},\mathbf{v}-\mathbf{u})
 - (\nabla \cdot\mathbf{T} \mathbf{u},\mathbf{v}-\mathbf{u})-  \label{eq:v2}\\ \nonumber
&(\rho \mathbf{ f},\mathbf{v}-\mathbf{u}))dt \geq
-|\sqrt{\rho(0)}(\mathbf{v}(0)-\mathbf{u}(0))|^2,\\ \nonumber
 &\partial_t\rho +\nabla \cdot(\rho \mathbf{ u})=0,\;  \text {in the sense of distributions}, \\ \nonumber
 &\nabla \cdot \mathbf{ u} =0,
  \mathbf{ u}(0)=\mathbf{ u}_0,\; \mathbf{ u}_{\Gamma}=0, \rho(0)=\rho_0,\\ \nonumber
 \end{align}
with
$\sqrt{\rho}\mathbf{u}\in L_{\infty}((0,T;L_2(\Omega)); \;\mathbf{u}\in L_p(0,T;\mathrm{V}^1_p\cap \mathrm{K}),\;
\mathbf{v}\in L_p(0,T;\mathbf{V}^1_p\cap \mathrm{K}),$

\noindent
$0\leq \rho \leq M,$
$\partial_t\mathbf{v}\in L_{p'}(0,T;L_{p'}(\Omega))$ .
\end{definition}
Briefly, we  deduce $(\ref{eq:v2})_1$ from  (\ref{eq:v1}).

Multiplying $(\ref{eq:v1})_3$ by $\mathbf{u}\cdot(\mathbf{v}-\mathbf{u})$,  integrating over $\Omega$ and
 adding the  result to
$(\ref{eq:v1} )_1$, integration in  $t$ gives
\begin{align}
 &\int_0^T((\partial_t(\rho\mathbf{u}(t)),\mathbf{v}(t)-\mathbf{u}(t)) -
 (\rho \mathbf{u}(t)\otimes\mathbf{u}(t),\nabla(\mathbf{v}(t)-\mathbf{u}(t))) -\label{eq:v3}
 \\ \nonumber
& (\nabla \cdot\mathbf{ T}\mathbf{u}(t) ,\mathbf{v}(t)-\mathbf{u}(t)))dt \geq
\int_0^T(\rho \mathbf{f}(t),\mathbf{v}(t)-\mathbf{u}(t))dt, \\ \nonumber
\end{align}
 and,  easily, follows
 \begin{align}
 &\int_0^T ((\partial_t(\rho\mathbf{v}(t)),\mathbf{v}(t)-\mathbf{u}(t)) -
 (\rho \mathbf{u}(t)\otimes\mathbf{u}(t),\nabla(\mathbf{v}(t)-\mathbf{u}(t)))-\label{eq:v4}
 \\ \nonumber
 & (\nabla \cdot\mathrm{T}\mathbf{u}(t),\mathbf{v}(t)-\mathbf{u}(t)) -
 (\rho \mathbf{f}(t),\mathbf{v}(t)-\mathbf{u}(t)))dt\geq\\ \nonumber
  &\int_0^T(\partial_t(\rho(\mathbf{v(t)-u(t)}) ),\mathbf{v}(t)-\mathbf{u}(t))dt.  \\ \nonumber
\end{align}
Notice
$$\int_0^T(\mathbf{v}(t)\partial_t\rho(t), \mathbf{v}(t)-\mathbf{u}(t))dt=
-\int_0^T(\nabla \cdot(\rho\mathbf{u}(t))\mathbf{v}(t),\mathbf{v}(t)-\mathbf{u}(t))dt,$$
and
$$\int_0^T((\mathbf{v(t)-u(t)})\partial_t \rho(t), \mathbf{v}(t)-\mathbf{u}(t))dt= $$
$$-\int_0^T(\nabla \cdot(\rho\mathbf{u}(t))(\mathbf{v}(t)-\mathbf{u}(t)),(\mathbf{v}(t)-\mathbf{u}(t))dt.$$
 In  addition,

 \begin{align*}
 & - (\nabla\cdot (\rho\mathbf{u}(t))\mathbf{v}(t),\mathbf{v}(t)-\mathbf{u}(t))+
 \frac{1}{2}(\nabla\cdot (\rho\mathbf{u}(t))(\mathbf{v}(t)-\mathbf{u}(t)),\mathbf{v}(t)-\mathbf{u}(t))-\\
 &(\rho\mathbf{u}(t)\cdot \nabla (\mathbf{v}(t)-\mathbf{u}(t)), \mathbf{u}(t))=
 (\nabla\cdot (\rho\mathbf{u}(t))\mathbf{v}(t),\mathbf{v}(t)-\mathbf{u}(t))-\\
 & (\rho\mathbf{u}(t)\cdot
 \nabla(\mathbf{v}(t)-\mathbf{u}(t)), \mathbf{v}(t))=
 (\rho \mathbf{u}(t)\cdot \nabla \mathbf{v}(t), \mathbf{v}(t)-\mathbf{u}(t)), \\
 \end{align*}
thus
$(\ref{eq:v1})_1$  assumes the following form
\begin{align*}
 &\int_0^T((\rho(t)\partial_t\mathbf{v}(t),\mathbf{v}(t)-\mathbf{u}(t)) +
 (\rho(t)\mathbf{u}(t)\cdot \nabla \mathbf{v}(t),\mathbf{v}(t)-\mathbf{u}(t))-\\
 & (\nabla \cdot\mathbf{T}\mathbf{u}(t),\mathbf{v}(t)-\mathbf{u}(t)) -
 (\rho \mathbf{ f}(t),\mathbf{v}(t)-\mathbf{u}(t)))dt\geq\\
&\frac{1}{2}(|\sqrt{\rho(T)}(\mathbf{v}(T)-\mathbf{u}(T))|^2 -|\sqrt{\rho(0)}(\mathbf{v}(0)-\mathbf{u}(0))|^2\geq
 -|\sqrt{\rho(0)}(\mathbf{v}(0)-\mathbf{u}(0))|^2.\\
\end{align*}

We  notice  that in  the  last   right hand side the  term $({\rho(T)}(\mathbf{v}(t)-\mathbf{u}(T)),\mathbf{v}(t)-\mathbf{u}(T))$  is  dropped.
The  reason is  clear.  We  are  not  able  to  prove  that,  for  Galerkin  approximations  $\rho^m,\mathbf{u^m}$,
 $\{({\rho^m}\mathbf{u}^m(t),\mathbf{u}^m(t))\}$ is  a compact  set  in  $C(0,T)$ (see Theorem 3).

\begin{theorem}
Assume the following hypotheses:
$$\mathbf{f}\in L_2(Q_T ), \;\mathbf{ u}_0\in L_p(\Omega)\cap\mathrm{K},$$
$$\rho_0\in L_{\infty}(\Omega),\;0\leq \rho_0\leq  M <\infty,  p\geq \frac{11}{5}.$$
Then, there exists a weak solution $(\mathbf{u},\rho)$ of the problem (\ref{eq:v2}).
\end{theorem}

\begin{proof}
To  study  the  existence of  weak  solutions to problem (\ref{eq:v2})
we use  the  compactness method  and  penalization  argument. First, we introduce  the  penalty operator  $\beta(\cdot)$ related  to $\mathrm{K}$.

\noindent
Let   $\mathrm{P}_{\mathrm{K}}:\mathrm{V}\rightarrow \mathrm{K}$  be the  projection  operator on the  convex  $\mathrm{K} \in L^2(\Omega)$.

    \noindent
Clearly, $\mathrm{P}_{\mathrm{K}}\mathrm{V}\subset\mathrm{K}$  and $(\mathbf{v}- \mathrm{P}_{\mathrm{K}}\mathbf{v}, \mathbf{u})\geq 0$ $\forall \mathbf{u}\in \mathrm{K}.$

  We  define $\beta$  as
$$  \beta: \mathrm{V}\rightarrow  (\mathrm{I}-\mathrm{P}_{\mathrm{K}})\mathrm{V},$$
and
$$ \beta(\mathbf{v})=\mathbf{v}-\mathrm{P}_{\mathrm{K}}\mathbf{v},  \;\forall \mathbf{v}\in\mathrm{V}.$$
$\beta$  is  a monotone and hemi-continuous operator;  moreover, the  following  relations  hold  true:
\begin{align*}
& (\beta(\mathbf{v}), \mathrm{P}_{\mathrm{K}}\mathbf{v})\geq 0,\\ \nonumber
& (\beta(\mathbf{v}), \mathrm{P}_{\mathrm{K}}\mathbf{v}-\mathbf{z})\geq 0, \;  \forall  \mathbf{z}\in\mathrm{K}.
\\ \nonumber
\end{align*}

\subsection{Semi-Faedo-Galerkin  approximation}

We  choose  a basis  as  in  section 3, and we  treat  the  system (\ref{eq:v2}) including the  penalty  term,  namely

\begin{align}
&\mathbf{ u}^m = \sum_{i=1}^mc^m_i(t)\mathbf{ w}_i(x)\in C(0,T;\mathrm{W}_{(m)})  ,\label{eq:v5} \\ \nonumber
&\rho^m \partial_t \mathbf{ u}^m + \rho^m \mathbf{ u}^m\cdot\nabla \mathbf{ u}^m -
 \nabla \cdot\mathbf{ T}\mathbf{ u}^m + \\ \nonumber
&m\beta( \mathbf{ u}^m)+ \nabla \pi^m-\rho^m \mathbf{f}=0, \quad i=1,...,m, \\ \nonumber
&\partial_t \rho^m +\mathbf{ u^m}\cdot \nabla \rho^m=0,\\ \nonumber
&\mathbf{ u}^m(0)= \mathbf{ u}_0^m, \quad \mathbf{u}_{\Gamma}=0,
\;\rho^m(0)= \rho^m_0.\\ \nonumber
\end{align}

The  local existence of  (\ref{eq:v5})  is  proved in  subsection 3.3.
The  global existence is  a  consequence  of  the estimates in subsection 3.5  and  of
\begin{align}
&{m}\int_0^T(\beta(\mathbf{ u^m}(t)),\mathbf{u^m}(t))dt= \label{eq:v6}\\ \nonumber
&{m}\int_0^T(\beta(\mathbf{u^m}(t)),\beta(\mathbf{u^m}(t))dt+
{m}\int_0^T(\beta(\mathbf{u^m}(t)),\mathrm{P}_{\mathrm{K}}\mathbf{u^m}(t))dt\leq c .\\ \nonumber
\end{align}

Moreover,   follows
$$\beta(\mathbf{u}^m) \rightarrow 0 \; \text { strongly  in}\;  L^2(Q_T).$$

Now,  we consider  the    estimate  that  permits  to apply the  compactness method  to variational inequalities.
\subsection{Time-translation estimate}

We   use   the  time-translation estimate instead    of    the  time-derivative
 estimate, in other words, we appeal to  the $L_p$- Ascoli-Arzel$\grave{a}$ theorem.

 Indeed,    we  will  prove
\begin{align}
&\int_0^{T-h}|\sqrt{\rho^m(t+h)}(\mathbf{u}^m(t+h)-\mathbf{u}^m(t))|_2^2dt\leq  c h^{\alpha},
\label{eq:v7}\\  \nonumber
\end{align}
$\forall h \in\mathbb{R}$  with $0<h<T$ and  $\alpha >0$.

Let $\mathbf{u}_h^m(t)  =\frac{1}{h}\int_{t-h}^{t}\mathbf{u}^m(s)ds$  be a  test  function.
   Multiplying $(\ref{eq:v5})_2$  by  $\mathbf{u}_h^m(t) $ and,  after  integration, we get
\begin{align}
&\int_h^T((\partial_t\rho^m \mathbf{ u}^m(t),\mathbf{u}_h^m(t)) +
((\rho^m \mathbf{ u}^m(t),\mathbf{u}^m(t)\cdot\nabla\mathbf{u}_h^m(t))+\label{eq;v8}\\ \nonumber
& m(\beta(\mathbf{u}^m(t)),\mathbf{u}^m_h(t))-(\nabla \cdot\mathbf{T}\mathbf{u}^m(t),\mathbf{u}_h^m(t)) -
 (\rho^m(t)\mathbf{f}(t),\mathbf{u}_h^m(t)))dt=\\ \nonumber
 &I_1+I_2+I_3+I_4+I_5=0. \\ \nonumber
\end{align}
Now,  we  estimate $I_i$ with  $i=1,2,3,4,5$.
\begin{align*}
   &I_1=\int_h^T(\partial_t\rho^m\mathbf{ u}^m(t),\mathbf{u}_h^m(t)) dt=
   (\rho^m(T) \mathrm{ u}^m(T),\mathbf{u}_h(T))-\\ \nonumber
   &(\rho^m(h) \mathbf{ u}^m(h),\mathbf{u}_h(h))-
   \frac{1}{h}\int_h^T(\rho^m(t) \mathbf{ u}^m(t),\mathbf{u}^m(t)- \mathbf{ u}^m(t-h))dt\leq\\ \nonumber
   &c\frac{1}{\sqrt{h}}(\int_0^T|\mathbf{ u}^m(t)|^2_2dt)^{1/2}-
   \frac{1}{h}\int_h^T(\rho^m(t) \mathbf{ u}^m(t),\mathbf{u}^m(t)- \mathbf{ u}^m(t-h))dt.\\  \nonumber
\end{align*}

 We analyze the  last  term  in  $I_1$.
\begin{align*}
&-\frac{1}{h}\int_h^T(\rho^m(t) \mathbf{ u}^m(t),\mathbf{u}^m(t)- \mathbf{ u}^m(t-h))dt=\\
&-\frac{1}{h}\int_h^T(\rho^m(t) \mathbf{ u}^m(t),\mathbf{ u}^m(t))dt +\frac{1}{h}\int_h^T(\rho^m(t) \mathbf{ u}^m(t),\mathbf{ u}^m(t-h))dt=\\
&-\frac{1}{2h}\int_h^T|\sqrt{\rho^m(t) }\mathbf{ u}^m(t)|_2^2dt + \frac{1}{2h}\int_h^T|\sqrt{\rho^m(t) }\mathbf{ u}^m(t-h)|_2^2dt- \\
&\frac{1}{2h}\int_h^T|\sqrt{\rho^m(t) }(\mathbf{ u}^m(t)-\mathbf{ u}^m(t-h))|_2^2dt=\\
&-\frac{1}{2h}\int_h^T|\sqrt{\rho^m(t) }\mathbf{ u}^m(t)|_2^2dt + \frac{1}{2h}\int_h^T|\sqrt{\rho^m(t-h) }\mathbf{ u}^m(t-h)|_2^2dt+\\
& \frac{1}{2h}\int_h^T((\rho^m(t)-\rho^m(t-h))\mathbf{u}^m(t-h) ,\mathbf{ u}^m(t-h))dt-\\
&\frac{1}{2h}\int_h^T|\sqrt{\rho^m(t) }(\mathbf{ u}^m(t)-\mathbf{ u}^m(t-h))|_2^2dt.\\
\end{align*}
Thanks  to  the  continuity  equation, we  get
\begin{align*}
&|\frac{1}{2h}\int_h^T((\rho^m(t)-\rho^m(t-h))\mathbf{u}^m(t-h) ,\mathbf{ u}^m(t-h))dt|=\\
&|\frac{1}{h}\int_h^T(\int_{t-h}^t\rho(s)\mathbf{u}^m(s)ds\cdot\nabla\mathbf{u}^m(t-h) ,\mathbf{ u}^m(t-h))dt|\leq\\
&c\frac{1}{\sqrt{h}}(\int_h^T\|\mathbf{u}^m(t)\|^2dt)^{3/2}.\\
\end{align*}
 In  conclusion,
 \begin{align*}
 &I_1=\int_h^T(\partial_t\rho^m\mathbf{u}^m(t),\mathbf{u}_h^m(t)) dt\leq c_1 +\frac{c_2}{\sqrt{h}}-\\
&\frac{1}{h}\int_h^T|\sqrt{\rho^m(t)}(\mathbf{u}^m(t- h )-\mathbf{u}^m(t))|_2^2dt.\\
\end{align*}
As  before,
\begin{align*}
&I_2=\int_h^T|(\rho^m \mathbf{ u}^m(t) \mathbf{u}^m(t),\nabla\mathbf{u}_h^m(t))|dt\leq\\ \nonumber
&c\int_h^T|\mathbf{u}^m(t)|_4|\mathbf{u}^m(t)|_4\|\mathbf{u}_h^m(t)\|dt\leq \\ \nonumber
&\frac{c}{\sqrt{h}} \int_h^T\|\mathbf{ u}^m(t)\|^2_2dt(\int_h^T\|\mathbf{u}_h^m(t)dt)\|^2)^{1/2}\leq \frac{c}{\sqrt{h}};\\ \nonumber
&I_3=\int_h^T(\nabla \cdot\mathbf{ T}\mathbf{u}^m(t), \mathbf{u}_h^m(t))dt
\leq \frac{c}{{h^{\gamma}}} \int_0^T\|\mathbf{u}^m(t)\|_p^pdt,\; \gamma >0.\\ \nonumber
\end{align*}
Now, we estimate the  crucial term $I_4$.
\begin{align*}
&I_4=m\int_h^T(\mathbf{u}^m(t)-\mathrm{P}_{\mathrm{K}}\mathbf{u}^m(t),\mathbf{u}_h^m(t))dt=\\ \nonumber
&m\int_h^T(\mathbf{u}^m(t)-\mathrm{P}_{\mathrm{K}}\mathbf{u}^m(t),\frac{1}{h}\int_{t-h}^t(\mathbf{u}^m(s)-
\mathrm{P}_{\mathrm{K}}\mathbf{u}^m(s)ds)dt+\\ \nonumber
&m\int_h^T(\mathbf{u}^m(t)-\mathrm{P}_{\mathrm{K}}\mathbf{u}^m(t),\frac{1}{h}\int_{t-h}^t(\mathrm{P}_{\mathrm{K}}\mathbf{u}^m(s)ds-
\mathrm{P}_{\mathrm{K}}\mathbf{u}^m(t))dt+\\ \nonumber
&m\int_h^T(\mathbf{u}^m(t)-\mathrm{P}_{\mathrm{K}}\mathbf{u}^m(t),\mathrm{P}_{\mathrm{K}}\mathbf{u}^m(t))dt.\\ \nonumber
\end{align*}
Since $\frac{1}{h}\int_{t-h}^t\mathrm{P}_{\mathrm{K}}\mathbf{u}^m(s)ds\in \mathrm{K}$ and   bearing   in
mind  the   properties   of  $\beta(\mathbf{u}^m )$,
it follows that, in  the right hand  side,
the  second term  is  negative and  the last  term is bounded.
Finally,
$$m\int_h^T(\mathbf{u}^m(t)-\mathrm{P}_{\mathrm{K}}\mathbf{u}^m(t),\frac{1}{h}\int_{t-h}^t(\mathbf{u}^m(s)-
\mathrm{P}_{\mathrm{K}}\mathbf{u}^m(s)ds)dt\leq$$
$$ \frac{m}{\sqrt{h}}\int_h^T|\mathbf{u}^m(t)-
\mathrm{P}_{\mathrm{K}}\mathbf{u}^m(t)
|\int^t_{t-h}|\mathbf{u}^m(s)-\mathrm{P}_{\mathrm{K}}\mathbf{u}^m(s)|^2_2ds)^{1/2} dt\leq$$
$$
 c\frac{m}{\sqrt{h}}\int_0^T|\beta(\mathbf{u}^m(t)|^2_2dt\leq c\frac{1}{\sqrt{h}}.$$

In  conclusion,

$$I_4 =m\int_h^T(\mathbf{u}^m(t)-\mathrm{P}_{\mathrm{K}}\mathbf{u}^m(t),\mathbf{u}_h^m(t))dt\leq c+
 \frac{c}{\sqrt{h}}.$$
Finally,
$$ I_5=\int_h^T(\rho^m \mathbf{f}(t),\mathbf{u}_h^m(t))dt\leq\int_0^T|\rho^m \mathbf{f}(t)|_2|\mathbf{u}_h^m(t)|_2dt\leq
c  \frac{1}{\sqrt{h}}.$$
Thanks to the  previous estimates, it follows
\begin{align}
&  \frac{1}{h}\int_0^{T-h}|\sqrt{\rho^m(t+h)}(\mathbf{u}^m(t+h)-\mathbf{u}^m(t))|_2^2dt\leq
 c h^{\alpha}.
\label{eq:v10}\\ \nonumber
\end{align}
(\ref{eq:v7}) is proved.

Since $\rho^m\in L_{\infty}(Q_T)$    uniformly   in   $m$,  thus
 \begin{align}
&  \int_0^{T-h}|\rho^m(t+h)(\mathbf{u}^m(t+h)-\mathbf{u}^m(t) )|_2^2dt\leq c h^{\alpha},
\label{eq:v11} \\ \nonumber
\end{align}
with $\alpha <1$.

Moreover, the  continuity  equation and  the energy  estimate imply
\begin{align}
& \|\rho^m(t+h)-\rho^m(t)\|_{H^{-1}(\Omega)}\leq c h.
\label{eq:v12}\\  \nonumber
\end{align}
Since $\mathrm{H}^{-1}(\Omega)\times \mathrm{H}^1_0(\Omega)\hookrightarrow  \mathrm{ W}^{-1}_r(\Omega)$  with   $r<3/2$,
plainly  follows
 \begin{align}
& \int_0^{T-h}\|\rho^m(t+h)\mathbf{u}^m(t+h)-\rho(t)\mathbf{u}^m(t)\|_{\mathrm{W}^{-1}_r(\Omega)}\leq c h^{\alpha},
\label{eq:v13}\\ \nonumber
\end{align}
so $\{\rho^m\mathbf{u^m}\}$ is   a  compact  set   in  $L_2(0,T;\mathrm{H}^{-1}(\Omega))$.
In  virtue of  the  estimates  in  subsection 3.5, we  have
$$\rho^m\mathbf{u}^m  \rightarrow \rho\mathbf{u}
\;  \text{strongly   in} \;  L_2(0,T;\mathbf{H}^{-1}(\Omega)).$$

  But  $\mathbf{u}^m\rightharpoonup \mathbf{u} $ weakly    in   $\mathrm{L}_p(0,T;V^1_p(\Omega))$,  it  follows
  $$\rho^mu^m_iu^m_j\rightarrow \rho u_iu_j$$
  in  the  sense  of  distributions.

  As  in  section 3,  it  remains  to  prove   $$\mathbf{\chi} =  P\mathbf{T}_p\mathbf{u}.$$
We  set  $c_{m}(s)=m\int_0^s (\beta(\mathbf{u}^m(t)),\mathbf{u}^m(t))dt \in \mathbb{R}$, thanks  to  the  Bolzano-Weirstrass
  theorem, $c_m(s)\rightarrow  c(s) \geq 0$  as $m\rightarrow \infty$ for  fixed $s$.
We use  the  relation
 \begin{align}
  &\frac{1}{2}|\sqrt{\rho}\mathbf{u}(s)|^2 +  c(s) +\int_0^s(\mathbf{\chi}(t),\mathbf{u}(t))dt=
 \int_0^s(\rho \mathbf{f}(t),\mathbf{u}(t))dt+\frac{1}{2}|\sqrt{\rho}\mathbf{u}_0|^2.
 \label{eq:v14}\\ \nonumber
 \end{align}

Now, we  introduce,  for  $\mathbf{\phi}\in L_p(0,T;V^1_p(\Omega))$,
$$X^{s}_m=\int_{0}^s(\mathbf{T}_p\mathbf{u^m}(t)-\mathbf{T}_p\mathbf{\phi}(t),\mathbf{u^m}(t)-\mathbf{\phi}(t))dt+
\frac{1}{2}|\sqrt{\rho^m}\mathbf{u}^m(s)|^2_2+c_m(s),$$
a.e. in $(0,T)$.x

Thanks to the  monotony of  the  operator  $\mathbf{T}_p$,  we  derive
\begin{align}
 &\liminf_{m  \rightarrow \infty}X^s_m\geq c(s) + \frac{1}{2}|\sqrt{\rho}\mathbf{u}(s)|^2_2.
\label{eq:v15}\\ \nonumber
\end{align}
 $(\ref{eq:v5})_2$ implies, for $m\rightarrow  +\infty$,
\begin{align*}
&X^{s}_m=\int_0^s(\rho(t)\mathbf{f}(t),\mathbf{u}^m(t))dt
+\frac{1}{2}|\sqrt{\rho^m}\mathbf{u}^m_0|^2_2- \\ \nonumber
&\int_0^s(\mathbf{T}_p\mathbf{u}^m(t),\mathbf{\phi}(t))dt-
\int_0^s(\mathbf{T}_p\mathbf{\phi}(t),\mathbf{u}^m(t)-\mathbf{\phi}(t))dt\rightarrow {X}^s,
\\ \nonumber
\end{align*}
with
$${X}^s= \int_{0}^s(\rho\mathbf{f}(t),\mathbf{u}(t))dt +
\frac{1}{2}|\sqrt{\rho}\mathbf{u}_0|^2_2- $$
$$\int_0^s(\mathbf{\chi}(t),\phi(t))dt-
\int_{t_0}^s(\mathrm{T}_p\mathbf{\phi}(t),\mathbf{u}(t)-\mathbf{\phi}(t))dt.$$
Consequently, with  (\ref{eq:v14}), we  get
$$\int_0^s(\rho\mathbf{f}(t),\mathbf{u}(t))dt +\frac{1}{2}|\sqrt{\rho}\mathbf{u}_0|^2_2-$$
$$\int_0^s(\mathbf{\chi}(t),\phi(t))dt-
\int_0^s(\mathrm{T}_p\mathbf{\phi}(t),\mathbf{u}(t)-\mathbf{\phi}(t))dt\geq c(s)+\frac{1}{2}|\sqrt{\rho}\mathbf{u}(s)|^2_2,$$

and, finally, (\ref{eq:v14})
 implies
$$\int_{0}^{s}(\mathbf{\chi}(t)-\mathbf{T}_p\mathbf{\phi}(t),\mathbf{u}(t)-\mathbf{\phi}(t))dt\geq 0, $$
 for almost  $s$.
Using the  classical  procedure, i.e. using  $\mathbf{\phi} =  \mathbf{u} +\lambda \mathbf{\psi}$  with
 arbitrary $\lambda\in \mathbb{R}$ and $\mathbf{\psi} \in L_p(0,T; V_p) $ in  the  last relation follows
$$\mathbf{\chi}=P\mathrm{T}_p\mathbf{u}.$$

\noindent
  At   this  stage  all   terms  in (\ref{eq:v5})    are convergent.  We show   that  the   functions    $(\rho,\mathbf{u})$  satisfy (\ref{eq:v2}).

First, since $\mathbf{T}_p$ is monotone,
$$(\mathbf{T}_p\mathbf{u}^m(t) , \mathbf{u}^m(t))=(\mathbf{T}_p \mathbf{u}^m(t) ),\mathbf{u}^{ m}(t) -\mathbf{u}(t))+ (\mathbf{T}_p\mathbf{u}^m(t), \mathbf{u}(t))\geq $$
$$(\mathbf{T}_p \mathbf{u}(t) ),\mathbf{u}^{ m}(t) -\mathbf{u}(t))+ (\mathbf{T}_p\mathbf{u}^m(t), \mathbf{u}(t)),$$
thus
$$\liminf_{m\rightarrow \infty}(\mathbf{T}_p \mathbf{u}^m(t),\mathbf{u}^{ m}(t) )\geq
(\mathbf{T}_p \mathbf{u}(t),\mathbf{u}(t) ).$$

 Therefore,
 $$\liminf_{m\rightarrow \infty}(\mathbf{T}_p \mathbf{u}^m(t),\mathbf{u}^{ m}(t) -\mathbf{v}(t))\geq (\mathbf{T}_p \mathbf{u}(t),\mathbf{u}(t) -\mathbf{v}(t)),$$
 for $\forall \mathbf{v}\in \mathrm{K}$.

Let   $\mathbf{v}(t)$  be   an   arbitrary regular function   such   that  $  \mathbf{v}(t)\in V^1_p(\Omega)\cap \mathrm{K}$   for  all  $t\in(0,T)$  and $\mathbf{v}(T)=0$.  Let $\mathbf{v^m(t)}\in \mathrm{W_{(m)}}$   such that

$$ \mathbf{v^m(t)} \rightarrow \mathbf{v(t)}\;  \text{strongly in} \;  V^1_p(\Omega),$$

\noindent
with   $\mathbf{v}^m(T)=0$ and   $\mathbf{v}^m(t)\in K$   for    all  $m\geq  m_0$  (  for  some finite  $ m_0$).

\noindent
Let     $ m   >\bar {m} >m_0$,  thus
$$(\beta(\mathbf{u}^m(t)),\mathbf{v}^{\bar  m}(t)-  \mathbf{u}^m(t))\leq 0.$$
By  taking   $ \mathbf{v}^{\bar  m}-\mathbf{u}^{ m}$   as  a  test   function  in the   momentum  equation   $(\ref{eq:v5})_2$, and   taking  into  account the last  inequality,  we   get
\begin{align}
&\int_0^T((\partial_t\rho^m\mathbf{ u}^m(t) +\nabla\cdot ( \rho^m \mathbf{ u}^m (t)\otimes\mathbf{ u}^m)(t),\mathbf{v}^{\bar m}(t)-\mathbf{u}^m(t))+ \label{eq:v16}\\ \nonumber
&(\mathbf{T}_p\mathbf{u}^m(t), \mathbf{v}^{\bar m}(t) -\mathbf{u}^m(t)))dt\geq
\int_0^T(\rho^m \mathbf{f}(t), \mathbf{v}^{\bar m}(t) -\mathbf{u}^m(t)))dt.\\ \nonumber
\end{align}
Using  the  procedure  to  obtain $(\ref{eq:v2})_1$, we  get
\begin{align}
&\int_0^T((\rho^m(t)\partial_t\mathbf{ v}^{\bar m}(t),\mathbf{v}^{\bar m}(t)-\mathbf{u}^m(t)) +
(\mathbf{T}_p\mathbf{u}^m(t) , \mathbf{v}^{\bar m}(t) -\mathbf{u}^m(t))+\label{eq:v17} \\ \nonumber
 & (\rho^m \mathbf{ u}^m(t)\cdot\nabla \mathbf{ u}^m(t),\mathbf{v}^{\bar m}(t)-\mathbf{u}^m(t))-
(\rho^m \mathbf{f}(t), \mathbf{v}^{\bar m}(t) -\mathbf{u}^m(t)))dt\geq \\ \nonumber
&-|\sqrt{\rho^m_0}\mathbf{u}^m_0|_2^2.\\ \nonumber
\end{align}
  Passing  to   the   limit $m\rightarrow +\infty$  in (\ref{eq:v17}),  it  is  plain  that  $(\mathbf{u}, \rho)$   is  a weak solution of the   problem
(\ref{eq:v2}).

  Theorem 5.1 is completely proved.

\end{proof}

Remark: The theorem   continues   to  hold  for  convex  $\mathrm{K}(t)$   which   depends  monotonically on  $t$, i.e.  $\mathrm{K}(t_1)  \subseteq \mathrm{K}(t_2))$  for
$t_1\leq t_2$.

The   proof   of  the   theorem  is the  same   once  we perform  the  change $\mathrm{P}_{\mathrm{K}}\rightarrow  \mathrm{P}_{\mathrm{K(t)}}$
 and   observing  that  $\frac{1}{h}\int_{t-h}^t\mathrm{P}_{\mathrm{K(s)}}\mathbf{u}^m(s)ds\in \mathrm{K(t)}$ with  $h\geq 0$.

\section{Local  Well-posedness problem}

We  recall that the  quasi-linear  differential operator  $\mathbf{T}_p(\mathbf{u},\mathbf{D})$ is defined  as
$$\mathrm{T}_p(\mathbf{u},\mathbf{D}):=\sum_{k,l=1}t_{i,j}^{k,l}(x,t) \partial_k\partial_l,$$

\noindent
in section  3.

\medskip
We  use  the  following notations.

\noindent
Set  $X_0 :=L_q(Q_T)$, $X_2 := W^2 _q(\Omega)$ and $Y_q :=W_q^{2-2/q}(\Omega)$
equipped with    the standard  norms that  we  denote $|\cdot|_q$, $|||\cdot|||_2$ and $\|\cdot\|_{(q)}$, respectively;  assume  $q >n+2$.

$\mathrm{T}_p(\mathbf{u})$  satisfies                :
$$\mathrm{T}_p: Y_q\rightarrow \mathcal{B} (X _2,X_0)$$
is   continuous ($\mathcal{B}$   stands  for  bounded  operator set)  and
$$  |\mathrm{T}_p(\mathbf{v})\mathbf{u}-\mathrm{T}_p(\bar{\mathbf{v}})\mathbf{u}|_q
\leq
c (\|\bar{\mathbf{ v}}\|_{(q)})\|\mathbf{v}-\bar{\mathbf{v}}\|_{(q)}|||\mathbf{u}|||_2,$$

 if  $\mathbf{ v},\bar{\mathbf {v}}\in  Y_q $  and  $  \mathbf{u}\in  X_2$,  $q> n+2$  and
 $c(\cdot)$  is  continuous  positive  function.

Furthermore,
 $$\mathcal{Z}:=W^1_q(0,T;L_q(\Omega))\cap L_q(0,T;W^2_q(\Omega))\hookrightarrow C(0,T; W_q^{2-2/q}(\Omega)) \hookrightarrow $$

 $C((0,T];C^1(\Omega)).$

 The   embedding  constant depends   on  $T$ and
 can blow  up as  $T\rightarrow 0_+$, in  general, if the  initial datum  is  different  from zero.
 In general,   $ W_q^{2-2/q}(\Omega)$   is  considered  as   a  time-trace  space.
Now,  we  consider
the well-posedness    of the following problem,
\begin{align}
&\rho \partial_t \mathbf{u} + \mathbf{T}_p(\mathbf{u})\mathbf{u}+\rho \mathbf{u}\cdot\nabla \mathbf{u}+
\nabla \pi =\rho \mathbf{f},\label{eq:s4}\\ \nonumber
&\partial_t\rho +\mathbf{u} \cdot  \nabla \rho= 0,\\ \nonumber
&\nabla  \cdot  \mathbf{u} =0, \;  \mathbf{u}(0)=0,  \; \rho(0)=\rho_0,
 \;\mathbf{u}=0  \; on\;  \Gamma.  \\ \nonumber
\end{align}

The  main result  of the  section   is  the  following,

\medskip
\noindent
\begin{theorem}   Let  $\Omega \subset \mathbb{R}^3$ be  a  domain  with compact  boundary  $\Gamma $
of  class $C^3$; let  $5 < q<\infty$,  and  assume  that  $\mu\in C^2(\mathbb{R}_+)$ is  such  that
$\mu(s) \geq c>0$ for  every  $s\geq 0$.
Then,  for each  $\mathbf{u}_0\in  W^{2-2/q}_q(\Omega)\cap \mathrm{V}$  and $\rho_0 \in W^1_q(\Omega)$, $\mathbf{f}\in L_q(Q_T),$
and $0 <m \leq \rho_0 \leq 1$, there exists  a  $\bar{T}>0$  such  that
there is   a  unique  solution  $(\mathbf{u},\rho,\pi)$ of (\ref {eq:s4})   on
the  time interval  $[0,\bar{T}]$  such  that
$$\mathrm{u}\in W^1_q(0,\bar{T};V^0_q(\Omega))\cap  L_q(0,\bar{T},W^2_q(\Omega)\cap V),\;\pi\in L_q(0,\bar{T};W^1_q(\Omega))/\mathbb{R}),$$
\quad $\rho\in L_{\infty}(Q_T)\cap L_{\infty}(0,T;W^1_q(\Omega)),\;
\partial_t\rho \in L_{\infty}(0,T;L_q(\Omega)).$
\end{theorem}

We  prove  theorem 6.1 assuming $\mathbf{u}_0=0$ for  simplicity  of  exposition.  The  non-homogeneous  case  is treated using  a translation.
Now, we  recall some  results concerning the solvability of the generalized
Stokes equations with variable coefficients. In \cite{rif3}, \cite{rif36},\cite{rif37} is considered the following problem
\begin{align}
&\partial_t \mathbf{u} +  \mathcal{A}(x,t,\frac{\partial}{\partial x}) \mathbf{u}(t)+\nabla \pi =  \mathbf{f},\label{eq:s0}\\ \nonumber
&\nabla  \cdot   \mathbf{u} =0, \;   \mathbf{u}(0)= \mathbf{u}_0,  \; \mathbf{u}=0  \; on\;  \Gamma, \\ \nonumber
  \end{align}

\noindent
where  $\mathcal{A}(x,t,D)=\sum_{k,l=1}^3\mathbf{A}^{k,l}(x,t)\partial_k\partial_l$ is a matrix   elliptic-type differential  operator and  $\mathbf{A}^{k,l}(x,t)$ are  regular  functions
 with real coefficients depending on $x,t$.
 Assuming the coefficients of the operator $\mathcal{A}$ are bounded, and the coefficients
of the operator $\mathbf{A}_0$ (principal  part  of $\mathcal{A}(x,t,D)$) are continuous with respect to $(x,t)$  and belong to $ W^1_r(\Omega)$, $ 1/r < 1/n+
min(1/q; (q-1)/q;1/n)$, for all $t$,  $\mathbf{f}\in  L_q(Q_T)$ , $\mathbf{u}_0 \in W^{2-2/q}_q(\Omega)$,
the  following theorem is  proved.

\begin{theorem}   Let $\Gamma\in  C^3$,  $\mathbf{f}\in L_q(\Omega)$ , $\mathbf{u}_0 \in W_q^{2-2/q} (\Omega)\cap V$.
 Then  there  exists  a   unique solution  $(\mathbf{u},\pi)$ of (\ref{eq:s0})  such  that

 $$u\in W^1_q(0,T;V)\cap L_q(0,T;W^2_q(\Omega)\cap V),  \;\pi\in L_q(0,T;  W^1_q(\Omega)/\mathbb{R})$$

 such  that
 $$\|\mathbf{u}\|_{W^{2,1}_q(Q_T)}^q +\|\nabla \pi\|_{L_q(Q_T)}^q   \leq \bar{c}(\|\mathbf{f}\|^q_{L_q(Q_T)}+
 \|\mathbf{u}_0\|^q_{(q )}).$$
\end{theorem}

The method of the  proof  consists  of  frozen techniques, Schauder's  type  estimates and of the construction
of a "regularizator ".

If the  coefficients $\mathbf{A}^{k,l}(x,t)$  are  constant or  perturbation  of  a  constant, the  operator $\mathcal{A}$
 has  the  so-called  " maximal regularity "  property.

\medskip
Consequence  of the  previous theorem    is  the existence of  the following    quasi-linear problem
\begin{align}
&\partial_t \mathbf{u} +  \mathcal{A}(\mathbf{v}(t))\mathbf{u}(t)+\nabla \pi= \mathbf{f}(t),\label{eq:s1}\\ \nonumber
&\nabla\cdot \mathbf{u}=0,\;  \mathbf{u}(0)=\mathbf{u}_0, \;
\mathbf{u}=0  \; on\;  \Gamma. \\ \nonumber
 \end{align}

\noindent
with $\mathbf{v}$ and $\mathbf{u}$  belong  to  the  spaces occurring in   theorem 6.1.
\begin{proof} We  prove  theorem 6.1.
  Let us  consider  the set
  $$\mathrm{B}=\{\mathbf{\phi}|\mathbf{\phi}(0)=0; \;sup( \|\mathbf{\phi}\|_{ L_q(0,T;W^2_q(\Omega))},  \|\mathbf{\phi}\|_{W^1_q(0,T;Lq(\Omega))})\leq r\},$$
  with  $r\in\mathbb{R_+}$.

The  existence  proof of  a  solution
 of  (\ref{eq:s4}) is done   by  formulate  the  existence    as  a  fixed  point in  $\mathrm{B} $ of  the  map
 $\mathrm{G}:   \mathbf{v} \rightarrow \rho \rightarrow \mathbf{u}$  where  $\mathbf{u}$  and  $\rho$  are  solutions
 of  the following problem for arbitrary  $\mathbf{v}\in  \mathrm{B}$,
\begin{align}
& \partial_t \mathbf{u}(t) + \mathrm{T}_p( \mathbf{v}(t))\mathbf{u}(t)+\nabla \pi(t) = \label{eq:s5}\\ \nonumber
&(1-\rho(t))\partial_t \mathbf{v}(t)- \rho(t)\mathbf{v}(t)\cdot\nabla \mathbf{v}(t)+
 \rho(t)\mathbf{f}(t),
\\ \nonumber
&\partial_t\rho(t) + \mathbf{v}(t)\cdot  \nabla \rho(t)= 0, \\ \nonumber
&\nabla  \cdot  \mathbf{u}(t) =0, \;  \mathbf{u}(0)=0,  \; \rho(0)=\rho_0,  \\\ \nonumber
&0 < m \leq\rho_0\leq  1,\;
2\bar {c}(1-\rho_0)<1,\;
\mathbf{u}=0  \; on\; \Gamma. \\ \nonumber
\end{align}

Theorem 6.2, and  the  results  in  section 3    imply  the  existence and  uniqueness   of (\ref{eq:s5}).  Bearing  in   mind  that  the  operator  $\mathrm{T}_p$  is  strongly  elliptic  we  have

$$\|\mathbf{u}\|^q_{W^{2,1}_q(Q_T)} +\|\nabla \pi\|^q_{L_q(Q_T)} \leq \bar{c}(\|f\|^q_{L_q(Q_T)}+
|F(\mathbf{v},\rho)|^q_{L_q(Q_T)}).$$

\medskip
Here
$F(\mathbf{v}):=(1-\rho(t))\partial_t\mathbf{v}(t)-\rho(t) \mathbf{v}(t)\cdot \nabla  \mathbf{v}(t).$

\medskip
\noindent
Clearly, $\mathrm{B}$  is  a compact  set  in $L_2(Q_T)$ framework .  As   we  are  going  to  use  a  fixed  point  argument,
we  have  to  show $G(\mathrm{B})\subseteq \mathrm{B}$,  and  $G$  is  continuous  in   $L_2(Q_T)$, for  example.

According,   we  prove    $G(\mathrm{B})\subseteq \mathrm{B}$ for suitable  $\bar  T$.

 We  notice that
$$|\rho \mathbf{v}\cdot \nabla  \mathbf{v}|_{L_q(Q_T)}\leq c\|\mathbf{v}\|_{L_{\infty}(Q_T)}\|\mathbf{v}\|_{L_{\infty}(0,T;W^1_q(\Omega))}T^{1/q}.$$

Moreover,  thanks  to  the  assumptions,  we  get

$$|(1-\rho)\partial_t  \mathbf{v}|_q< \frac{1}{2\bar{c}}|\partial_t \mathbf{v}|_q.$$

In  conclusion,  for  suitable  $T:=\bar{T}$,  we  get
$$G\mathrm{B}\subseteq  \mathrm{B}.$$

\medskip
Now,  we  prove  the  continuity  of  $G$  in  $L_2(Q_{\bar{T}})$.

First,  we  observe  that if  $\{ \mathbf{v} ^n\}\subset \mathrm{B}$ there  exists  a  subsequence ( denoted  again $\{\mathbf{ v} ^n\}$)  such  that
as  $n\rightarrow +\infty$, $ \mathbf{v}^n\rightarrow \mathbf{ v}$ weakly in  $L_q(0, \bar{T}; W^2_q(\Omega))$,  weakly* in  $L_{\infty}(0, \bar{T};W^1_q(\Omega))$,
and  $\partial_t \mathbf{v}^n  \rightarrow   \partial_t  \mathbf{v} $ weakly  in $L_q(Q_{ {T}})$.

Let
$\{\mathbf{ v}^n\}\subset\mathrm{B}$ such  that $\mathbf{ v} ^n\rightarrow \mathbf{v} $  strongly  in  $L_2(Q_T)$.
Clearly, $\mathbf{v}\in  \mathrm{B}$ and Proposition 1   provides  strong  convergence  in the  intermediate spaces
we  need.

\medskip
Let  $\rho^n$  and  $\rho$ be  the  solutions  of
$$\partial_t \rho^n + \mathbf{v}^n \cdot\nabla \rho^n = 0,  \;  \rho^n(0)=\rho_0,$$
and
$$\partial_t \rho + \mathbf{v}\cdot \nabla\rho = 0,  \;  \rho(0)=\rho_0,$$
respectively.

The regularity of $\rho^n$  and $\rho$ is  proved in section  3.

Now,   $\tilde \rho^n = \rho^n-\rho$   satisfies
\begin{align}
&\partial_t\tilde \rho^n  + ( \mathbf{v}^n - \mathbf{v})\cdot \nabla\rho  + \mathbf{v}^n\cdot \nabla \tilde \rho^n=0,  \;  \tilde \rho^n(0)=0.
\label{eq:s6}\\ \nonumber
\end{align}
 Multiplying    (\ref{eq:s6} )  by $\tilde \rho^n$ and  after  integration  over  $Q_{{T}}$ we   have
$$|\tilde \rho^n(t)|_2^2\leq  e^{ct}\int_0^{t}|( \mathbf{v}^n- \mathbf{v})\cdot \nabla \rho|_2^2dt.$$

 This  implies    $\rho^n \rightarrow\rho$  strongly  in  $L_{\infty} (0,\bar{T};L_2(\Omega))$.
Thanks  to  $\tilde \rho^n\in L_{\infty}(Q_T)  $ it follows
$\rho^n \rightarrow\rho$  strongly  in  $L_{\infty} (0,T;L_r(\Omega))$ for  every  finite  $r>1$.

Now,  let  $\mathbf{u}^n$ and $\mathbf{u}$   be  the  solutions of
\begin{align}
& \partial_t \mathbf{u}^n +  \mathbf{T}_p(\mathbf{v}^n)\mathbf{u}^n + \nabla \pi^n=
 (1-\rho^n)\partial_t \mathbf{v}^n-\rho^n  \mathbf{v}^n\cdot \nabla \mathbf{v}^n+\rho^n f, \label{eq:s7}\\ \nonumber
&\partial_t\rho^n + \mathbf{v}^n\cdot  \nabla \rho^n= 0,\\ \nonumber
&\nabla  \cdot  \mathbf{u}^n =0, \;\mathbf{u}^n(0) =0, \; \rho^n(0)=\rho_0,  \;  \;\mathbf{u}^n=0  \;
on\;  \Gamma,\\ \nonumber
  \end{align}
and
\begin{align}
& \partial_t \mathbf{u} +   \mathbf{T}_p(\mathbf{v})\mathbf{u} + \nabla \pi=  (1-\rho)\partial_t \mathbf{v}-\rho
\mathbf{v}\cdot \nabla \mathbf{v} +\rho f,
 \label{eq:s8}\\ \nonumber
&\partial_t \rho + \mathbf{v} \cdot  \nabla \rho= 0,\\ \nonumber
&\nabla  \cdot  \mathbf{u} =0, \;  \mathbf{u}(0)=0,  \; \rho(0)=\rho_0,  \;
\mathbf{u}=0  \; on\;  \Gamma,\\ \nonumber
  \end{align}
respectively.

\noindent
Subtracting $(\ref{eq:s7})_1$ and  $(\ref{eq:s8})_1$
we  obtain
\begin{align}
& \partial_t (\mathbf{u}^n-\mathbf{u})+  (\mathbf{T}_p(\mathbf{v}^n) -
\mathbf{T}_p(\mathbf{v}))\mathbf{u}^n+\mathbf{T}_p(\mathbf{v})(\mathbf{u}^n-\mathbf{u})+ \nabla (\pi^n-\pi) =
\label{eq:s9}\\ \nonumber
&(1-\rho^n)\partial_t(\mathbf{v}^n-\mathbf{v})-
\tilde\rho^n\partial_t\mathbf{v}-\rho^n \mathbf{v}^n\cdot \nabla  \mathbf{v}^n +
\rho \mathbf{v}\cdot \nabla  \mathbf{v}+\tilde \rho \mathbf{f}.\\ \nonumber
 \end{align}
Denoting  $\mathbf{V}^n=\mathbf{v}^n-\mathbf{v}$ , $\mathbf{U}^n=\mathbf{u}^n-\mathbf{u}$,    we get

  \begin{align}
& \partial_t \mathbf{U}^n(t)+ \mathbf{T}_p(\mathbf{v})\mathbf{U}^n(t)+(\mathbf{T}_p(\mathbf{v}^n(t))-
 \mathbf{T}_p(\mathbf{v}(t)))\mathbf{u}^n(t)  +\nabla (\pi^n -\pi)=  \label{eq:s10}\\ \nonumber
 & (1-\rho^n(t))\partial_t \mathbf{V}^n(t)-
 \tilde  \rho^n(t)\partial_t \mathbf{v}(t)-\rho^n(t)\mathbf{V}^n(t)\cdot \nabla  \mathbf{v}^n(t) -\\ \nonumber
 &\tilde\rho^n(t) \mathbf{v}(t)\cdot\nabla\mathbf{v}^n(t)-
 \rho(t) \mathbf{v}(t)\cdot \nabla \mathbf{V}^n(t) +\tilde \rho^n(t) \mathbf{ f}(t).\\ \nonumber
\end{align}

  Multiplying  (\ref{eq:s10})   by  $\mathbf{U}^n(t)$  and  integrating over  $\Omega$
we  get
  \begin{align}
&d_t|\mathbf{U}^n(t)|^2_2 +\|\mathbf{U}^n(t)\|^2 \leq
 -\int_{\Omega}\tilde  \rho^n\mathbf{U}^n(t) \partial_t\mathbf{v}(t) dx +\label{eq:s11}\\ \nonumber
&\partial_t\int_{\Omega}(1-\rho^n(t))\mathbf{U}^n(t) \mathbf{V}^n(t)dx-
\int_{\Omega}\mathbf{U}^n(t)^n \mathbf{V}^n(t)\partial_t(1-\rho^n(t))dx -\\  \nonumber
&\int_{\Omega}(1-\rho^n(t))\mathbf{V}^n(t)\partial_t\mathbf{U}^n(t)dx+
|\rho^n(t)|_{\infty}|\nabla \mathbf{v}^n(t)|_2|\mathbf{V}^n(t)|_3|\mathbf{U}^n(t)|_6+\\ \nonumber
&|\rho \mathbf{v}(t)|_3|\nabla\mathbf{V}^n(t)|_2|\mathbf {U}^n(t)|_6+  |\tilde \rho^n(t)|_6| \mathbf{v}(t)|_{\infty}|\nabla\mathbf{v}^n(t)|_3|\mathbf{U}^n(t)|_2 +\\ \nonumber
&|\tilde \rho^n(t)|_2 |f(t)|_2|\mathbf{U}^n(t)|_{\infty}+ \int_{\Omega}( \mathbf{T}_p(\mathbf{v}^n(t))- \mathbf{T}_p(\mathbf{v}(t)))\mathbf{u}^n(t)\mathbf{U}^n(t)dx.\\ \nonumber
\end{align}

Integration (\ref{eq:s11}) in $t$  gives
\begin{align}
&|\mathbf{U}^n(s)|^2_2 \leq |1- \rho^n(s)|_{\infty}|\mathbf{U}^n(s)|_2 |\mathbf{V}^n(s)|_2 +\label{eq:s12}\\ \nonumber
&c\int_0^s|\mathbf{U}^n(t)|_{\infty}(|\mathbf{V}^n(t)|_2|\partial_t\rho^n(t)|_2+|\tilde{\rho}^n\partial_t\mathbf{v}(t)|_2)dt+\\
&c(r)\int_0^s(|\mathbf{V}^n(t)|_2 +|\nabla\mathbf{V}^n(t)|_2+|\tilde{\rho}^n(t)|_2)|\mathbf{U}^n(t)|_2 dt+\\ \nonumber
  &\int_0^s((\mathbf{T}_p( \mathbf{v}^n(t))- \mathbf{T}_p(\mathbf{v}(t)))\mathbf{u}^n(t),  \mathbf{U}^n(t))dt.\\ \nonumber
\end{align}

Bearing  in mind  that $\mu(\cdot)$  is a  $C^2$-function of its
argument and
$$sup(|\mathbf{v}^n|_{L_{\infty}(Q_{\bar{T}})}, | \nabla{V}^n|_{L_{\infty}(Q_{\bar{T}}) })\leq    c(r),$$
follows
$$ |( \mathbf{T}_p( \mathbf{v}^n(t))- \mathbf{T}_p(\mathbf{v}(t)))\mathbf{u}(t),\mathbf{U^n}(t)|\leq c(r)|\nabla\mathbf{V}^n(t)|_{\infty} |||\mathbf{u}(t)|||_2|\mathbf{U}^n(t)|_2.$$

Besides, by  the multiplicative  inequality ( proposition 1), it  is  routine   matter  to
prove that  $\mathbf{U}^n\rightarrow 0$ strongly in  $L_2(Q_T)$.

Consequently,   the  map $\mathrm{G}$  is  continuous in
$L^2(Q_{\bar{T}})$,  and  the  existence  of  a  local solution is
completely  proved.
 The  proof of   the uniqueness runs  like that  the  continuity  of  $G$.

  Let  $(\rho,\mathbf{u})$  and $(\bar{\rho}, \bar{\mathbf{u}})$ be  two  solutions  of (\ref{eq:s4})   and  let
  $\tilde \rho  =\rho -\bar  \rho$  and
   $\mathbf{U}=\mathbf{u}-\bar {\mathbf{u}}$. Then  $\tilde {\rho}  =\rho -\bar {\rho}$ and $\mathbf{U}$ satisfy the  equation.
\begin{align}
&\partial_t\tilde {\rho} + \mathbf{U}\cdot \nabla\bar{\rho}  + \mathbf{u}\cdot \nabla \tilde {\rho}=0,  \;  \tilde {\rho(0)}=0.\label{eq:s13}\\ \nonumber
\end{align}
Multiplying    (\ref{eq:s13} )  by $\tilde \rho$ and  after  integration  over  $Q_{{T}}$ we   have

 $$|\tilde {\rho}(t)|_2^2\leq  e^{ct}\int_0^{t}|\mathbf{U}\cdot \nabla \bar{\rho}|_2^2dt.$$

 Moreover, $\tilde{\rho}$ and  $\mathbf{U}$ satisfy
    \begin{align}
& \rho\partial_t \mathbf{U}(t)+\mathrm{T}_p(\mathbf{u}(t))\mathbf{U}(t)+(\mathrm{T}_p({\mathbf{u}}(t))-
 \mathrm{T}_p(\bar{\mathbf{u}}(t)))\bar{\mathbf{u}}(t)+\nabla\Pi =  \label{eq:s14}\\ \nonumber
 & \tilde{\rho}\partial_t \bar{\mathbf{u}}(t)-\rho  \mathbf{u}(t)\cdot \nabla  \mathbf{u}(t) +
 \bar{\rho}\bar{\mathbf{u}}(t)\cdot\nabla\bar{\mathbf{u}}(t) +\tilde {\rho}  f.\\ \nonumber
\end{align}

    Therefore it  is  easy to  derive  the inequality
  \begin{align}
&\frac{1}{2}d_t(|\sqrt{\rho}  {\mathbf{U}}|^2_2 +|\tilde  {\rho}|_2^2)+c\| {\mathbf{U}}\|^2 \leq \label{eq:s15}\\ \nonumber
&\omega(t)( \sqrt{\rho} {\mathbf{U}}|^2_2+|\tilde {\rho}|_2^2)+
 \delta  |\nabla {\mathbf{U}}|_2^2,\\ \nonumber
  \end{align}
where

$$\omega(t)=c(|\partial_t\bar{\mathbf{u}}|_2^2 +|\bar {\mathbf{u}}\cdot \nabla \bar {\mathbf{u}}|_3^2+ |\nabla \bar{\rho}|^2_3 +|f|_3^2+ |\nabla \bar {\mathbf{ u}}|^2_{\infty} +  |\nabla \mathbf{u}|^2_{\infty} +|\tilde{\rho}|_{\infty}^2).  $$

Now,  integrating  the  differential inequality  (\ref{eq:s15}),  we   get, for  suitable  $\delta $,

$$|\sqrt{\rho}{\mathbf{ U}}|^2_2 +|\tilde {\rho}|_2^2)=0 \;  a.e.  \; in \;  (0,\bar{T}),$$

$$  \int_0^{\bar{T}}|\nabla {\mathbf{U}}|^2dt=0.$$

Hence  $  {\mathbf{U}} =0$  and $\tilde {\rho}=0$  a.e.in $Q_{\bar{T}}$.

The  theorem  is  completely  proved.

\end{proof}

\section{Multiphase problem}
Up to now, we have  considered  problems  in  fixed  domains.
 Many  physical  problems deal  with  unsteady fluid-fluid, fluid-vacuum or  fluid-structure interaction phenomena.
 These  phenomena  are  of  major  importance  for  aerospace,  mechanical or  biomedical applications.
 The  problem is  to  describe   the  evolution of  a  viscous  fluid  coupled  with  a   moving structure (solid or  liquid).  Several conditions determine the  coupling  between the  media at the surface of  separation or  interface.
 If the  fluid domain   varies under  an  assigned  law the  domain is called  "   non-cylindrical
 domain"  in the  other  cases  it is  called  " free  boundary  domain  ".

The fluid-fluid or  fluid-vacuum problems are  well posed in the  theory of  multiphase systems
used  in fluid mechanics. There,  the  equations of  conservation  of  mass, momentum,  energy and  chemical
species are  written separately  for  each phase,  assuming  that  temperature, pressure, density and
composition  of  each  phase are   equal to  their  equilibrium values.
Accordingly,  these  equations are  supplemented by  boundary conditions  at  interface,  namely
\begin{align}
&\mathbf{ \sigma}|_+^- =\mathcal{C} \mathrm{t} \mathbf{n} -(\mathbf{I}-\mathbf{n}\otimes\mathbf{n} )\cdot\nabla \mathrm{t},  \label{eq:m1}\\ \nonumber
&\mathbf{v}|_+^- =0,\; \mathrm{\theta}|_+^-=0,\\ \nonumber
\end{align}
 with  $\mathbf{n}$ denoting  the  normal at  interface,  stating  that  the  jump of  the
 stress tensor $\mathbf{\sigma}$  at  the  interface is  related  to the  curvature $\mathcal{C}$,  the surface  tension
 $\mathrm{t}$ and  its  gradient, while  the velocity $\mathbf{v}$  and  the  temperature  $\mathrm{\theta} $  are  continuous,
(see \cite{rif1}, \cite{rif19}).

Similar boundary  conditions  exist for  the  transport of  energy  and mass
\begin{align}
& \mathbf{J}_q|_+^-\cdot\mathbf{n} =0,\; \mathbf{J}^i|_+^-\cdot\mathbf{n} =0\label{eq:m2} \\ \nonumber
& \rho^i|_+^-=(k-1)\rho^i|^+\\ \nonumber
\end{align}
stating that  standard (Fourier) heat flux $\mathbf{J}_q$, and the  diffusive  flux  of  any  chemical species  $i$,
 $\mathbf{J}^i$ are  continuous  across the interface (assuming  no phase   transition and  no surface reactions)
while  the  concentration $\rho^i$ can  undergo a  jump, depending  on  a partition coefficient $k$,   given  by
thermodynamics.

Mathematically,  the  model presents serious  difficulties.  A  solution  is  obtained,  in  long
time,  in  the  context  of  Caccioppoli domains. In  this  case the inner normal and  the  interface  are
understood in  a measure theoretic  sense and  non  in  topological  sense (see \cite{rif33}).

Any way, the  free-boundary description is  an effective  model in a wide range  of situations. However, there
 are important instances  where  it   breaks  down, i.e. the  interfacial
thickness is  comparable  to  the  length scale   of  the
 phenomenon.    So an  other approach  was proposed by Rayleigh  and  Van der Waals
 who assumed    that  the  interface  has  a  non-zero thickness, i.e. it is diffuse. Diffuse-interface  models
 provide an  alternative description of interface  motion ( also in  the  case  of fluid-rigid body problem).
 Quantities  that  in the free-boundary  formulation are  localized   in  the  interfacial surface   are  distributed
 throughout interfacial  region.
 The main  characteristic of  the  diffuse
interface model is the use   of an  order parameter  which  undergoes  a rapid   but  continuous
 variation  across the  interphase boundaries,  while  varies  smoothly in  each   bulk phase. In  view
  of  the  arbitrary choice of  " order  parameter " instead  of " diffuse  model " we  write  "field  phase  model".
 Phase-field methods
are based on models of fluid free energy (see \cite{rif1}). The simplest model of free energy density  goes back to
Van der Waals. Cahn and Hilliard \cite{rif5},\cite{rif6} extended Van der Waals hypothesis
to time-dependent situations by approximating interfacial diffusion fluxes as being proportional to chemical potential gradients.
 The  Cahn-Hilliard equation is
\begin{align*}
   & \partial_t C = \lambda\Delta  \phi= - \lambda\Delta(\Delta C - \psi'(C)),\\
   \end{align*}
where $C$ is  order  parameter, $\phi$  is  the  chemical potential, $\psi(C)$  is the  bulk energy  density  that  models  the  fluid  components  immiscibility and  $\lambda$  is a diffusion  parameter.
Mathematical  models are  based  on coupling of  a  Cahn-Hilliard equation, incompressible  homogeneous Navier- Stokes  equations
and   heuristic  speculative formulations.
 Accordingly, the  incompressible  Navier-Stokes equations  are   modified by  the  addition of
 the   continuum forcing related  to  chemical potential (Korteweg-type  stress  tensor ).  The  density (concentration) is variable that  distinguishes the  bulk fluids and the intervening of   interface.

The  fluid-model obtained  has  two characteristics: 1) is  compressible, 2) the  order parameter
  is  constrained  in the  interval $[-1,1]$.  These two conditions make complicate
    the mathematical treatment of  the problem.

   We  quote  the paper \cite{rif4} as  a  reference for  the  construction  of  an  incompressible  model.
   In  appendix, we adapt the procedure  suggested  in \cite{rif2}, \cite{rif8}  to  deduce the
following  incompressible model including  the  effect of   advective-diffusion process.
\begin{align}
&\rho\partial_t \mathbf{ u} +\rho \mathbf{ u}\cdot\nabla \mathbf{ u} -
\lambda \nabla \theta\cdot \nabla \mathbf{u}-\lambda  \mathbf{ u}\cdot\nabla{\nabla\theta}+
\nabla \pi
 -\mu \Delta \mathbf{ u} +\label{eq:t15}\\ \nonumber
&\lambda^2\nabla\cdot(\frac{\nabla \theta \otimes \nabla\theta}{ \rho})  =
\rho \mathbf{f},\\ \nonumber
&\partial_t\rho +\mathbf{u}\cdot \nabla \rho -\Delta \theta=0,\; \nabla\cdot \mathbf{u}=0.\\ \nonumber
\end{align}
Here $\theta =\rho- D\Delta\rho$, $\rho$ is  the  density,  $\lambda$ is   the  diffusion coefficient and $D$ is  the                                  mobility coefficient.
  In  the   next section we  discuss  the  existence for  the  system (\ref{eq:t15}).

\subsection{Existence problem for the diffusion equation and a priori estimates. }

We deduce a priori estimates of the solution of the Neumann problem
\begin{align}
&\partial_t \rho + \mathbf{\psi} \cdot \nabla \rho- \lambda\Delta(\rho -D\Delta \rho) = 0,  \label{eq:1d}\\ \nonumber
&\rho(0) = \rho_0,\;
\partial_{\mathbf{n}}\rho =  \partial_{\mathbf{n}}\Delta\rho =0 \;on\; \Gamma.\\ \nonumber
\end{align}
Here $\mathbf{\psi}$ is a smooth divergence free function  vanishing  on $\Gamma$.

For the mathematical setting of the problem, we introduce
$$\mathcal{H}:=\{\phi|\phi\in H^2(\Omega),\;\partial_{\mathbf{n}}\phi_{\Gamma}=0\},$$
$$\mathcal{H}^2:=\{\phi|\phi\in H^4(\Omega),\;\partial_{\mathbf{n}}\phi_{\Gamma}=\partial_{\mathbf{n}}\Delta
\phi_{\Gamma}=0\}.$$
$\mathcal{H}$   is closed  subspace  of $H^2(\Omega)$ and is   endowed  with the  scalar
product   and  the  norm
$$((\phi,\chi))=\int_{\Omega} (\Delta \phi, \Delta \chi)dx,\;  \|\phi\|_*^2=((\Delta \phi,\Delta \phi)).$$
$a(\phi,\chi):=((\phi,\chi))=\int_{\Omega} (\Delta \phi, \Delta \chi)dx$  is  a  bilinear  form on $H^2$.
By   the   Poicare's  inequality and  regularity  of  elliptic  problem, $\|\phi\|_*+ c|\phi|_2$ or $\|\phi\|_*+ c\int_{\Omega}\phi dx$  are  norms   in  $\mathcal{H}$ equivalent  to  the   norm  induced  by $H^2(\Omega)$. In  general, denoted
$m(\phi)=\frac{1}{|\Omega|}\int_{\Omega}\phi(x)dx$ ,   $s\geq 0$,
$$\|\phi-m(\phi)||_1\leq  c|\nabla \phi|_2,\; \|\phi-m(\phi)\|_{s+2}\leq c\|\Delta \phi\|_s. $$

For  convenience,  we set  $\lambda=D=1$. The existence for the diffusion equation (\ref{eq:1d}) can be  performed
through Faedo-Galerkin method.

This procedure is well known in
literature so we omit details.
We  prove now a priori estimates, formally.

First, we notice that  the maximum principle does  not   hold, in
general. But  the  average  value  is  conserved:
$$\int_{\Omega}\rho(t)dx =\int_{\Omega}\rho_0. $$
Moreover, since  $\rho$  is a  solution of  the  continuity  equation (see appendix (\ref{eq:t3})) $\rho(t)\geq 0$  if $ \rho_0\geq 0$.
Indeed, $\rho$  satisfies
$$ (\inf \rho_0)exp(-\int_0^t|\nabla\cdot\mathbf{w}(\tau)|_{\infty})d\tau\leq \rho(x,t).$$
Now, we prove five levels of regularity for $\rho$.

Multiplying $(\ref{eq:1d})_1$  by $\rho$ and integrating by parts in
$\Omega$,  we get
$$
d_t|\rho|_2^2 + 2|\nabla \rho|_2^2+ 2|\Delta \rho|_2^2 = 0,
$$
then
\begin{align}
|\rho(t)|_2^2 + \int_0^t (|\nabla \rho(\tau)|_2^2+|\Delta \rho(\tau)|_2^2)d\tau\leq |\rho_0|_2^2.
\label{eq:2d}\\ \nonumber
\end{align}

\noindent
The  estimate   (\ref{eq:2d})   is   independent   of  $\mathbf{\psi}$.

Now, multiplying (\ref{eq:1d}) by $- \Delta \rho$ and after
integration by parts on $\Omega$,  we
obtain
$$
\frac{1}{2} d_t|\nabla \rho|_2^2 + |\Delta \rho|_2^2+|\nabla\Delta \rho|_2^2
= ( \mathbf{\psi}\cdot \nabla \rho, \Delta \rho).
$$
In virtue of
\begin{align*}
&|(\mathbf{\psi}(t)\cdot
\nabla \rho(t), \Delta \rho(t))| \leq c |\psi(t)\rho(t)|_2^2   + \frac{1}{2}|\nabla\Delta\rho(t)|^2_2\leq\\
  &c  |\mathbf{\psi}(t)|^2_2|\Delta \rho(t)|^{3/2}_2|\rho(t)|_2^{1/2}
  + \frac{1}{2}|\nabla\Delta\rho(t)|^2_2,\\
  \end{align*}
thus we get

\begin{align}
&|\nabla \rho(t)|_2^2 + \int_0^t(|\Delta \rho(\tau)|_2^2+ |\nabla\Delta \rho(\tau)|_2^2)d\tau
\leq \\ \nonumber
 & c\sup_{0\leq \tau\leq t}|\mathbf{\psi}(\tau)|_2^2\sqrt[4]{t}+ |\nabla\rho_0|^2_2:=\Psi_1(t).\\ \nonumber
\end{align}

Notice that the above estimate requires  $\mathbf{\psi} \in
L_{\infty}(0,T;L_2(\Omega))$, only.
 Now,  we deduce  $H^4$ -estimate for $\rho$.

 Multiplying $(\ref{eq:1d})_1$ by $\Delta^2 \rho$ and  after integration by
parts ( bearing in mind the boundary terms vanish),  we deduce
$$
 \frac{1}{2}d_t|\Delta \rho|_2^2 + |\Delta \nabla
\rho|_2^2 +|\Delta^2 \rho|_2^2 =-(\mathbf{\psi}\cdot \nabla \rho, \Delta^2 \rho).
$$
In virtue of
\begin{align*}
&|(\mathbf{\psi}(t)\cdot \nabla \rho(t),\Delta^2\rho(t))| \leq c|\mathbf{\psi}(t)|_2|\nabla\rho(t) |_{\infty}|\Delta^2\rho(t)|_2\leq \\
&c|\mathbf{\psi}(t)|^2_2|\nabla\rho(t)|_2^{1/2}|\nabla\Delta\rho(t)|^{^3/2}_2 +
\frac{1}{2}|\Delta^2\rho(t)|^2_2,\\
\end{align*}
thus we obtain

\begin{align}
&|\Delta \rho(t)|_2^2 +
\int_0^t(|\nabla \Delta \rho(\tau)|_2^2+|\Delta^2\rho(\tau)|^2_2)d\tau\leq\label{eq:3d}\\ \nonumber
&c\sqrt[4]{t}(\sup_{0\leq
\tau\leq t}|\mathbf{\psi}(\tau)|^2_2)\Psi_1(t)+|\Delta\rho(0)|^2_2:=\Psi_2(t).\\ \nonumber
\end{align}
Next,  we  prove   $H^5$-estimate.
First,   we  apply  the   $\nabla$  operator to $(\ref{eq:1d})_1$ and
then  we  multiply   the  result  by  $\nabla \Delta^2\rho$, after
integration    by parts, we  obtain
\begin{align*}
& \frac{1}{2}d_t|\nabla\Delta \rho|_2^2 + |\nabla\Delta^2\rho|_2^2 +|\Delta^2 \rho|_2^2 =
-(\mathbf{\psi}\cdot \nabla \nabla\rho, \nabla\Delta^2 \rho)-(\nabla\mathbf{\psi}\cdot  \nabla\rho, \nabla\Delta^2 \rho)\leq\\
&c(|\psi|_6^2|\nabla\nabla \rho|^2_3+|\nabla \psi|_2^2 |\nabla\rho|_{\infty}^2) +
\frac{1}{2}|\nabla\Delta^2 \rho|_2^2\leq c|\nabla\mathbf{ \psi}|^2_2|\nabla\rho|_6|\nabla\Delta\rho|_2+
\frac{1}{2}|\nabla\Delta^2 \rho|_2^2.\\
\end{align*}
We obtain
\begin{align}
&|\nabla\Delta \rho(t)|_2^2 +  \int_0^t(|\nabla \Delta^2 \rho(\tau)|_2^2+|\Delta^2\rho(\tau)|^2_2)d\tau\leq
\label{eq:4d}\\ \nonumber
&\sqrt{t}(\sup_{0\leq\tau\leq t}
|\nabla\mathbf{\psi}(\tau|^2_2|)\Psi_2(t)+
|\nabla\Delta \rho(0)|_2^2:=\Psi_3(t).
\\ \nonumber
\end{align}
According   to  the   above   estimates, if  $\mathbf{\psi}\in
L^2(0,T;H_0^1(\Omega))$,   we  have

$$\partial_t\rho \in L^2(0,T;H^1(\Omega)).$$

In addition, if we assume $\partial_t \mathbf{\psi} \in L^2(Q_T)$, we derive  analogous estimates  for
$\partial_t \rho$. Indeed, setting $\eta=\partial_t \rho$ we  get
\begin{align}
   &\partial_t \eta +\mathbf{\psi}\cdot\nabla \eta-\Delta(\eta-\Delta \eta)= -\partial_t \mathbf{\psi}\cdot\nabla \rho.
    \label{eq:5d}\\ \nonumber
    \end{align}
Similarly,  we  obtain   $H^4$- estimates for  $\eta$ considering
$ -\partial_t \mathbf{\psi}\cdot \nabla\rho$ as  a given right  hand  side, namely   adding to  the above  estimates
the $L^2(Q_T)$  norm  of $ -\partial_t \mathbf{\psi}\cdot \nabla\rho$
 with $\eta(0)\in H^3(\Omega)$.

Consequences   of  the  previous  estimates and  proposition 1 are :
\begin{align}
&1.\;\int_0^t|\nabla \rho(\tau)|^2_{\infty}d\tau \leq c\sqrt[4]{t}\Psi_1(t);\label{eq:6dd}
\\ \nonumber
&2. \;\sup_{0\leq\tau\leq t}|\Delta \rho(\tau)|_2^2\leq
 c\Psi_2(t);\\ \nonumber
&3. \;\int_0^t|\Delta^2\rho(\tau)|^2_2d\tau \leq c\Psi_2(t);\\ \nonumber
&4. \;\int_0^t|\Delta^2\nabla\rho(\tau)|^2_2d\tau \leq c\sqrt{t}(sup_{0\leq\tau\leq t}|\nabla\mathbf{\psi}(\tau)|^2_2)\Psi_2(t) +|\Delta \nabla\rho(0)|_2^2\\ \nonumber
&:=\Psi_3(t);\\ \nonumber
&5.\;\sup_{0\leq\tau\leq t}|\Delta\nabla\rho(\tau)|^2_2 \leq c\sqrt{t}(sup_{0\leq\tau\leq t}|\nabla\mathbf{\psi}(\tau)|^2_2)\Psi_2(t) +
|\Delta \nabla\rho(0)|_2^2\\ \nonumber
&:=\Psi_3(t)\\ \nonumber
&6. \;\int_0^t|\partial_{\tau}\Delta^2\rho(\tau)|^2_2d\tau\leq c (\sqrt{t}(sup_{0\leq\tau\leq t}|\nabla\mathbf{\psi}(\tau)|^2_2)\Psi_2(t) + \\ \nonumber
&|\Delta \nabla\rho(0)|_2^2)\int_0^t|\partial_{\tau}\mathbf{\psi}|^2_2d\tau +
  \int_0^t|\mathbf{\psi}\cdot\nabla \partial_t\rho(\tau)|_2^2d\tau + |\Delta \partial_t\rho(0)|^2_2\leq \\ \nonumber
&c \Psi_3(t)\int_0^t|\partial_{\tau}\mathbf{\psi}(\tau)|^2_2d\tau +
 \sqrt[4]{t}\sup_{0\leq \tau \leq t}|\nabla\mathbf{\psi}(\tau)|^2_2\Psi_2(t)
\int_0^t|\partial_{\tau}\mathbf{\psi}(\tau)|^2_2d\tau +\\ \nonumber
  &|\partial_t \rho(0)|_2^2 +|\Delta \partial_t\rho(0)|_2^2:=\Psi_4(t).\\ \nonumber
\end{align}

\subsection{ Graffi's type  model}

In   this  section  we  discuss  a  simplified   model instead  of (\ref{eq:t15}).  We
notice   that the  term

$$\lambda  (\mathrm{u}\cdot \nabla)\nabla \theta$$
in  (\ref{eq:t15}) is   a serious   obstacle to prove the  existence   of  a sort  of  solution  (  weak   or  strong)  .  If
$\theta  =\rho$  (see \cite{rif2}), the   above   term is  efficiently estimated   by   using   the maximum  principle,
$|\nabla \rho|_4\leq  c|\rho|_{\infty}|\Delta \rho|_2$, and   a constrain  on $\lambda$, $\mu$  and  the maximum  of $\rho$.

For  future investigations, in this section we study the  system  omitting  the terms
$\lambda  \mathbf{u}\cdot \nabla\nabla \theta$ and O($\lambda^2$). Graffi  in \cite{rif13} considered a system of  equations type  (\ref{eq:t15}) but
discarding   the  two O($\lambda$) and O($\lambda^2$) terms
  in  $(\ref{eq:t15})_1$. In  other  words  we   consider  the following    simplified model ($\lambda =1$)

\begin{align}
&\rho\partial_t \mathbf{u}+\rho \mathbf{u}\cdot\nabla \mathbf{u}
-\nabla \theta\cdot\nabla \mathbf{u}
-\nabla \pi - \mu \Delta \mathbf{u} = \rho \mathbf{f},\label{eq:6d}\\ \nonumber
&\partial_t \rho + \mathbf{u}\cdot\nabla \rho -\Delta(\rho  -\Delta \rho)=0,\\ \nonumber
 &\nabla\cdot \mathbf{u} =0, \;\rho(0)=\rho_0, \;\mathbf{u}(0)=\mathbf{u}_0.\\ \nonumber
\end{align}

We will derive essential a priori estimates, formally.  The complete proof  can  be  performed through
 a  Faedo-Galerkin process. Now,  we  prove  a priori  estimates which  allow to   deduce   the  existence of   a  weak  solution
of  (\ref{eq:6d}).  We  consider  the  approximate   system

\begin{align}
&\rho\partial_t \mathbf{u}^n+\rho^n \mathbf{u}^{n-1}\cdot\nabla \mathbf{u}^n
-\nabla \theta^n\cdot\nabla \mathbf{u}^n-
\nabla \pi^n - \mu \Delta \mathbf{u}^n =\rho^n \mathbf{f},\label{eq:7d}\\ \nonumber
&\partial_t\rho + \mathbf{u}^{n-1}\cdot\nabla \rho^n -\Delta(\rho^n  -\Delta \rho^n)=0,\\ \nonumber
 &\nabla\cdot \mathbf{u}^n =0 , \rho^n(0)=\rho^n_0, \;\mathbf{u}^n(0)=\mathbf{u}^n_0, \\ \nonumber
\end{align}

\noindent
where $\mathbf{u}^{n-1}$   is  a suitable solenoidal function. Moreover, we   make   use   of  the results   in  section 3
and of the  usual   initial- boundary  conditions.
Multiplying $(\ref{eq:7d})_1$  by  $\mathbf{u}^n$  and   after
integration by  parts we   get
\begin{align}
&d_t|\sqrt {\rho^n}\mathbf{u}^n|_2^2 + \mu|\nabla \mathbf{u}^n|_2^2- \frac{1}{2}(\mathbf{u}^n\partial_t\rho^n,\mathbf{u}^n)+
 (\rho^n \mathbf{u}^{n-1} \cdot\nabla \mathbf{u}^n,\mathbf{u}^n)+ \label{eq:8d}\\   \nonumber
 &\frac{1}{2}(\Delta\theta^n \mathbf{u}^n,\mathbf{u}^n)-
 (\rho^n \mathbf{f},\mathbf{u}^n)=0.\\ \nonumber
\end{align}
Multiplying $(\ref{eq:7d})_2$ by $\frac{|\mathbf{u}^n|^2}{2}$ and summing the  result to (\ref{eq:8d}) we   get
\begin{align}
&d_t|{\sqrt \rho}\mathbf{u}^n|_2^2 + \mu|\nabla \mathbf{u}^n|_2^2=
(\rho f,u^n),
\label{eq:d8}\\ \nonumber
\end{align}
and  easily we obtain the  standard   energy  estimate
\begin{align}
&|\sqrt {\rho}\mathbf{u}^n(t)|_2^2 + \mu\int_0^t|\nabla \mathbf{u}^n(\tau)|_2^2d\tau\leq |\sqrt {\rho(0)^n}\mathbf{u}^n(0)|_2^2 +
\int_0^t(|\rho^n f(\tau)|_2^2d\tau.
\label{eq:9d}\\ \nonumber
\end{align}
Notice   that  the right-hand  side in (\ref{eq:9d})  does   not
depend   on  $\mathbf{u}^{n-1}$.

Consequently, there exists  a subsequence
$\{\mathbf{u}^n,\rho^{n} \}$ such that
\begin{align}
& \mathbf{u}^n\rightarrow \mathbf{u}  \; \text{weakly in }\; L^2(0,T;V),\label{eq:11d}\\ \nonumber
  &\rho^{n}
\rightarrow \rho \;\text{ weakly  in
} L^{2}(0,T;H^4(\Omega)),\; \text{ strongly  in } \;L^2(0,T;H^2(\Omega)),\\ \nonumber
&\mathbf{u}^n \rho^{n}\rightarrow \mathbf{u}\rho, \quad
u_i^{n-1}u_j^{n}\rho^{n} \rightarrow \alpha_{ij} \text{ in  the  sense   of  distributions}.\\ \nonumber
\end{align}

To complete the existence proof of a weak solution we have to show
$\alpha_{ij} = \rho u_iu_j $. For this we estimate  the time
derivative of  $\mathbf{u}^n$.

\subsection{Time derivative estimates and compactness result.} Now, let $ \phi $ be  a smooth solenoidal  function vanishing  on  $\Gamma$.
Multiplying  $(\ref{eq:7d})_1$ by $ \mathbf{\phi}$  and after integration by parts we get
\begin{align}
&\int_0^T((\partial_t\rho ^{n}\mathbf{u}^{n}(t),\mathbf{\phi}(t)) +
(\rho^{n}(t)\mathbf{u}^{n-1}(t),\mathbf{u}^{n}(t)\cdot \nabla \mathbf{\phi}(t))  - \mu
(\nabla \mathbf{u}^{n}(t),\nabla \mathbf{\phi}(t)) -\label{eq:12d}
\\ \nonumber
&((\nabla \theta^{n}(t)\cdot
\nabla)\mathbf{\phi}(t),\mathbf{u}^n(t))+
(\rho^n f(t),\mathbf{\phi}(t)))dt = 0.\\ \nonumber
\end{align}
The estimates (\ref{eq:6dd}), (\ref{eq:9d}) show that
$\partial_t(P\rho^{n}\mathbf{u}^n(t))$ is bounded
in

\noindent
$L_2(0,T;H^{-2}(\Omega))$,
uniformly with respect to $n$, while
$\rho^{n}\mathbf{u}^{n}$ and thus $P\rho^{n}\mathbf{u}^{n}$ are
bounded

\noindent
in $L_{\infty}(0,T;L_2(\Omega))$, uniformly  with respect to
$n$.

Hence, by classical compactness theorems,
$\{P\rho^{n}\mathbf{u}^{n}\}$
is a compact set in

\noindent
$ L_2(0,T;H^{-1}(\Omega))$.

\noindent
In particular, since (subsequence) $\{\rho^{n}\mathbf{u}^{n}\}$
converges weakly to $\rho \mathbf{u}$, $\{P\rho^{n}\mathbf{u}^{n}\}$
converges to $P\rho \mathbf{u}$ in $ L_2(0,T;H^{-1}(\Omega))$.

Now,  making  use  of  the  procedure  in theorem 4, we  get
the strong
convergence in $L^2(Q_T)$ of $\sqrt {\rho^n}\mathbf{u}^{n}$ to
${\sqrt \rho} \mathbf{u} $. This convergence implies that  $\alpha_{ij}= \rho u_i u_j$.

The  existence  proof    of   a  weak   solution  of (\ref{eq:6d})    is thus  complete.

\section{Local existence  problem  for  the  system (\ref{eq:t15}).}
In this  section  we  prove  the  following  theorem
\begin{theorem}   Let  $\Omega \subset \mathbb{R}^3$ be  a  domain  with compact  boundary  $\Gamma $  of  class $C^3$.
Assume  $\mathbf{u}_0\in\mathrm{V}$, $(\rho_0,\partial_t\rho(0))\in (H^3(\Omega), H^2 (\Omega))$ and  $\mathbf{f}\in
L_{\infty}(0,T;L_2(\Omega))$. Then there exists  a  $\bar{T}>0$  such  that
there is   a  unique  solution  $(\mathbf{u},\rho,\pi)$ of
\begin{align}
&\bar{\rho}\partial_t\mathbf{u} - \mu \Delta \mathbf{u} + \bar{\rho}\mathbf{u}\cdot \nabla \mathbf{u}
-\lambda((\mathbf{u}\cdot \nabla)\nabla \theta+(\nabla \theta\cdot \nabla)\mathbf{u})+\label{eq:0r}\\ \nonumber
&\frac{\lambda^2}{\bar{\rho}}((\nabla \theta\cdot \nabla)\nabla
\theta - \frac{1}{\bar{\rho}}(\nabla \bar{\rho}\cdot\nabla \theta)\nabla\theta+
\Delta \theta\nabla \theta)+ \nabla \pi -\bar{\rho} \mathbf{f} = 0,\\ \nonumber
&\partial_t \rho + \mathbf{u} \cdot \nabla \rho- \lambda\Delta\theta = 0,  \\ \nonumber
&\nabla\cdot \mathbf{u}=0,\; \mathbf{u}(0)=\mathbf{u}_0,\;\rho(0) = \rho_0,\;
\partial_{\mathbf{n}}\rho =\partial_{\mathbf{n}}\Delta \rho =0  \;on\; \Gamma.\\ \nonumber
\end{align}
 on the  time interval  $[0,\bar{T})$  such  that
 \begin{align*}
&\mathbf{u}\in H^1(0,\bar{T},V^0(\Omega))\cap L_2(0,\bar{T},H^2(\Omega)\cap V),\;\pi\in L_2(0,\bar{T};H^1(\Omega))/\mathbb{R}),\\
&\rho\in  L_2(0,T;H^5(\Omega)),\;
\partial_t\rho \in L_2(0,T;H^4(\Omega)).\\
\end{align*}
\end{theorem}

Here  $\bar{\rho}= \rho +m$ with $m$  a  positive  number.

\begin{proof}

We prove Theorem 8.1 by fixed point argument
 following the scheme of the Theorem 6.1.
Let $$\mathrm{B}(r)=\{\mathbf{\phi}| sup( \|\mathbf{\phi}\|_{ L_2(0,T;H^2(\Omega))},  \|\mathbf{\phi}\|_{H^1(0,T;L_2(\Omega))})\leq r\}.$$
  We fix a function $\mathbf{v}\in \mathrm{B}(r)$  and subsection 7.1 gives
 the solution of
diffusion equation (\ref{eq:1d}) with $\mathbf{\psi}\equiv\mathbf{v} $ and relative  estimates on
${\rho}$ up to the $H^5$-regularity.

Next, we consider the  linear   problem with  fixed  $\mathbf{v}\in  \mathrm{B}(r)$,
\begin{align}
&\bar{\rho}\partial_t\mathbf{u} - \mu \Delta \mathbf{u} + \bar{\rho}\mathbf{v}\cdot \nabla \mathbf{v}
-\lambda((\mathbf{v}\cdot \nabla)\nabla \theta+(\nabla \theta\cdot \nabla)\mathbf{v})+\label{eq:1r}\\ \nonumber
&\frac{\lambda^2}{\bar{\rho}}((\nabla \theta\cdot \nabla)\nabla
\theta - \frac{1}{\bar{\rho}}(\nabla \bar{\rho}\cdot\nabla \theta)\nabla\theta+
\Delta \theta\nabla \theta)+ \nabla \pi -\bar{\rho}\mathbf{ f} = 0,\\ \nonumber
\end{align}
complemented with the usual initial-boundary conditions.
The
existence of a solution of system (\ref{eq:1r})  is established in \cite{rif31}   making  use of Lax-Milgram theorem.

\subsection{ A priori estimates.}

For  convenience we  set $\lambda= \mu=1$. Moreover  we  make  use  of  the inequalities
$$|\phi|^2_{\infty}\leq c|\phi|_6|\nabla\nabla \phi|_2 \; ,|\phi|_{\infty}\leq c|\phi|^{1/4}_2|\nabla\nabla \phi|^{3/4}_2.$$

 We multiply now by $\mathbf{u}$ (\ref{eq:1r}) and after integration by parts,  integration  in $t$
gives

\begin{align}
&|\sqrt{\bar{\rho}}\mathbf{u}(t)|_2^2 + \int_0^t |\nabla \mathbf{u}(\tau)|_2^2d\tau \leq  |\sqrt{\bar{\rho}}\mathbf{u}(0)|_2^2+
\int_0^t|\nabla \nabla \theta(\tau)|_4^4d\tau+\label{eq:r2}\\  \nonumber
&\sqrt{t}\sup_{0\leq \tau\leq t}|\nabla\mathbf{v}(\tau)|^2_2\Psi_2(t)(\int_0^t|A\mathbf{v}(\tau)|_2^2d\tau)^{1/2}
+t\rho(0)\sup_{0\leq \tau\leq t}|\mathbf{f}(\tau)|_2^2+\\ \nonumber
&\int_0^t(\mathbf{u}(\tau) \partial_{\tau}{\rho}(\tau),\mathbf{u}(\tau))d\tau. \\ \nonumber
\end{align}

Now, we multiply (\ref{eq:1r}) by $\partial_t \mathbf{u} $,
integrate over $\Omega$, and obtain
\begin{align}
&|\sqrt{\bar{\rho}(t)}\partial_t\mathbf{u}(t)|^2_2 +  d_t|\nabla \mathbf{u}(t) |_2^2  \leq\label{eq:r3}\\ \nonumber
& c|\bar{\rho}(t) |^2_{\infty}|\mathbf{v}(t)|_{\infty}^2|\nabla \mathbf{v}(t)|_2^2 +
 |\mathbf{v}(t)|^2_{\infty}|\nabla\nabla \theta(t)|^2_2 +
 |\nabla \theta(t)|^2_{\infty}| \nabla \mathbf{v}(t)|_2^2 +\\ \nonumber
 &|\nabla\nabla \theta(t)|^2_2|\nabla\theta(t)|^2_{\infty} +
 c|\partial_t \mathbf{u}|^2_2 \leq
|\bar{\rho}(t) |^2_{\infty}\|\mathbf{v}(t)\|^3|A\mathbf{v}(t)|_2 + \\ \nonumber
  & |\nabla \theta(t)|^2_{\infty}| \nabla \mathbf{v}(t)|_2^2+\|\mathbf{v}(t)\||A\mathbf{v}(t)|_2|\Delta \theta(t)|_2^2+ |\Delta \theta(t)|^3_2\|\theta(t)\|_3+\\ \nonumber
&|\rho(t)|_{\infty}^2|\mathbf{f}(t)|_2^2 +
\frac{1}{2}|\partial_t \mathbf{u}|^2_2.\\ \nonumber
\end{align}

Then, integrating (\ref{eq:r3}) with respect to $t$, we get
\begin{align}
&\int_0^t|\sqrt {\bar{\rho}(\tau)}\partial_{\tau}\mathbf{u}(\tau)|^2_2d\tau  + \|\nabla\mathbf{u}(t)\|_2^2  \leq
|\nabla \mathbf{u}(0)|^2_2 +\label{eq:r4}\\ \nonumber
& \sqrt{t}\Psi_2(t)\sup_{0\leq \tau\leq t}
 |\nabla\mathbf{v}(\tau)|^3_2(\int_0^t|A\mathbf{v}(\tau)|_2^2d\tau|)^{1/2}+\\  \nonumber
& \sqrt{t}\Psi_4(t)(\sup_{0\leq \tau\leq t}|\nabla\mathbf{v}(\tau)|_2(\int_0^t|A\mathbf{v}(\tau)|_2^2d\tau)^{1/2}
+ \sup_{0\leq\tau\leq t}|\nabla \mathbf{v}(\tau)|^2_2) +\\ \nonumber
  &c\sqrt{t}\Psi_3(t)\Psi_4(t) + \sqrt[4]{t}|\rho_0|^2_2\sup_{0\leq \tau \leq t}|\mathbf{f}(\tau)|_2^2.  \\ \nonumber
\end{align}

Now, we consider the Stokes problem
\begin{align}
& A\mathbf{u} = -P(\bar{\rho}\partial_t\mathbf{u}+ \bar{ \rho} \mathbf{v}\cdot \nabla \mathbf{v}
-((\mathbf{v}\cdot \nabla)\nabla \theta +(\nabla \theta\cdot \nabla)\mathbf{v})+  \label{eq:2r}\\ \nonumber
&\nabla \cdot(\frac{1}{\bar{\rho}}(\nabla \theta\otimes \nabla
\theta)) -\bar{\rho}\mathbf{f}).\\ \nonumber
\end{align}
From the theory of the Stokes problem there exists a constant $c$ such that
\begin{align}
&\int_0^t|A \mathbf{u}(\tau)|^2_2d\tau\leq
|\nabla \mathbf{u}(0)|^2_2 + c \sqrt{t}\sup_{0\leq
\tau\leq t}|\nabla \mathbf{v}(\tau)|_2^2|\Psi_4(t)+c\sqrt{t}\Psi_3(t)\Psi_4(t)\label{eq:3r}\\ \nonumber
& +c\sqrt{t}\Psi_2(t)\sup_{0\leq \tau\leq t}|\nabla\mathbf{v}(\tau)|^3_2(\int_0^t|A\mathbf{v}(\tau)|_2^2d\tau|)^{1/2}+\\  \nonumber
 & \sqrt{t}\Psi_4(t)\sup_{0\leq \tau\leq t}|\nabla\mathbf{v}|_2(\int_0^t|A\mathbf{v}(\tau)|_2^2d\tau)^{1/2}
 + \sqrt[4]{t}|\rho_0|_2^2\sup_{0\leq \tau \leq t}|\mathbf{f}(\tau)|_2^2.  \\ \nonumber
\end{align}

Choosing $r^2 \geq c(|\nabla \mathbf{u}_0|_2^2+|\Delta \nabla\rho_0|_2^2 +|\Delta\partial_t\bar{\rho}(0)|^2 +\sup_{0\leq t\leq T}|\mathbf{f}|^2_2)$
with $c$ big enough, (\ref{eq:6dd}), (\ref{eq:2r}), (\ref{eq:3r}) yield
 \begin{align}
&\int_0^t(|{\sqrt \rho}\partial_{\tau}\mathbf{u}(\tau)|^2_2+|A\mathbf{u}(\tau)|_2^2)d\tau  +  \sup_{0\leq\tau\leq t}|\nabla \mathbf{u}(\tau)|_2^2  \label{eq:4r} \leq\\ \nonumber
&c(|\nabla \mathbf{u}(0)|^2_2  + c\sqrt tr^3(\int_0^t|A\mathbf{v}|_2^2d\tau)^{1/2} + \\ \nonumber
&\sqrt t r sup_{0\leq \tau \leq t}|\Delta \theta(\tau)|_2^2(\int_0^t|A\mathbf{v}(\tau)|_2^2d\tau)^{1/2} +\\ \nonumber
&\sqrt t sup_{0\leq \tau \leq t}|\Delta \theta(\tau)|_2^3(\int_0^t\|\theta(\tau)\|_3^2d\tau)^{1/2}+
\rho(0)\sup_{0\leq \tau\leq t}|\mathbf{f}(\tau)|^2_2)\leq \\ \nonumber
&c(|\nabla \mathbf{u}(0)|^2_2 +|\Delta \nabla{\rho}_0|_2^2 +|\Delta \partial_t\rho(0)|_2^2 +\sqrt tr^4
+ t\rho^2(0)\sup_{0\leq \tau\leq t}|\mathbf{f}(\tau)|^2_2)\leq r^2.\\ \nonumber
  \end{align}
for $t= \bar T$ small enough.

Hence (\ref{eq:4r}) implies
$$G\mathrm{ }\subseteq \mathrm{}.$$

 We  prove now the continuity of $G$.

 Let $\{\mathbf{v}^n\} \subset \mathrm{B}$  be a sequence such that
 $\mathbf{v}^n\rightarrow \mathbf{v}$ in $L^2(Q_{\bar T} )$,  strongly. We notice that $\mathbf{v}\in \mathrm{B}$.
 Thanks  to Proposition 1,  we deduce
 \begin{align}
 &\mathbf{v}^n\rightarrow\mathbf{v} \; \text{in} \;L_2(0,T;L_{\infty}(\Omega))\cap L_2(0,T;W^1_3(\Omega)), \label{eq:g}\\ \nonumber
 &\theta^n \rightarrow\mathbf{\theta} \; \text{in}\; L_2(0,T;L_{\infty}(\Omega))
 \cap L_q(0,T;W^1_2(\Omega))\cap \\ \nonumber
 &L_4(0,T;W^1_3(\Omega))\cap L_{\infty}(0,T;L_2(\Omega)),\\  \nonumber
  \end{align}
  with $q>1$, strongly.

Let $\rho^n$, $\rho$ be solutions of
 \begin{align}
  &\partial_t \rho^n + \mathbf{v}^n\cdot \nabla \rho^n-  \Delta \theta^n = 0;\; \rho^n(0)= \rho_0,
   \partial_{\mathbf{n}}\rho^n=\partial_{\mathbf{n}}\Delta \rho^n =0 \; on  \;\Gamma,\label{eq:5r}\\ \nonumber
  &\partial_t \rho + \mathbf{v}\cdot \nabla \rho- \Delta \theta = 0;\; \rho(0)= \rho_0,\;
  \partial_{\mathbf{n}}\rho =\partial_{\mathbf{n}}\Delta \rho=0 \;on  \;\Gamma,\\ \nonumber
\end{align}
respectively.

    Now $\tau^n = \rho^n -\rho$ and $\Theta^n  =\theta^n-\theta$ satisfy
    \begin{align}
     \partial_t \tau^n + \mathbf{v}^n\cdot \nabla \tau^n-  \Delta \Theta^n = -(\mathbf{v}^n-\mathbf{v})\cdot \nabla \rho,\label{eq:r35}\\\nonumber
     \tau^n(0)= 0,\;
        \partial_{\mathbf{n}}\tau^n =\partial_{\mathbf{n}}
        \Theta=0 \;on \;\Gamma.
\\ \nonumber
\end{align}
 The estimates (\ref{eq:2r}), (\ref{eq:5r}), Gronwall's lemma and
Proposition  1  yield  that $\tau^n\rightarrow 0$ in $L_2(0,\bar
{T};H^2)\cap L_{\infty}(0,\bar {T};H^1)$.

Now, let $\mathbf{u}^n, \mathbf{u}$ be the solutions of

\begin{align}
 &\bar{\rho}^n\partial_t \mathbf{u}^n -  \Delta \mathbf{u}^n +\bar{\rho}^n \mathbf{v}^n\cdot \nabla \mathbf{v}^n -
 (\mathbf{v}^n\cdot \nabla)\nabla \theta^n+(\nabla \theta^n\cdot
\nabla)\mathbf{v}^n+\label{eq:6r}\\ \nonumber
&\nabla \cdot(\frac{1}{\bar{\rho}^n}(\nabla
\theta^n\otimes \nabla \theta^n))+ \nabla \pi^n - \bar{\rho}^n\mathbf{f} = 0,\\ \nonumber
   &\bar{\rho}\partial_t \mathbf{u}  -  \Delta \mathbf{u} + \bar{\rho} \mathbf{v}\cdot \nabla \mathbf{v}
   -(\mathbf{v}\cdot \nabla)\nabla \theta +(\nabla \theta\cdot
   \nabla)\mathbf{v}+\\ \nonumber
   &\nabla \cdot(\frac{1}{\bar{\rho}}\nabla \theta\otimes
   \nabla \theta))+ \nabla \pi -\bar{\rho} \mathbf{f} = 0,\\ \nonumber
\end{align}
with $\mathbf{u}^n (0) = \mathbf{u}(0) =\mathbf{u}_0$, respectively.

 Now $\mathbf{U}^n = \mathbf{u}^n-\mathbf{u} $ satisfies
   \begin{align}
  & \bar{\rho}\partial_t\mathbf{U}^n - \Delta \mathbf{U}^n = H(\mathbf{v}^n,\mathbf{v},\rho^n,\rho) - \nabla (\pi^n - \pi)  -\tau^n \partial_t \mathbf{u}^n.
    \label{eq:7r}\\ \nonumber
    \end{align}

 It is easy to  trace  $H(\cdot)$ and to prove that  $H(\mathbf{v}^n,  \mathbf{v},\rho^n,\rho)\rightarrow 0$ as $n\rightarrow \infty$ in
$L_2(Q_T).$
Indeed,
$$\bar{\rho}(\mathbf{v}^n\cdot \nabla \mathbf{v}^n-\mathbf{v}\cdot \nabla \mathbf{v})=$$
$$\bar{\rho}((\mathbf{v}^n\cdot \nabla \mathbf{v}^n-\mathbf{v}\cdot \nabla \mathbf{v}^n)+
(\mathbf{v}\cdot \nabla \mathbf{v}^n-\mathbf{v}\cdot \nabla \mathbf{v}))\rightarrow 0,$$
thanks to (\ref{eq:g}).

 Similarly, we  can  proceed  for  every  term in $H(\cdot)$.

Now, multiplying (\ref{eq:7r}) by $\partial_t\mathbf{U}^n$, after
integration by parts, we get
$$ |\nabla \mathbf{U}^n(t)|^2_2  + \int_0^t|\sqrt {\bar{\rho}}\partial_t\mathbf{U}^n(\tau)|^2_2d\tau \leq c
\int_0^t(\|H\|^2_2+\|\tau^n\|_{L_{\infty}(Q_T)}^2|\partial_t\mathbf{u}^n|_2^2)d\tau.
$$

\noindent
Gronwall's lemma produces   $\mathbf{U}^n\rightarrow 0$ in $L^2(Q_{\bar
T})$. Thus the map $ G$ is continuous in  $L^2(Q_{\bar T})$.

The
uniqueness can be proved making  use  of  the  procedure in  theorem 6.1.
 We omit details.  The existence of a
local solution is completely proved.
\end{proof}

\section{Appendix:  Phase  field  model}

 A multi-phase  flow  consists of $n$-fluid constituents which occupy  a, possibly time-dependent, region $\Omega$ whereas  $\Omega_i\subseteq\Omega$ is the region
occupied by constituent $i$; $i= 1,...,n $  denotes   the quantities pertaining to the corresponding
constituent.  The  multi-phase  flow is  quite  naturally  framed within  the  scheme of   fluid     mixtures.
 In the   mixture model   the  principles  of  continuum  mechanics
for  a  single  phase  are  generalized to  several inter-penetrable  continua.  The   basic  assumption
is  that, at  any  instant of  time, all  phases  are  present  at  every  material point. The  equations
 of  balance are  postulated  for   mass  and  momentum  conservation.  Furthermore,  constitutive
 relations  are  required   to  close  the  system  of  equations.

 Let $\rho_i$ and $r_i$ be the mass density and the mass growth. The balance of mass requires that
\begin{align}
&\partial_t \rho_i +\nabla\cdot\rho_i \mathbf{v}_i = r_i, \label{eq:t1} \\ \nonumber
\end{align}
$\mathbf{v}_i$  is the  divergence free  velocity.

\noindent
The mass conservation implies that

$$\sum_{i=1}^n r_i=0.$$
We define the mass density $\rho$    and the mean mass  velocity
$\mathbf{w} $ of the mixture as
 $$\rho :=\sum_{i=1}^n \rho_i, \quad \mathbf{w}:=\frac{1}{\rho}\sum_{i=1}^n \rho _i\mathbf{v}_i.$$

Other average  velocities  can  be   defined,  for   example,   the
volume   average  velocity.

\noindent
$ \mathbf{w}$,  not  solenoidal,  is  used in the  momentum and  energy   balance.
Consequently, summation of  (\ref{eq:t1}) over i  and account of mass  conservation give
$$\partial_t\rho+\nabla\cdot\rho\mathbf{ w} =0.  $$
Let $c_i$ and $\mathbf{u}_i$,
$$c _i:=\frac{\rho_i}{\rho},  \quad \mathbf{u}_i:= \mathbf{v}_i-\mathbf{w},$$
be  the  concentration and  the  relative  (or diffusion) velocity of  constituent  $i$, we  obtain
\begin{align*}
&\partial_tc_i\rho+\nabla\cdot\rho (\mathbf{w} +\mathbf{u}_i)=r_i.\notag\\
\end{align*}
By  the  continuity equation, we  get
\begin{align*}
&\rho(\partial_tc_i+\nabla\cdot c_i \mathbf{w}) =\tau_i-\nabla\cdot \mathbf{J}_i,\notag\\
\end{align*}
where
 $$\mathbf{J}_i= \rho_i\mathbf{u}_i$$
is the diffusion flux of constituent $i$.

\noindent
By definition, it follows that

$$\sum _{i=1}^n\mathbf{J}_i=0.$$
The continuity equation (\ref{eq:t1}) can also be written
\begin{align*}
&\partial_t \rho_i +\nabla\cdot\rho_i\mathbf{w} =r_i - \nabla \cdot \mathbf{J}_i.\notag\\
\end{align*}
The balance of linear momentum can  be  written
\begin{align*}
&\rho_i\partial_t\mathbf{v}_i +\rho_i\mathbf{v}_i\cdot \nabla\mathbf{v}_i   =\nabla \cdot \mathbf{ T}_i  +
\rho_i \mathbf{ f}_i  + \mathbf{g}_i\; \\ \nonumber
&i=1,...,n,\\ \nonumber
\end{align*}
where   $\mathbf{T}_i$ is  the Cauchy stress  tensor, $\mathbf{
f}_i $ the body force, $\mathbf{g}_i $ the supply of linear
momentum from the other constituents.

\noindent
The  whole mixture may  be  viewed as  a single  body. The  balance  of  linear  momentum is  written   as
\begin{align}
&\rho\partial_t \mathbf{w} +\rho \mathbf{ w}\cdot \nabla\mathbf{ w}  =\nabla \cdot\mathbf{ T} +
\rho \mathbf{ f} .\label{eq:t2}\\ \nonumber
\end{align}
The  balance  of  the  angular momentum results  in  the  symmetry
of $ \mathbf{ T}$.

The  continuity  equation is, in absence of reaction,
\begin{align}
&\partial_t \rho +\nabla\cdot\rho \mathbf{ w} =0. \label{eq:t3}\\ \nonumber
\end{align}
In general,  $\mathbf{ w} $  is  not  divergence  free  even $\nabla
\cdot \mathbf{ v}_i   =0$,  $\forall i$.

\noindent
Now we consider  a  mixture of  two miscible  fluids  which before
mixing  are each  incompressible. In  their  unmixed  states let the
density  of  fluids (1) and  (2)  be     $\rho_{10}$  and
$\rho_{20}$ (constant).

\noindent
We  deduce  an incompressible   model following  the  approach  in \cite{rif2}, \cite{rif11}.

In  the  mixture, the  densities  of  the
fluids  at  point  $x$  and  at time  $t$  are  denoted
$\rho_1(x,t):=\rho_1$ ,  $\rho_2(x,t):=\rho_2$, respectively.  Then
from volume additivity of  the two  constituents  at  the  outset,
$$\frac{\rho_1}{\rho_{10}}+ \frac{\rho_2}{\rho_{20}}=1.$$
The  total  density  $\rho(x,t)$   in  mixture  is  defined  by
 $\rho =   \rho_1 +\rho_2$.

\noindent
The  balance  of mass in  the  mixture  gives
\begin{align}
&\partial_t\rho_1 +\nabla \cdot \rho_1\mathbf{ v}_1 =r_1, \label{eq:t4}\\ \nonumber
& \partial_t\rho_2 +\nabla \cdot \rho_2\mathbf{ v}_2 =r_2.\\ \nonumber
\end{align}
The  mean mass velocity $\mathbf{ w} $ is
$$\rho \mathbf{ w} =\rho_1 \mathbf{ v}_1 +\rho_1 \mathbf{ v}_2 .$$
The  continuity  equation for  the  mixture  is
$$\partial_t\rho +\nabla \cdot \rho \mathbf{ w}=0.$$
Now  divide  $(\ref{eq:t4})_1$ by  $\rho_{10}$, and $(\ref{eq:t4})_2$
by $\rho_{20}$, respectively,  and add
 to obtain
\begin{align}
&\partial_t (\frac{\rho_1}{ \rho_{10}}+\frac{\rho_2}{ \rho_{20}} )+\nabla\cdot(\frac{\rho_1}{ \rho_{10}} \mathbf{v}_1+
\frac{\rho_2}{ \rho_{20}}\mathbf{v}_2) =0.\label{eq:t5}\\ \nonumber
\end{align}
Define the  mean volume    velocity of  the   mixture
$$\mathbf{v}=
\frac{\rho_1}{ \rho_{10}}\mathbf{v}_1 + \frac{\rho_2}{ \rho_{20}}\mathbf{v} _2.$$
From (\ref{eq:t5}) we find $\nabla\cdot \mathbf{v} =0.$

Now,  we  derive  the  fundamental relation
$$\mathbf{w}=  \mathbf{v} -\frac{\lambda}{\rho} (\nabla \rho+ D\nabla \Delta \rho)$$
where  $\lambda $ is  the   diffusion coefficient, and $D$ the   mobility  coefficient,
making  use  of the  generalized Fick's  law of  diffusion
$$   \mathbf{v}_1 =\mathbf{w}-\lambda \frac{\nabla \omega(c)}{c},$$

\noindent
where  $\omega(c)$  is  the  chemical  potential
and  $c =  \frac{\rho_1}{\rho}$ is  the mass  concentration.

\noindent
If  $D=0$  we  obtain  the  relation  considered   by
Kazhikhov-Smagulov \cite{rif14}.

This  relation  is  important  because expresses the vector
$\mathbf{w}$  in terms  of  a  divergence  free vector. Inserting
this  relation  in  the  linear momentum  balance  we  obtain a  generalization of the
  system  considered by  Kazhikhov-Smagulov. For  completeness  we  report  a  proof given  in \cite{rif11}.

\noindent
Denote  $\alpha=\frac{\rho_1}{\rho_{10}}$  the  volume  concentration  of  constituent 1, so
$\frac{\rho_2}{\rho_{20}}=  1-\alpha$,  consequently
\begin{align*}
&\mathbf{v}=\alpha \mathbf{v}_1+(1-\alpha)\mathbf{v}_2,\;
1-c=\frac{\rho -\rho_1}{\rho}=\frac{\rho_2}{\rho},\notag\\
\end{align*}
hence
$$\mathbf{w}= c\mathbf{v}_1+(1-c)\mathbf{v}_2.$$
Observing  that
$$c=\frac{\alpha\rho_{10}}{\rho}, \quad \rho=\rho_1+\rho_2= \alpha\rho_{10}+(1-\alpha)\rho_{20},$$
it is
$$\alpha=\frac{\rho -\rho_{20}}{\rho_{10} -\rho_{20}}.$$
By  differentiation     we   get
\begin{align*}
&\nabla c=  \frac{\rho_{10}\nabla\alpha}{\rho}- \frac{\alpha\rho_{10}\nabla\rho}{\rho^2}, \;
\nabla  \rho=(\rho_{10}- \rho_{20})\nabla \alpha, \\
&\nabla \alpha= \frac{\nabla  \rho}{\rho_{10}- \rho_{20}}.\\
\end{align*}
In  addition
$$
c=\frac{\alpha\rho_{10}}{\rho}= \frac{\rho_{10}}{\rho}\big(\frac {\rho-
\rho_{20}}{\rho_{10}- \rho_{20}}\big ),$$
and,  from the  previous relations, it    holds
$$ {\nabla c}= \frac{\rho_{10}}{\rho}\frac{\nabla \rho}  {\rho_{10} -\rho_{20}}-
\frac{\rho_{10}}{\rho ^2}\nabla  \rho\frac{\rho -\rho_{20}}  {\rho_{10} -\rho_{20}}
=\frac{\rho_{10}\rho_{20}\nabla \rho}{
(\rho_{10} -\rho_{20})\rho^2}.$$
Consequently,  upon simplification,
\begin{align}
&\frac{\nabla c}{c}=\frac{\rho_{20}\nabla \rho}{\rho(\rho -\rho_{20})}.
\label{eq:t0}\\ \nonumber
\end{align}
Next, we  eliminate  $\mathbf{v}_2$  from the  relations
$$ \mathbf{v}= \alpha \mathbf{v}_1 + (1-\alpha)\mathbf{v}_2,  \; \mathbf{w}= c\mathbf{v}      _1 + (1-c)\mathbf{v}_2,  $$
to  find
\begin{align}
&\mathbf{v}_1\big(\frac{\alpha}{1-\alpha}-\frac{c}{1-c}\big) =
\frac{1}{1-\alpha}\mathbf{v}-\frac{1}{1-c}\mathbf{w}.
\label{eq:t6}\\ \nonumber
\end{align}
By noting
$$\frac{\alpha}{1-\alpha}=\frac{\rho_1\rho_{20}}{\rho_2\rho_{10}},\;\;
 \frac{c}{1-c}=\frac{\rho_1}{\rho_2},$$
so that
 $$\frac{\alpha}{1-\alpha}-
 \frac{c}{1-c}=\frac{\rho_1}{\rho_2}(\frac{\rho_{20}}{\rho_{10}}-1),$$
 and  then, from (\ref{eq:t6}),
\begin{align}
 & \mathbf{v}_1=  \frac{\rho_2}{\rho_1}\frac {\rho_{10}}{\rho_{20}-
 \rho_{10}}\big[\frac{\rho_{20}}{\rho_2}\mathbf{v}  -\frac{\rho}{\rho_2}\mathbf{w}\big].
\label{eq:t7}\\ \nonumber
\end{align}
We now  eliminate $\mathbf{v}_1$   between (\ref{eq:t7})  and  the generalized  Fick's law,  we  obtain
$$\mathbf{w} \big[\frac{\rho_1(\rho_{20}- \rho_{10})+ \rho\rho_{10} }{\rho_1(\rho_{20}- \rho_{10})}\big]=
\frac{\rho_{10}\rho_{20}}{\rho_1( \rho_{20}-\rho_{10})} \mathbf{v} +
  \frac{\lambda}{c}\nabla \omega(c).$$
   Now  the  coefficient  of  $\mathbf{w}$  can  be  written
\begin{align}
&\frac{\rho_1\rho_{20}+\rho_2\rho_{10}}{\rho_1(\rho_{20}-\rho_{10})}.\label{eq:t8}\\ \nonumber
\end{align}
Since $\rho=\rho_1 +\rho_2$,  $\frac{\rho_1}{\rho_{10}}
=1-\frac{\rho_2}{\rho_{20}}$,  we  have
$\rho_1=\frac{\rho_{10}}{\rho_{20}}(\rho_{20}-\rho_2)$   and hence
\begin{align}
 &\rho_1\rho_{20}+\rho_2\rho_{10}=\rho_{10}\rho_{20}\label{eq:t9}\\ \nonumber
 \end{align}
Using  (\ref{eq:t8}) and   (\ref{eq:t9})
the  coefficient of   $ \mathbf{w}$  simplifies and
$$\mathbf{w}\frac{\rho_{10}\rho_{20}}{\rho_1(\rho_{20}- \rho_{10})} =
  \mathbf{v}\frac{\rho_{10}\rho_{20}}{\rho_1(\rho_{20}-
\rho_{10})}+\frac{\lambda}{c}\nabla \omega(c),$$
and  finally
\begin{align}
&\mathbf{w} =\mathbf{v} +\frac{\lambda}{c}\nabla \omega(c)\frac{\rho_1(\rho_{20}- \rho_{10})}{\rho_{10}\rho_{20}}.\label{eq:t10}\\ \nonumber
\end{align}

Bearing in  mind (\ref{eq:t0} )  and
$\omega(c) =   c  -   D \Delta c,$
we    obtain

$$\nabla\omega(c)= \frac{\rho_{10}\rho_{20}}{\rho_{10} -\rho_{20}}(\frac{\nabla \rho}{\rho^2}-
 D(\frac{\nabla\Delta \rho}{\rho^2} +\nabla\nabla\rho\nabla\frac{1}{\rho^2}+\nabla\cdot(\nabla\rho\nabla\frac{1}{\rho^2}). $$

   Discarding  the  terms  $\rho^{-\gamma}$  with $\gamma \geq 3$, follows
\begin{align}
  &\nabla\omega(c)= \frac{\rho_{10}\rho_{20}}{\rho_{10} -\rho_{20}}(\frac{\nabla \rho}{\rho^2}-
    D\frac{\nabla\Delta \rho}{\rho^2}),  \label{eq:t11}\\ \nonumber
\end{align}
and inserting  (\ref{eq:t11})   in   (\ref{eq:t10})   we   obtain
\begin{align}
&\mathbf{w}=\mathbf{v}  -
\frac{\lambda}{\rho}(\nabla  \rho- D\nabla\Delta  \rho).\label{eq:t12}\\ \nonumber
\end{align}

\section{ Derivation   of   the  incompressible   model}

We  look (\ref{eq:t12}) as   a Helmholtz-type decomposition of  the  vector $\mathbf{w}$.
Consider  the  equations  governing   flow   of   a  binary  mixture   obtained  above
\begin{align}
&\rho\partial_t \mathbf{w} +\rho\mathbf{w}\cdot\nabla\mathbf{w} =\nabla \cdot \mathbf{T} + \rho\mathbf{f},\label{eq:t13}\\ \nonumber
&\partial_t \rho +\nabla\cdot\rho \mathbf{w} =0.\\ \nonumber
\end{align}
It is  straightforward to  deduce   from   $(\ref{eq:t13})_2$  using   (\ref{eq:t11}) that   the  continuity
  equation becomes
 \begin{align}
   &\partial_t \rho +\nabla\cdot\rho \mathbf{v}-\lambda\Delta(\rho-D\Delta \rho)=0.
    \label{eq:t14}\\ \nonumber
    \end{align}
Now,  upon substituting  (\ref{eq:t12})  in (\ref{eq:t13})and  setting $\theta =  \rho  -D\Delta \rho$ we  find
\begin{align*}
&\rho\partial_t\mathbf{ v} +\lambda\frac{\nabla\theta}{\rho}\partial_t\rho -\lambda\partial_t\nabla\theta
+\rho \mathbf{ v}\cdot\nabla \mathbf{ v}-
\lambda\nabla \theta\cdot \nabla \mathbf{ v}-\\ \nonumber
& \lambda \rho \mathbf{ v}\cdot\nabla\frac{\nabla \theta }{\rho}+
\nabla \pi -\mu \Delta \mathbf{ v}+\mu  \lambda\Delta \frac{\nabla \theta}{\rho}+
\eta \nabla \nabla\cdot \frac{\nabla \theta}{\rho}   +\\ \nonumber
&\lambda^2\nabla \theta \cdot\nabla \frac{\nabla\theta}{ \rho} =\rho \mathbf{ f}.\\ \nonumber
 \end{align*}

Using   the   continuity  equation to  substitute for  $\partial_t\rho$    in  the  above  equation   we   find

\begin{align*}
&\rho\partial_t \mathbf{ v} +\rho \mathbf{ v}\cdot\nabla \mathbf{ v} +
\lambda \nabla \theta\cdot \nabla \mathbf{v}-\\ \nonumber
&\lambda  \mathbf{ v}\cdot\nabla{\nabla\theta}+
\nabla \pi
 -\mu \Delta \mathbf{ v} +\\ \nonumber
&\lambda^2(\nabla\cdot(\frac{1}{\rho}\nabla \theta \otimes\nabla {\theta})
  =\rho \mathbf{ f}\\ \nonumber
\end{align*}
where  $\pi  = \eta \lambda\nabla \cdot \frac{\nabla \theta}{\rho}-\partial_t \theta- \lambda
\mu  \Delta log\rho -
\mu D\Delta \frac{\Delta \rho}{\rho}$ and
 we have discarded the non linear  term $\Delta (\nabla\rho^{-1}\Delta \rho)$, for  convenience.

 Neglecting all terms of  O($\lambda^2)$ we find   a  model  type
Kazhikhov-Smagulov,

\noindent
and   it   is  thus    a  model  for  small
diffusion.

\begin{align*}
&\rho\partial_t \mathbf{ v} +\rho \mathbf{ v}\cdot\nabla \mathbf{ v}
-\lambda\nabla \theta\cdot\nabla \mathbf{ v}
-\lambda  \mathbf{ v}\cdot\nabla\nabla \theta +\\
&\nabla \pi+
\mu \Delta \mathbf{ v}=
\rho \mathbf{ f}.\\
\end{align*}
Instead, if  $D=0$ we find   the
model  studied, completely, in \cite{rif33}.

\end{document}